\documentclass{amsart}
\usepackage{amsmath,amsthm}
\usepackage{amsfonts,amssymb}

\usepackage[original]{imakeidx}
\usepackage[
colorlinks = true, linkcolor = blue,citecolor = red,
breaklinks=true, 
]{hyperref}

\usepackage{enumerate}

\newtheorem{theorem}{Theorem}
\newtheorem{Prop}[theorem]{Proposition}
\newtheorem{lemma}[theorem]{Lemma}

\numberwithin{theorem}{section}

\newtheorem{Cor}[theorem]{Corollary}

\theoremstyle{definition}
\newtheorem{Def}[theorem]{Definition}

\theoremstyle{remark}
\newtheorem{Opm}[theorem]{Remark}

\usepackage{appendix} 
 
\usepackage{tikz}
\usepackage{tikz-cd}
\usepackage{graphicx}
\usepackage{comment}
\usepackage{float}
\usepackage{caption}
\usepackage{color}
\usepackage{tensor}

\usetikzlibrary{shapes.geometric}
\usetikzlibrary{matrix, calc, arrows}

\usepackage{ dsfont }
\usepackage{ mathrsfs }
\usepackage{ytableau}
\usepackage{mathtools}

\def\R{\mathds{R} }
\def\E{\mathds{E} }
\def\Z{\mathds{Z} }
\def\N{\mathds{N} }
\def\C{\mathds{C} }

\def\K{\mathds{K} }

\DeclareMathOperator{\sdim}{sdim}
\DeclareMathOperator{\Sym}{Sym}
\DeclareMathOperator{\TKK}{TKK}

\DeclareMathOperator{\ad}{ad}
\DeclareMathOperator{\Ad}{Ad}
\DeclareMathOperator{\Inn}{Inn}
\DeclareMathOperator{\End}{End}
\DeclareMathOperator{\bessel}{\mathcal B_\lambda}

\DeclareMathOperator{\SB}{SB}

\DeclareMathOperator{\Lie}{Lie}
\newcommand{\bbo}[1]{\left<{#1}\right>_{\varpi}}
\DeclareMathOperator{\Ann}{Ann}

\newcommand{\bfip}[1]{\left<{#1}\right>_\mathcal B}

\newcommand{\fip}[1]{\left<{#1}\right>_\mathcal F}
\newcommand{\fipbar}[1]{\overline{\left<{#1}\right>}_\mathcal F}
\newcommand{\fp}[1]{\left<{#1}\right>_F}
\newcommand{\fpbar}[1]{\overline{\left<{#1}\right>}_F}

\newcommand{\ip}[1]{\left<{#1}\right>}
\newcommand{\ipW}[1]{\left<{#1}\right>_W}
\newcommand{\ipO}[1]{\left<{#1}\right>_{\mc O}}
\newcommand{\ipJ}[1]{\left({#1}\right)_J}

\newcommand{\bbf}[1]{\left<{#1}\right>_\beta}
\newcommand{\pt}[1]{\partial_{#1}}
\newcommand{\pu}[1]{\partial^{#1}}
\newcommand{\mf}[1]{\mathfrak{#1}}
\newcommand{\ds}[1]{\mathds{#1}}

\newcommand{\mc}[1]{\mathcal{#1}}

\newcommand{\pushright}[1]{\ifmeasuring@#1\else\omit\hfill$\displaystyle#1$\fi\ignorespaces}
\newcommand{\ol}[1]{\overline{#1}}
\newcommand{\ul}[1]{\underline{#1}}
\newcommand{\ot}[1]{\widetilde{#1}}
\newcommand{\wh}[1]{\widehat{#1}}

\newcommand{\id}{{\rm id}}

\newcommand{\minus}{\scalebox{0.9}{{\rm -}}}
\newcommand{\plus}{\scalebox{0.6}{{\rm+}}}

\DeclareMathOperator{\tr}{tr}

\DeclareMathOperator{\pil}{d\pi_\lambda}
\DeclareMathOperator{\rol}{d\rho_\lambda}
\DeclareMathOperator{\dotpi}{d\ot{\pi}}
\DeclareMathOperator{\dotrho}{d\ot{\rho}}
\DeclareMathOperator{\dnu}{d\nu}
\DeclareMathOperator{\hatpi}{d\wh{\pi}}
\DeclareMathOperator{\SpO}{SpO}
\DeclareMathOperator{\SO}{SO}
\DeclareMathOperator{\Sp}{Sp}
\DeclareMathOperator{\Spin}{Spin}
\DeclareMathOperator{\Mp}{Mp}

\numberwithin{equation}{section}

\DeclarePairedDelimiter\abs{\lvert}{\rvert}%
\DeclarePairedDelimiter\norm{\lVert}{\rVert}%
\makeatletter
\let\oldabs\abs
\def\abs{\@ifstar{\oldabs}{\oldabs*}}
\let\oldnorm\norm
\def\norm{\@ifstar{\oldnorm}{\oldnorm*}}
\makeatother

\usepackage[T1]{fontenc}

\makeindex[name=sym, title = List of notation, intoc, options = -s \jobname.ist]

\usepackage{idxlayout}

\newcounter{indexcount}
\newcommand{\threedigit}[1]{\ifnum #1<100 0\fi\ifnum #1<10 0\fi#1}

\newcommand{\symindex}[1]{%
  \ifcsname\detokenize{SYM@@#1}\endcsname
    \index[sym]{\csname\detokenize{SYM@@#1}\endcsname @#1}%
  \else
    \stepcounter{indexcount}%
    \expandafter\xdef\csname SYM@@\detokenize{#1}\endcsname{%
      \expandafter\threedigit\expandafter{\romannumeral-`Q\theindexcount}%
    }%
    \index[sym]{\threedigit{\theindexcount}@\unexpanded{\unexpanded{#1}}}%
  \fi
}

\begin{document}
\title[Minimal representations of the metaplectic supergroup]{Minimal representations of the metaplectic Lie supergroup and the super Segal-Bargmann transform}

\author{Sam Claerebout}
\address{Department of Electronics and Information Systems \\Faculty of Engineering and Architecture\\Ghent University\\Krijgslaan 281, 9000 Gent\\ Belgium.}
\email{Sam.Claerebout@UGent.be}

\date{\today}

\begin{abstract}
We construct a Schr\"odinger model and a Fock model of a minimal representation of the metaplectic Lie supergroup $\Mp(2m|2n,2n)$. Then, we show that the Schr\"odinger model of the minimal representation leads to an already known Schr\"odinger model of the metaplectic representation of $\Mp(2m|2n,2n)$. Therefore, the Fock model of the minimal representation allows us to construct a Fock model of this metaplectic representation. We then construct an intertwining super Segal-Bargmann transform which extends the classical Segal-Bargmann transform.
\end{abstract}

\maketitle

\tableofcontents

\section{Introduction} 

\subsection{The classical setting}

The metaplectic group $\Mp(2m, \ds R)$ is a double cover of the symplectic group $\Sp(2m, \ds R)$. A well-known and well-studied unitary representation of the metaplectic group is the metaplectic representation, also called oscillator representation or Shegal-Shale-Weil representation. It has two interesting realisations, one more analytic in nature and one more algebraic in nature. The more analytic one is called Schr\"odinger model or $L^2$-model $(\pi, L^2(\ds R^m))$. Here, the action $\pi$ is generated by creation and annihilation operators and $L^2(\ds R^m)$ is the space of square-integrable functions on $\ds R^m$. The more algebraic one is called the Fock model $(\rho, \mc F(\ds C^m))$. Here $\rho$ is generated by coordinate multiplication and differentiation operators and $\mc F(\ds C^{m})$ is the classical Fock space, i.e.\ the space of entire functions on $\ds C^m$ which are square integrable with respect to the Gaussian measure $\exp(-\norm{z}^2)$. These two models are connected via the classical Segal-Bargmannn transform $\SB: L^2(\ds R^m)\rightarrow \mc F(\ds C^{m})$ defined by
\begin{align*}
\SB(f(x))(z) := \exp(-\frac{1}{2}\norm{z}^2)\int_{\ds R^{m|2n}}\exp(2(z\cdot x))\exp(-\norm{x}^2)f(x)dx.
\end{align*}
It is a unitary isomorphism which intertwines the Schr\"odinger model with the Fock model, i.e.\ we have the following commuting diagram
\begin{equation}\label{Dia_Int}
\begin{tikzcd}[row sep=2cm,column sep=2cm,inner sep=1ex,every label/.append style={font=\normalsize}]
L^2(\R^m)  \arrow[thick,swap] {d}{\pi}  \arrow[thick]{r}[name=U]{\SB}
&
\mathcal{F}(\C^m)  
\arrow[thick] {d}{\rho}
\\
L^2(\R^m)   \arrow[thick,swap]{r}[name=D]{\SB}   & \mathcal{F}(\C^m) 
\arrow[to path={(U) node[midway,scale=3] {$\circlearrowleft$}  (D)}]{}
\end{tikzcd}
\end{equation}
For more information about the classical Segal-Bargmannn transform and the metaplectic representation, we refer to \cite{Folland}.

The metaplectic representation decomposes into two irreducible unitary parts, which are given by restricting the action to either even or odd functions. Both parts are examples of the minimal representations of $\Mp(2m, \ds R)$. An irreducible unitary representation of a real simple Lie group is called a minimal representation if the annihilator of the representation on the Lie algebra level is a specific primitive ideal called the Joseph ideal. See, e.g. \cite{GanSavin} for the technical details and exact definitions concerning minimal representations.

For minimal representations of a large class of Lie groups, there exists a unified generalisation of the constructions above, developed in \cite{HKM} and \cite{HKMO}. In these papers the minimal representations of $\Mp(2m, \ds R)$ are realised in $L^2$ and Fock spaces on the minimal orbit instead. In particular, the Schr\"odinger model is realised on the Hilbert space $L^2(\mc O_{\ds R})$, where  $\mc O_{\ds R}$ is the set of all positive real symmetric matrices of rank one. In \cite[Section 1.7]{HKMO} and \cite[Section 3.2]{HKMO} the folding map
\begin{align*}
\begin{matrix}
p: &\ds R^{m}\setminus \{0\} &\rightarrow &\mc O_{\ds R},\\
 &x= (x_1, \ldots, x_{m}) &\mapsto &x^t\cdot x
\end{matrix}
\end{align*}
and the complexified folding map
\begin{align*}
\begin{matrix}
p_\ds C: &\ds C^{m}\setminus \{0\} &\rightarrow &\mc O_{\ds C},\\
 &z= (z_1, \ldots, z_{m}) &\mapsto &z^t\cdot z,
\end{matrix}
\end{align*}
respectively, are defined. The folding map induces a unitary isomorphism $\psi_{\ds R}$ between $L^2_{\text{even}}(\ds R^{m})$ and $L^2(\mc O)$, while the complexified folding map induces a unitary isomorphism $\psi_{\ds C}$ between the Fock space restricted to even functions and the Fock space on the complexification of $\mc O_{\ds R}$.

A Segal-Bargmann transform $\wh\SB$ that intertwines the models on the minimal orbit is then also constructed in \cite[Section 3]{HKMO}. It is linked to the classical Segal-Bargmann transform by the following commuting diagram of unitary isomorphisms.

\begin{equation}\label{Dia_Fold}
\begin{tikzcd}[row sep=2cm,column sep=2cm,inner sep=1ex,every label/.append style={font=\normalsize}]
L^2(\mc O_{\ds R})  \arrow[thick,swap] {d}{\wh\SB}  \arrow[thick]{r}[name=U]{\psi_{\ds R}}
&
L^2_{\text{even}}(\ds R^{m})
\arrow[thick] {d}{\SB}
\\
\mc F(\mc O_{\ds C})   \arrow[thick,swap]{r}[name=D]{\psi_{\ds C}}   & \mc F_{\text{even}}(\ds C^{m})
\arrow[to path={(U) node[midway,scale=3] {$\circlearrowleft$}  (D)}]{}
\end{tikzcd}
\end{equation}

\subsection{Minimal representations of Lie supergroups}

In \cite{BCARXIV2} a general approach is outlined to construct minimal representations and the related models for Lie supergroups. Lie supergroups are a generalisation of Lie groups that allow us to mathematically describe supersymmetry. In particular, we have the metaplectic Lie supergroup $\Mp(2m|p,q)$ which is a generalisation of the metaplectic group $\Mp(2m, \ds R)$. The present paper uses the approach outlined in \cite{BCARXIV2} explicitly for $\Mp(2m|p,q)$ to obtain a generalisation of the metaplectic representation. Specifically, for $\Mp(2m|2n,2n)$ we construct
\begin{itemize}
\item minimal representations,
\item Schr\"odinger models,
\item Fock models,
\item and intertwining Segal-Bargmann transforms,
\end{itemize}
such that they generalise the respective constructions of the metaplectic representation. We refer to the general approach given in \cite{BCARXIV2} simply as the \textit{general approach}.

Note that this general approach is an extension of the non-super cases given in \cite{HKM} and \cite{HKMO}. Therefore, we first obtain minimal representations of $\Mp(2m|2n,2n)$ realised in $L^2$ and Fock spaces on a generalisation of the minimal orbit. Then, to obtain the realisations that generalise the classical Schr\"odinger and Fock model, we apply a generalisation of the folding isomorphisms $\psi_{\ds R}$ and $\psi_{\ds C}$, respectively.

The restriction to $p=q=2n$ in $\Mp(2m|p,q)$ is required since our approach depends on a specific 3-grading of the underlying Lie algebra of $\Mp(2m|p,q)$. The underlying Lie superalgebra of $\Mp(2m|p,q)$ is the orthosymplectic Lie superalgebra $\mf{spo}(2m|p,q)\cong \mf{osp}(p,q|2m)$ and the 3-grading we use corresponds to the TKK-construction of the Jordan orthosymplectic superalgebra $JOSP(m|2n)$. We give this construction explicitly in Sections \ref{SSTKK} and \ref{SSTKK2} and refer to \cite{BC2} for more information about Tits-Kantor-Koecher (TKK) superalgebras.

Our approach also depends on a character of the structure algebra of $JOSP(m|2n)$. When following the general approach with our choice of 3-grading, we find that two distinct characters correspond with minimal representations. One of these minimal representations leads to a generalisation of the metaplectic representation and is the main focus of this paper. The other one is instead a generalisation of the minimal representation of the split orthogonal group ${\rm O}(2n,2n)$ and will only be discussed briefly in Section \ref{S_lambda_1}.

\subsection{Relation to other works}

In \cite{Nishiyama} and \cite{dGM} a Schr\"odinger model generalisation for $\Mp(2m|p,q)$ was already constructed using a representation of the Heisenberg supergroup. In Section \ref{S_Meta} we show that our generalised metaplectic representation is `superunitarily equivalent' to the one constructed in \cite{dGM} for $p=q=2n$.

For the orthosymplectic Lie supergroup ${\rm OSp}(p,q|2m)\cong {\rm SpO}(2m|p,q)$ a minimal representation has already been studied explicitly in \cite{BF} and \cite{BCD}. We have that $\Mp(2m|p,q)$ is a $\Z/2\Z\times \Z/2\Z$-cover of the identity component of ${\rm SpO}(2m|p,q)$. In particular, they have the same underlying Lie algebra.  However, the minimal representation of ${\rm OSp}(p,q|2m)$ depends on a different 3-grading of $\mf{osp}(p,q|2m)$. Moreover, it is not a generalisation of the metaplectic representation, but a generalisation of the minimal representation of the indefinite orthogonal group ${\rm O}(p,q)$. We will refer to the results and constructions in \cite{BF} and \cite{BCD} as the \textit{orthosymplectic case}. We will compare our results with the ones in the orthosymplectic case throughout the present paper.

In \cite{C2} tensor product representations of $\mf{osp}_{\ds C}(d|2m)$ are studied. One of these representations also generalised the minimal representation of $\mf{sp}_{\ds C}(2m)$ and in \cite{CSS2} it is proven that the annihilator ideal of this representation is a Joseph-like ideal. In Section \ref{SS_Joseph} we show that the annihilator ideal of our (complexified) minimal representation is also this Joseph-like ideal. This indicates that our representation of $\Mp(2m|2n,2n)$ is a natural generalisation of the minimal representation of $\Mp(2m, \ds R)$.

In \cite{BC3} the minimal representation of the exceptional Lie superalgebra $D(2,1;\alpha)$, with $\alpha\in \ds C$ is constructed. In \cite{C3} it is then integrated to the group level. $D(2,1;\alpha)$ is a deformation of $\mf{spo}(2|2,2)$ and if $\alpha\in\{1, -2, -\frac{1}{2}\}$ then they are isomorphic. In particular, if $\alpha= -\frac{1}{2}$ and $m=n=1$ then the 3-grading used for $D(2,1;\alpha)$ in \cite{BC3} corresponds to the 3-grading used for $\mf{spo}(2m|2n,2n)$ in the present paper. Therefore, our results for $m=n=1$ either coincide with or extend the results in \cite{BC3} for $\alpha=-\frac{1}{2}$. Since the Lie supergroup $\ds D(2,1, \alpha)$ considered in \cite{C3} is not isomorphic to $\Mp(2|2,2)$, our results no longer coincide on the group level.

\subsection{Superunitarity}

In this paper we will only use the notions of superunitarity and Hilbert superspaces as defined in \cite{dGM}, rather than the more standard notions in, e.g. \cite{CCTV}. The main reason for this choice is that many Lie supergroups do not admit non-trivial superunitary representations with respect to the standard definitions. In particular, minimal representations would no longer be unitary representations. In \cite{dGM} they prove, using their less restrictive definitions, that the metaplectic representation is a superunitary representation of $\Mp(2m|p,q)$. For this paper, it will therefore be sufficient to use these definitions. Note however that even for these less restrictive definitions the minimal representation in the orthosymplectic case does, in general, not seem to be superunitary.

\subsection{Main results}

Throughout the paper, we construct the following.
\begin{itemize}
\item Two minimal representations of $\mf{spo}(2m|2n,2n)$, which generalise the minimal representations of $\mf{sp}(2m)$ and $\mf{o}(2n,2n)$.
\item A Schr\"odinger model $(L^2(\mc O_{\ds R}), \pi_0, \pil)$ and a Fock model $(\mc F(\mc O_{\ds C}), \rho_0, \rol)$ of the minimal representation of $\Mp(2m|2n,2n)$.
\item A Schr\"odinger model $(L^2(\ds R^{m|2n}), \ot\pi_0, \dotpi)$ and a Fock model $(\mc F(\ds C^{m|2n}), \ot\rho_0, \dotrho)$ of the metaplectic representation of $\Mp(2m|2n,2n)$.
\item Folding isomorphisms $\psi_{\ds R}$ and $\psi_{\ds C}$ which connect the Schr\"odinger and Fock models, respectively.
\item Segal-Bargmann transforms $\wh\SB$ and $\SB$ for the minimal and metaplectic representations of $\Mp(2m|2n,2n)$, respectively.
\end{itemize}
Aside from the Schr\"odinger model of the metaplectic representation, all these constructions are new when $m$ and $n$ are nonzero. Moreover, while the Schr\"odinger model itself is not new, the approach we use to construct it is different from the approach used in \cite{dGM}. For $m=0$, we retrieve the classical constructions with respect to $\Mp(2m, \ds R)$. 

The main theorems are the following:
\begin{itemize}
\item Theorem \ref{Th Intertwining Prop} shows that our Segal-Bargmann transfroms intertwine the respective Schr\"odinger and Fock models. 
\item Theorem \ref{Th_superunitary_isom} shows that our Segal-Bargmann transform and Folding isomorphisms are superunitary.
\item Theorem \ref{Th_SUR} shows that our Schr\"odinger and Fock models are superunitary representations.
\end{itemize} 
They can be summarised as generalising diagrams \eqref{Dia_Int} and \eqref{Dia_Fold}, i.e.\ we have
\begin{equation*}
\begin{tikzcd}[row sep=2cm,column sep=2cm,inner sep=1ex,every label/.append style={font=\normalsize}]
L^2(\ds R^{m|2n})  \arrow[thick,swap] {d}{\ot \pi_0}  \arrow[thick]{r}[name=U]{\SB}
&
\mathcal{F}(\ds C^{m|2n})  
\arrow[thick] {d}{\ot \rho_0}
\\
L^2(\ds R^{m|2n})   \arrow[thick,swap]{r}[name=D]{\SB}   & \mathcal{F}(\ds C^{m|2n}) 
\arrow[to path={(U) node[midway,scale=3] {$\circlearrowleft$}  (D)}]{}
\end{tikzcd}\, , \, \begin{tikzcd}[row sep=2cm,column sep=2cm,inner sep=1ex,every label/.append style={font=\normalsize}]
L^2(\mc O_{\ds R})  \arrow[thick,swap] {d}{\wh\SB}  \arrow[thick]{r}[name=U]{\psi_{\ds R}}
&
L^2_{\text{even}}(\ds R^{m|2n})
\arrow[thick] {d}{\SB}
\\
\mc F(\mc O_{\ds C})   \arrow[thick,swap]{r}[name=D]{\psi_{\ds C}}   & \mc F_{\text{even}}(\ds C^{m|2n})
\arrow[to path={(U) node[midway,scale=3] {$\circlearrowleft$}  (D)}]{}
\end{tikzcd}
\end{equation*}

\subsection{Structure of the paper}

The paper is organised as follows. In section \ref{S_Prelim}, we introduce the Jordan superalgebras, Lie superalgebras, Lie supergroups and their associated structures used in this paper. In particular, we construct $\mf{spo}(2m|2n,2n)$ explictly as the TKK-algebra of $JOSP(m|2n)$ in Section \ref{SSTKK2}. Then, in Section \ref{SS_Special_Cases} we highlight three special cases of this construction with respect to the parameters $m$ and $n$ of $\Mp(2m|2n,2n)$.

In Section \ref{SecConstr} we construct the Schr\"odinger and Fock models of the minimal representations of $\mf{spo}(2m|2n,2n)$ by following the general approach. This construction builds upon a (super)polynomial realisation defined in \cite{BC1}, which uses the Bessel operators. Therefore we start the section by introducing the Bessel operators and defining the space of superpolynomials.

In Section \ref{S_lambda12} we restrict ourselves to the minimal representation that generalises the minimal representation of $\mf{sp}(2m)$. Then, we construct the folding isomorphisms and use them to obtain Schr\"odinger and Fock models which more directly generalise the classical ones to the super setting. 

In Section \ref{S_Prop_Meta} we prove that our metaplectic representation has some of the properties we would expect a generalisation of the metaplectic representation to have. In particular:
\begin{itemize}
\item Theorem \ref{ThDecF} gives a decomposition of our Fock model in terms of (super) spherical harmonics and shows that our minimal representation is always indecomposable, but only irreducible if $M=m-2n \not\in -2\ds N$.
\item Theorem \ref{Th_Joseph} shows that the annihilator ideal of our (complexified) minimal and metaplectic representations is a Joseph-like ideal.
\item Theorem \ref{Th Gelfand-Kirillov dimension} shows that the Gelfand-Kirillov dimension of our minimal representation is the same as in the classical setting, i.e.\ it is equal to $m$.
\end{itemize}

In Section \ref{S_Herm_superspace} we introduce the notions of super-inner products and Hermitian superspaces from \cite{dGM} and show that our Schr\"odinger and Fock models give rise to well-defined Hermitian superspaces. We also prove that our representations act infinitesimally superunitary on these Hermitatian superspaces, i.e.\ the actions are skew-supersymmetric with respect to the super-inner products. For the Fock spaces, we also construct reproducing kernels.

In Section \ref{S_SB}, we construct the Segal-Bargmann transforms and prove that they intertwine the respective Schr\"odinger and Fock models. The classical Hermite polynomials can be obtained as the preimages of the monomials in the Fock model under the Segal-Bargmann transform. In Section \ref{SS_Hermite} we use this fact to define Hermite superpolynomials. We also extend the Hermitian superspaces from the previous section to Hilbert superspaces and prove that the Segal-Bargmann transforms and folding isomorphisms are superunitary isomorphisms.

In Section \ref{S_Meta}, we show that our representations of $\mf{spo}(2m|2n,2n)$ integrate to superunitary representations of the Metaplectic Lie supergroup $\Mp(2m|2n,2n)$. The way we show this is by proving that our Schr\"odinger model of the metaplectic representation is, up to a Fourier transform, equal to the one in \cite{dGM}. Therefore, a Fourier transform is introduced in Section \ref{SS_Fourier} and a brief introduction to the metaplectic representation from \cite{dGM} is given in Section \ref{SS_meta_rep}.

In Section \ref{S_lambda_1}, we briefly discuss the other minimal representation we found, the one that generalises the minimal representation of $O(2n,2n)$.

In Appendix \ref{AppS_Long} we give the proofs of two results concerning the Bessel operators, Proposition \ref{Prop_Explicit_Bessel} and Lemma \ref{Lemma_Q_Bessel}. These proofs are straightforward calculations which were too long to include in the paper itself.

At the end of the paper, there is a list of the various notations we used, together with the page numbers corresponding to where the notations are defined.

\subsection{Notations and conventions}

We use $\ds K$\symindex{$\ds K$, $\ds R$, $\ds C$ and $\ds N$} to denote either the field of real numbers $\ds R$ or the field of complex numbers $\ds C$ when results and constructions hold for both choices. Jordan and Lie superalgebras will be defined over $\ds K$ unless otherwise stated. If we wish to specify the field, we will denote the field by a subindex $\ds K$. Function spaces will always be defined over $\ds C$. We use the convention $\ds N = \{0,1,2,\ldots\}$ and denote the complex unit by $\imath$.\symindex{$\imath$} A sesquilinear map means a left-linear and right-antilinear map.

A supervector space is a $\Z/2\Z$-graded vector space $V=V_{\ol 0}\oplus V_{\ol 1}$. An element $v\in V$ is said to have homogeneous parity if $v\in V_i$, $i\in \Z/2\Z$ and we call $i$ its parity and denote it by $|v|$\symindex{$\vert\cdot\vert$}. When we use $|v|$ in a formula, we consider elements of homogeneous parity, with the implicit convention that the formula must be extended linearly for arbitrary elements. If $\dim(V_i) = d_i$, then we write $\dim(V) = (d_{\ol 0}|d_{\ol 1})$\symindex{$\dim$}. We denote the supervector space $V$ with $V_{\ol 0} = \ds K^m$ and $V_{\ol 1} = \ds K^n$ as $\ds K^{m|n}$\symindex{$\ds K^{m\vert n}$}. A superalgebra is a supervector space $A=A_{\ol 0}\oplus A_{\ol 1}$ for which $A$ is an algebra and $A_iA_j\subseteq A_{i+j}$.

Let $V$ be a supervector space with $\dim(V) = (m|n)$. We will always assume a homogeneous basis $(e_i)_{i=1}^{m+n}$ is ordered such that the even parity elements have lower indices than the odd parity elements. We define the parity $|i|$ of an index $i$ by the parity of the associated basis element $e_i$, i.e.\ $|i| := |e_i|$.

\section{Preliminaries}\label{S_Prelim}

Let us start by giving the formal definitions of our superalgebras and Lie supergroups.

\begin{Def}
A \textbf{Lie superalgebra} is a superalgebra $\mf{g}=\mf{g}_{\ol 0}\oplus\mf{g}_{\ol 1}$ with a bilinear multiplication $[\cdot \,, \cdot]$ satisfying the following axioms:
\begin{itemize}
\item Skew-supersymmetry: $[a,b]=-(-1)^{|a||b|}[b,a]$,
\item Super Jacobi identity: $[a,[b,c]]=[[a,b],c]+(-1)^{|a||b|}[b,[a,c]]$.
\end{itemize}
The multiplication on $\mf g$ is called its \textbf{Lie bracket}.
\end{Def}

\begin{Def}
A \textbf{Jordan superalgebra} is a supercommutative superalgebra $J = J_{\ol 0}\oplus J_{\ol 1}$ satisfying the Jordan identity
\begin{align*}
(-1)^{|x||z|}[L_{x}, L_{yz}]+(-1)^{|y||x|}[L_{y}, L_{zx}]+(-1)^{|z||y|}[L_{z}, L_{xy}]=0 \text{ for all } x,y,z \in J.
\end{align*}
Here the operator $L_x$ is left multiplication with $x$ and $[\cdot\,,\cdot]$ is the supercommutator, i.e.\ $[L_x,L_y] := L_xL_y - (-1)^{|x||y|}L_yL_x$. The algebra product on $J$ is called its \textbf{Jordan product}. A Jordan superalgebra is called \textbf{unital} if there exists an $e\in J_{\ol 0}$ such that $a\cdot e = e\cdot a = a$, for all $a\in J$.
\end{Def}

\begin{Def}
A \textbf{Lie supergroup} is a pair $G = (G_0, \mf g)$ together with an action $\Ad : G_0\times \mf g\rightarrow \mf g$ where $G_0$ is a Lie group and $\mf g$ is a Lie superalgebra for which
\begin{itemize}
\item $\mf g_{\ol 0}$ is the Lie algebra of $G_0$, i.e.\ $\Lie(G_0)\cong \mf g_{\ol 0}$.
\item The action $\Ad$ extends the adjoint representation of $G_0$ on $\mf g_{\ol 0}$
\item For all $X\in \mf g_{\ol 0}$ and $Y\in \mf g$ we have
\begin{align*}
d\Ad(X)Y= \left.\dfrac{d}{dt}\Ad(\exp(tX))Y\right|_{t=0}=[X,Y].
\end{align*}
\end{itemize}
\end{Def}

Note that in the existing literature, the definition above is sometimes used for ``super Harisch-Chandra pairs'' instead. The name ``Lie supergroups'' is then reserved for another structure based on supermanifolds. However, these definitions result in categorically equivalent structures. Therefore, we will only use the definition above for Lie supergroups and refer to \cite[Chapter 7]{CCF} for more information about these structures and how they are connected.

\subsection{The general linear superalgebras}\label{SSgl}

The endomorphisms of the supervector space $\ds K^{m|n}$ form an associative superalgebra. We can express it using matrices as\symindex{$\End$} 
\begin{align*}
\End(\ds K^{m|n}):= \left\lbrace \left(\begin{array}{c|c}
 a & b\\ \hline
 c & d\\
\end{array}\right)\mid a\in \ds K^{m\times m}, b\in \ds K^{m\times n}, c\in \ds K^{n\times m}, d\in \ds K^{n\times n} \right\rbrace,
\end{align*}
where the block diagonal matrices $a$ and $d$ give the even part, while the odd part is given by the off-diagonal blocks $b$ and $c$. 

The general linear Lie superalgebra is defined as $\mf{gl}(m|n) := \End(\ds K^{m|n})$\symindex{$\mf{gl}(m\vert n)$} with the Lie bracket given by \[[x,y] := xy - (-1)^{|x||y|}yx,\] for $x,y\in \mf{gl}(m|n)$.

The general linear Jordan superalgebra is defined as $JGL(m|n) := \End(\ds K^{m|n})$\symindex{$JGL(m\vert n)$} with the Jordan product given by \[x\cdot y := \frac{1}{2}(xy+(-1)^{|x||y|}yx),\] for $x,y\in JGL(m|n)$.

Let $E_{ij}$\symindex{$E_{ij}$} be the $(m|n) \times (m|n)$-matrix where the only non-zero entry is a $1$ on the $i$th row and $j$th column. Then $\{E_{ij}: 1\leq i,j \leq m+n\}$ is a basis of $\End(\ds K^{m|n})$. We define the \textbf{supertranspose}\symindex{$\cdot^{ST}$ and $\cdot^\Pi$} of $X = \sum_{i,j=1}^{m+n}X_{ij}E_{ij}\in \End(\ds K^{m|n})$ as \[X^{ST} := \sum_{i,j=1}^{m+n}(-1)^{|j|(|i|+|j|)}X_{ij}E_{ji}\in \End(\ds K^{m|n}).\] We define the \textbf{parity transpose} of $X = \sum_{i,j=1}^{m+n}X_{ij}E_{ij}\in \End(\ds K^{m|n})$ as \[X^\Pi := \sum_{i,j=1}^{m+n}X_{ij}E_{i+n, j+n}\in \End(\ds K^{n|m}),\] where the indices $i,j \in \{1, \ldots, m+n\}$ are taken modulo $m+n$. 

Let $\{e_i\}_{i=1}^{m+n}$ and $\{e_i'\}_{i=1}^{m+n}$ denote a homogeneous basis of $\ds K^{m|n}$ and $\ds K^{n|m}$, respectively and recall that, by our conventions, it is ordered such that the even elements have lower indices than the odd elements. Then we also define a \textbf{parity transpose} of $x =\sum_{i=1}^{m+n}x_i e_i\in\ds K^{m|n}$ by 
\begin{align*}
\Pi(x) := \sum_{i=1}^{m+n} x_i e_i'\in \ds K^{n|m}.
\end{align*}

\subsection{The orthosymplectic superalgebras}\label{SSosp}

Let $\bbf{\cdot\,,\cdot}$\symindex{$\bbf{\cdot\,,\cdot}$} be a supersymmetric, non-degenerate, even bilinear form on $\ds K^{m|2n}$ with basis $\{e_i\}_{i=1}^{m+2n}$. We denote the matrix components by $\beta_{ij} := \bbf{e_i,e_j}$\symindex{$\beta_{ij}$ and $\beta^{ij}$}. Then, we have $\bbf{x,y}=x^t \beta y$ for all $x,y\in \ds K^{m|2n}$ and $\beta = (\beta_{ij})_{i,j=1}^{m+2n}$. Denote the matrix components of the inverse matrix by $\beta^{ij}$, i.e.\ $\beta^{ij}$ is defined such that $\sum_{j}\beta_{ij}\beta^{jk}= \delta_{ik}$.

The \textbf{orthosymplectic Lie superalgebra} $\mf{osp}(m|2n,\beta)$\symindex{$\mf{osp}(m\vert 2n,\beta)$} is defined as the subalgebra of $\mf{gl}(m|2n)$ preserving $\bbf{\cdot\,,\cdot}$, i.e.\
\begin{align*}
\mf{osp}(m|2n,\beta) &:= \left\lbrace X\in \mf{gl}(m|2n)|\bbf{Xu,v}+(-1)^{|u||X|}\bbf{u,Xv} = 0,\forall u,v \in \ds K^{m|2n} \right\rbrace\\
&= \left\lbrace X\in \mf{gl}(m|2n)|X^{ST}\beta+\beta X = 0 \right\rbrace.
\end{align*}
The \textbf{Jordan orthosymplectic superalgebra} $JOSP(m|2n,\beta)$\symindex{$JOSP(m \vert 2n,\beta)$} is defined as the subalgebra of $JGL(m|2n)$ consisting of selfadjoint operators with respect to $\bbf{\cdot\,,\cdot}$, i.e.\
\begin{align*}
&JOSP(m|2n,\beta)\\
&\quad :=\left\lbrace X\in JGL(m|2n)|\bbf{Xu,v}-(-1)^{|u||X|}\bbf{u,Xv} =0,\forall u,v \in \ds K^{m|2n} \right\rbrace\\
&\quad = \left\lbrace X\in JGL(m|2n)|X^{ST}\beta-\beta X = 0 \right\rbrace.
\end{align*}
We define the elements\symindex{$\ell_{ij}$}
\begin{align*}
\ell_{ij} &:= \sum_{k=1}^{m+2n}\beta_{jk}E_{ik}+(-1)^{|i||j|}\beta_{ik}E_{jk},
\end{align*}
for $i,j\in \{1, \ldots, m+2n\}$. A basis of $JOSP(m|2n,\beta)$ is then given by $\ell_{ij}$ for $i<j$ and $\ell_{ii}$ for $|i|=0$. The Jordan multiplication in terms of the basis elements is given by
\begin{align*}
2 \ell_{ij}\cdot \ell_{kl} = \beta_{jk}\ell_{il} + (-1)^{|i||j|}\beta_{ik}\ell_{jl} + (-1)^{|k||l|}\beta_{jl}\ell_{ik} + (-1)^{|i||j|+|k||l|}\beta_{il}\ell_{jk}
\end{align*}
and the unit in terms of the basis elements is given by\symindex{$e$}
\begin{align*}
e := \dfrac{1}{2}\sum_{i,j=1}^{m+2n}\ell_{ij}\beta^{ij} = \dfrac{1}{2}\sum_{i=1}^{m}\ell_{ii}\beta^{ii}+\sum_{1\leq i < j \leq m+2n}\ell_{ij}\beta^{ij}.
\end{align*}

A supersymmetric, non-degenerate, even bilinear form bilinear form $\bbf{\cdot\,,\cdot}$ on $\ds K^{m|2n}$ induces a skew-supersymmetric non-degenerate, even bilinear $\left<\cdot\,,\cdot\right>_{\varpi}$ on $\ds K^{2n|m}$ defined by $\left<\Pi(u) ,\Pi(v)\right>_{\varpi}:=\bbf{u,v}$, for $u,v \in \ds K^{m|2n}$, where $\Pi$ is the parity transposition defined in Section \ref{SSgl}. We define\symindex{$\mf{spo}(2n\vert m,\varpi)$}
\begin{align*}
&\mf{spo}(2n|m,\varpi)\\
&\quad := \left\lbrace X\in \mf{gl}(2n|m)|\left<Xu,v\right>_\varpi+(-1)^{|u||X|}\left<u,Xv\right>_{\varpi} = 0,\forall u,v \in \ds K^{m|2n} \right\rbrace,
\end{align*}
as the subalgebra of $\mf{gl}(2n|m)$ preserving $\left<\cdot\,,\cdot\right>_{\varpi}$. In most literature it is, somewhat confusingly, also called the orthosymplectic Lie superalgebra. This is justified by the fact that $\mf{osp}(m|2n,\beta)\cong \mf{spo}(2n|m,\varpi)$ by parity transposition. A fitting alternative name for $\mf{spo}(2n|m,\varpi)$ might be ``symplecthogonal Lie superalgebra''.

If the $(\beta_{ij})_{i,j=1}^m$ part of $\beta$ has signature $(p,q)$, with $p+q=m$, we also denote the Lie superalgebras $\mf{osp}(m|2n,\beta)$, $JOSP(m|2n,\beta)$ and $\mf{spo}(2n|m,\varpi)$\symindex{$\mf{osp}(p,q \vert 2n)$}\symindex{$JOSP(p,q \vert 2n)$}\symindex{$\mf{spo}(2n \vert p,q)$} by $\mf{osp}(p,q|2n)$, $JOSP(p,q|2n)$ and $\mf{spo}(2n|p,q)$, respectively.

\subsection{The Heisenberg Lie superalgebra}\label{SS_Heisenberg}

Let $\bbo{\cdot\,,\cdot}$ be a skew-supersymmetric, non-degenerate, even bilinear form on $\ds K^{2m|n}$ with basis $\{e_i\}_{i=1}^{2m+n}$. We denote the matrix components by $\varpi_{ij} := \bbo{e_i,e_j}$.

\begin{Def}
The \textbf{Heisenberg Lie superalgebra} $\mf{h}(2m|n,\varpi)$\symindex{$\mf{h}(2m\vert n,\varpi)$} is defined as the superalgebra $\ds K^{2m|n}\oplus \ds K Z$, where $Z$ is an even generator and the Lie bracket is given by
\begin{align*}
[p_1+ a_1 Z, p_2 + a_2 Z] = \bbo{p_1,p_2}Z,
\end{align*}
for all $p_1,p_2\in \ds K^{2m|n}$ and $a_1,a_2\in\ds K$.
\end{Def}

Note that $\ds K Z$ is the center subalgebra of $\mf{h}(2m|n,\varpi)$ and
\begin{align*}
[e_i,e_j] = \varpi_{ij}Z,
\end{align*}
for all $i,j\in\{1, \ldots, 2m+n\}$.

If the $(\varpi_{ij})_{i,j=2m+1}^{2m+n}$ part of $\varpi$ has signature $(p,q)$, with $p+q=n$, we also denote $\mf{h}(2m|p+q,\varpi)$ by $\mf{h}(2m|p,q)$.\symindex{$\mf{h}(2m\vert p,q)$}

\subsection{The unitary Lie superalgebra}\label{SS_unitary}

The following Lie superalgebra is only defined over $\ds R$. Let $\left<\cdot\,,\cdot\right>_\sigma$ be a non-degenerate sesquilinear form on $\ds C^{m|n}$ with basis $\{e_i\}_{i=1}^{m+n}$. We denote the matrix components by $\sigma_{ij} := \left<e_i,e_j\right>_{\sigma}$. Let $\ol x$ denote the complex conjugate of $x$.

\begin{Def}\symindex{$\mf u(m\vert n, \sigma)$}
The \textbf{unitary Lie superalgebra} $\mf u(m|n, \sigma)$ is defined as the subalgebra of $\mf{gl}_{\ds C}(m|n)$ preserving $\left<\cdot\,,\cdot\right>_\sigma$, i.e.\
\begin{align*}
\mf{u}(m|n,\sigma) &:= \left\lbrace X\in \mf{gl}_{\ds C}(m|n)|\left<Xu,v\right>_\sigma+(-1)^{|u||X|}\left<u,Xv\right>_\sigma = 0,\forall u,v \in \ds C^{m|n} \right\rbrace\\
&= \left\lbrace X\in \mf{gl}_{\ds C}(m|n)|\ol X^{ST}\sigma+\sigma X = 0 \right\rbrace.
\end{align*}
\end{Def}

\subsection{The metaplectic Lie supergroup}\label{SS_metaplectic}

We define the Lie supergroup $\SpO(2m|p,q)$\symindex{$\SpO(2m\vert p,q)$} as the pair $(\Sp(2m, \ds K)\times O(p,q), \mf{spo}(2m|p,q))$, where the adjoint representation is given by
\begin{align*}
\Ad(M)X := MXM^{-1},
\end{align*}
for all $M\in \Sp(2m, \ds K)\times O(p,q)$ and $X\in \mf{spo}(2m|p,q)$. We denote its connected component at the identity by $\SpO^\circ(2m|p,q) = (\Sp(2m, \ds K)\times \SO^\circ(p,q), \mf{spo}(2m|p,q))$.\symindex{$\SpO^\circ(2m\vert p,q)$} Let $\Mp(2m, \ds K)$ be the metaplectic group and $\Spin^\circ(p,q)$\symindex{$\Spin^\circ(p,q)$} the connected component of the spin group at the identity. Then, we have a canonical $\Z/2\Z\times \Z/2\Z$-covering
\begin{align*}
P : \Mp(2m, \ds K)\times \Spin^\circ(p,q)\rightarrow \Sp(2m, \ds K) \times \SO^\circ(p,q).
\end{align*}

\begin{Def}\symindex{$\Mp(2m\vert p,q)$}
The \textbf{metaplectic Lie supergroup} $\Mp(2m|p,q)$ is a $\Z/2\Z\times \Z/2\Z$-covering of $\SpO^\circ(2m|p,q)$, i.e.\
\[\Mp(2m|p,q) = (\Mp(2m, \ds K)\times \Spin^\circ(p,q), \mf{spo}(2m|p,q)),\]
where the adjoint representation is given by
\begin{align*}
\ot{\Ad}(M)X := \Ad(P(M))X,
\end{align*}
for all $M\in \Mp(2m, \ds K)\times \Spin^\circ(p,q)$ and $X\in \mf{spo}(2m|p,q)$.
\end{Def}

\subsection{The TKK-constuction}\label{SSTKK}
With each Jordan (super)algebra one can associate a $3$-graded Lie (super)algebra via the TKK-construction. There exist different TKK-constructions in the literature, see \cite{BC2} for an overview, but for the Jordan or\-tho\-sym\-plec\-tic superalgebra $JOSP(m|2n,\beta)$, with $(m,n)\neq (0,1)$, all constructions lead to $\mf{spo}(2m|2n,2n)$. Note that this is independent of the bilinear form $\beta$. We will quickly review the Koecher construction here.
Let $J$ be a unital Jordan superalgebra. The space of inner derivations of $J$ is defined as\symindex{$\Inn$}
\begin{align*}
\Inn(J) := \mbox{span}_{\ds K}\left\lbrace [L_x, L_y]| x,y\in J\right\rbrace.
\end{align*}
Here the operator $L_x$\symindex{$L_x$} is left multiplication by $x$ and $[\cdot\,,\cdot]$ is the supercommutator bracket, i.e.\ $[L_x,L_y] := L_xL_y - (-1)^{|x||y|}L_yL_x$.

The inner structure algebra of $J$ is defined as\symindex{$\mf{istr}$}
\begin{align*}
\mf{istr}(J) = \{L_x | x\in J\} \oplus \Inn(J) = \mbox{span}_{\ds K}\left\lbrace L_x, [L_x, L_y]| x,y\in J\right\rbrace.
\end{align*}

Let $J^+$ and $J^-$ be two copies of $J$. As vector spaces, we define the TKK-algebra of $J$ as\symindex{$\TKK$}
\begin{align*}
\TKK(J) &:= J^-\oplus \mf{istr}(J) \oplus J^+.
\end{align*}
The Lie bracket is defined as follows. We embed $\mf{istr}(J)$ as a subalgebra of $\TKK(J)$ and for homogeneous $x,y\in J^+$, $u,v\in J^-$, $a,b \in J$ we set
\begin{align*}
[x,u] &= 2L_{xu}+2[L_x,L_u], &&& [x,y] &= [u,v] = 0,\\
[L_a,x] &= ax, &&& [L_a,u] &= -au,\\
[[L_a,L_b], x] &= [L_a,L_b]x, &&& [[L_a,L_b], u] &= [L_a,L_b]u.
\end{align*}

%

\subsection{The TKK-algebra $\mf {spo}(2m|2n,2n)$}\label{SSTKK2}

From now on we will always assume $(m,n)\neq (0,1)$. Since the results are independent of the chosen bilinear forms, we will only work with the following explicit realisations throughout this paper.
Let\symindex{$\beta$}
\begin{align*}
\beta := \left(\begin{array}{c|cc}
 I_m & 0 & 0\\ \hline
 0 & 0 & -I_n\\
 0 & I_n & 0\\
\end{array}\right),
\end{align*}
be the matrix realisation of a supersymmetric, non-degenerate, even bilinear form on $\ds K^{m|2n}$ and\symindex{$\Omega$}
\begin{align*}
\Omega := \left(\begin{array}{cc|cccc}
 0 & -I_{m} & 0 & 0 & 0 & 0\\
 I_{m} & 0 & 0 & 0 & 0 & 0\\ \hline
 0 & 0 & 0 & 0 & 0 & I_{n}\\
 0 & 0 & 0 & 0 & -I_{n} & 0\\
 0 & 0 & 0 & -I_{n} & 0 & 0\\
 0 & 0 & I_{n} & 0 & 0 & 0\\
\end{array}\right),
\end{align*}
be the matrix realisation of a skew-supersymmetric, non-degenerate, even bilinear form on $\ds K^{2m|4n}$. We define\symindex{$J$ and $\mf g$}
\begin{align*}
J &:= JOSP(m|2n,\beta) \quad \text{ and }\quad \mf g := \mf{spo}(2m|4n,\Omega).
\end{align*}
We will now show that $\mf g$ is isomorphic to the TKK-construction of $J$. Specifically, we will show that the following theorem holds by constructing the isomorphisms explicitly.
\begin{theorem}\label{ThTKK-JP}
For $(m,n)\neq (0,1)$ we have $\mf{istr}(JOSP(m|2n,\beta)) \cong \mf{gl}(m|2n)$ and $\TKK(JOSP(m|2n,\beta)) \cong \mf{spo}(2m|4n,\Omega)$.
\end{theorem}
A basis of $J$ was given in Section \ref{SSosp} and a basis of $\mf g$ is given by\symindex{$U_{ij}$}
\begin{align*}
U_{ij} &:= \sum_{k=1}^{2m+4n}\Omega_{jk}E_{ik}+(-1)^{|i||j|}\Omega_{ik}E_{jk}, &&\text{ for } i<j,\\
U_{ii} &:= 2\sum_{k=1}^{2m+4n}\Omega_{ik}E_{ik}, &&\text{ for } |i|=0.
\end{align*}
The Lie bracket on $\mf g$ in terms of these basis elements is given by
\begin{align*}
[U_{ij},U_{kl}] = \Omega_{jk}U_{il} + (-1)^{|i||j|}\Omega_{ik}U_{jl} + (-1)^{|k||l|}\Omega_{jl}U_{ik} + (-1)^{|i||j|+|k||l|}\Omega_{il}U_{jk}.
\end{align*}

Define\symindex{$\ul i$ and $\ot i$}
\begin{align*}
\ul i := \left\lbrace\begin{matrix}
i+m, & \text{ if } i \leq m,\\
i+m+2n, & \text{ if }i \geq m+1
\end{matrix}\right. \quad \text{ and } \quad
\ot i := \left\lbrace\begin{matrix}
i, & \text{ if }i \leq m,\\
i+m, & \text{ if }i \geq m+1.
\end{matrix}\right.,
\end{align*}
for all $i\in\{1, \ldots, m+2n\}$ and consider the short subalgebra
\begin{align}\label{EqSl2sub}
&\left\lbrace -\dfrac{1}{2}\sum_{i=1}^{m} U_{\ul i, \ul i}- \sum_{i=m+1}^{m+n} U_{\ul i,\ul{i+n}}, \sum_{i=1}^m U_{\ot i,\ul i}+ \sum_{i=m+1}^{m+n} U_{\ot i,\ul{i+n}} - U_{\ot{i+n},\ul i},\right.\\  \nonumber
&\quad \left. \dfrac{1}{2}\sum_{i=1}^m  U_{\ot i, \ot i}+ \sum_{i=m+1}^{m+n}U_{\ot i,\ot{i+n}} \right\rbrace.
\end{align}
This subalgebra is isomorphic to $\mf{sl}(2)$ and the decomposition of $\mf g$ as eigenspaces under $\ad(\sum_{i=1}^m U_{\ot i,\ul i}+ \sum_{i=m+1}^{m+n} U_{\ot i,\ul{i+n}} - U_{\ot{i+n},\ul i})$ gives us the $3$-grading $\mf g = \mf g_- \oplus \mf g_0 \oplus \mf g_+$, with
\begin{align*}
\mf g_- &= \mbox{span}_{\ds K}\left\lbrace U_{\ul i,\ul j}| 1\leq i,j \leq m+2n \right\rbrace,\\
\mf g_+ &= \mbox{span}_{\ds K}\left\lbrace U_{\ot i, \ul j}|1\leq i,j \leq m+2n\right\rbrace,\\
\mf g_0 &= \mbox{span}_{\ds K}\left\lbrace U_{\ot i, \ot j}|1\leq i,j \leq m+2n \right\rbrace.
\end{align*}
We can now construct an isomorphism $\phi$\symindex{$\phi$} between $\mf g = \mf g_- \oplus \mf g_0 \oplus \mf g_+$ and $\TKK(J) = J^-\oplus \mf{istr}(J)\oplus J^+$. It is given by
\begin{align*}
\phi(\ell_{ij}^-) = -U_{\ul i, \ul j}, & & \phi(\ell_{ij}^+) = U_{\ot i, \ot j}, & & \phi(2L_{\ell_{ij}}) = U_{\ot i, \ul j} + (-1)^{|i||j|}U_{\ot j, \ul i},
\end{align*}
and
\begin{align*}
\phi(4[L_{\ell_{ij}}, L_{\ell_{rs}}]) &= \beta_{jr}(U_{\ot i, \ul s}-(-1)^{|s||i|}U_{\ot s, \ul i}) + (-1)^{|i||j|+|r||s|}\beta_{is}(U_{\ot j, \ul r} - (-1)^{|r||j|}U_{\ot r, \ul j})\\
&\quad + (-1)^{|r||s|}\beta_{js}(U_{\ot i, \ul r}- (-1)^{|r||i|}U_{\ot r, \ul i})\\
&\quad +(-1)^{|i||j|}\beta_{ir}(U_{\ot j, \ul s} - (-1)^{|s||j|}U_{\ot s, \ul j}).
\end{align*}
Here $x^{\pm}\in J^\pm$\symindex{$\cdot^{\pm}$} denotes the element $x\in J$ interpreted as in the copy $J^\pm$ of $J$.
Under the isomorphism $\phi$, the $\mf{sl}(2)$-triple \eqref{EqSl2sub} of $\mf g$ becomes
\begin{align*}
\left\lbrace e^-, 2L_{e},e^+ \right\rbrace,
\end{align*}
where $e^\pm$ is the unit element of $J^{\pm}$.

Set $I = \{m+1, \ldots, m+n\}$. An explicit isomorphism between $\mf{gl}(m|2n)$ and $\mf{istr}(JOSP(m|2n,\beta))$  is given by
\begin{align*}
E_{ij} &\mapsto -U_{\ot i, \ul j} && \text{ if } i\in I, \, j \not\in I \text{ or } i\not\in I, \, j\in I,\\
E_{ij} &\mapsto U_{\ot i, \ul j} && \text{ otherwise. }
\end{align*}

Note that from \[\mf{istr}(J) = \{L_x | x\in J\} \oplus \Inn(J) \cong \mf{gl}(m|2n),\] we also find that a basis of $\Inn(J)\cong \mf{osp}(m|2n, \beta)$ is given by
\begin{align*}
&U_{\ot i, \ul j} - (-1)^{|i||j|}U_{\ot j, \ul i}, &&\text{ for } i<j,\\
&2U_{\ot i, \ul i}, &&\text{ for } |i|=1.
\end{align*}

\subsection{Special cases}\label{SS_Special_Cases}
We can distinguish the following three special cases based on the parameters $m$ and $n$.
\begin{itemize}
\item \textbf{The symplectic case} ($n=0$).\\
If $n=0$, we have $J = JOSP(m|0) = \Sym (m)$, the space of $m\times m$ symmetric matrices with entries in $\ds K$ and $\mf g = \mf{spo}(2m|0) = \mf{sp}(2m, \ds K)$, the symplectic Lie algebra. This is the classical setting.
\item \textbf{The split orthogonal case} ($m=0, n\geq 2$).\\
If $m=0$ and $n\geq 2$, we have $J = JOSP(0|2n) \cong {\rm Skew} (2n)$, the space of $2n\times 2n$ skew-symmetric matrices with entries in $\ds K$ and $\mf g = \mf{spo}(0|2n,2n) = \mf{so}(2n,2n)$, the split orthogonal Lie algebra. This case is covered by the (classical) general approach in \cite{HKM}, but it is not discussed explicitly.

Note that the split orthogonal Lie algebra is itself a special case of the indefinite orthogonal Lie algebra $\mf{so}(p,q)$ and the minimal representation of $\mf{so}(p,q)$ has been studied extensively, see e.g.\ \cite{KO1}, \cite{KO2}, \cite{KO3} and \cite{KM}. It is also studied in more detail and covered by the (classical) general approach in \cite{HKM}. There $\mf{so}(p,q)$ is seen as the TKK-construction of the spin factor Jordan algebra. In particular, for $\mf{so}(2n,2n)$ we can obtain its minimal representation by viewing $\mf{so}(2n,2n)$ as the TKK-construction of ${\rm Skew} (2n)$ or the spin factor Jordan algebra. From e.g. \cite{GanSavin} we know that the minimal representation of $\mf{so}(2n,2n)$ is unique, up to equivalence, i.e.\ both methods result in equivalent representations.
\item \textbf{The $D(2,1;\alpha)$ case} ($m=n=1$, $\alpha=-\frac{1}{2}$).\\
If $m=n=1$, then $J$ is isomorphic the Jordan superalgebra $D_{\alpha}$. As mentioned in the introduction, $\mf g$ is then isomorphic to the exceptional Lie superalgebra $D(2,1;\alpha)$, where $\alpha=-\frac{1}{2}$ and our results coincide with or extend the results in \cite{BC3} for $\alpha=-\frac{1}{2}$.
\end{itemize}

\section{The minimal representations of $\mf{spo}(2m|2n,2n)$}\label{SecConstr}

To construct the minimal representations we first introduce a polynomial realisation of $\mf g=\mf{spo}(2m|2n,2n)$. This realisation is given for more general Lie superalgebras in \cite{BC1} and it generalises the conformal representations considered in \cite{HKM} to the super setting. This polynomial realisation depends on a character $\lambda \colon \mf {istr}(J) \to \ds K$\symindex{$\lambda$} and a crucial role is played by the Bessel operators.

\subsection{The Bessel operators $\bessel$}

Define $\wh m :=  \frac{1}{2}m(m+1)+n(2n-1)$\symindex{$\wh m$ and $\wh n$} and ${\wh n}:=mn$, then $\dim(J)=(\wh m|2{\wh n})$.

\begin{Def} Consider a character $\lambda \colon \mathfrak{istr}(J) \to \ds K$. For any $u,v \in J^{\minus}$ we define $\lambda_{u} \in (J^{\plus})^\ast$ and $\widetilde{P}_{u,v} \in J^{\minus}\otimes (J^{\plus})^\ast $ by 
$$\lambda_{u}(x) := -2\lambda(L_{xu})$$
and
$$\widetilde{P}_{u,v}(x) := (-1)^{\abs{x}(\abs{u}+\abs{v})}(L_u L_v + (-1)^{|u||v|}L_vL_u-L_{uv})(x)$$
for all $x\in J^{\plus}$. Then we define the \textbf{Bessel operator}\symindex{$\bessel$} as
\begin{align*}
\bessel = \sum_{i=1}^{\wh m+2{\wh n}} \lambda_{z_{i}}\pt{z_{i}}+ \sum_{i,j=1}^{\wh m+2{\wh n}} \widetilde{P}_{z_{i},z_{j}}\pt{z_{j}}\pt{z_{i}},
\end{align*}
\end{Def}
with $(z_i)_i$ a homogeneous basis of $J^-$. Here the $\pt{z_{i}}$'s denote supercommutative partial derivatives on $J^-$, i.e.\ we have $\pt{z_{i}}(z_j) = \delta_{ij}$ and $\pt{z_{i}}\pt{z_{j}}= (-1)^{|z_i||z_j|}\pt{z_{j}}\pt{z_{i}}$. Note that $\bessel(x)$ is a second order differential operator on $J^-$, for all $x\in J^+$.

We obtain the following result from proposition 4.2 in \cite{BC1}.

\begin{Prop}\label{PropBesCom}
The family of operators $\bessel(x)$ for $x\in J^+$, supercommutes for fixed $\lambda$, i.e.\
\begin{align*}
[\bessel(u), \bessel(v)] = 0,
\end{align*}
for $u,v\in J^+$.
\end{Prop}

Let $\lambda$ be a character of $\mf{gl}(m|2n)$. Since
\begin{align*}
0=\lambda([E_{ij}, E_{kl}])=\delta_{jk}\lambda(E_{il})-(-1)^{(|i|+|j|)(|k|+|l|)}\delta_{il}\lambda(E_{kj}),
\end{align*}
we find that the character is uniquely determined by the value of $\lambda := (-1)^{|1|}\lambda(E_{11})$ with $\lambda(E_{ij})=(-1)^{|i||j|}\delta_{ij}\lambda$. Therefore, the isomorphism between $\mf{istr}(J)$ and $\mf{gl}(m|2n)$ implies that a character $\lambda$ of $\mf{istr}(J)$ is defined by $\lambda(L_{\ell_{ij}})=\beta_{ij}\lambda$, for a parameter $\lambda\in \C$.

\begin{Prop}\label{Prop_Explicit_Bessel}
We have
\begin{align*}
\bessel(\ell_{ij}) &= -2\lambda\sum_{k,l=1}^{m+2n}(1+\delta_{kl})\beta_{jk}\beta_{il}\pt{\ell_{kl}}\\
&\quad +\sum_{k,l,r,s=1}^{m+2n}(-1)^{|k||i|}(1+\delta_{kl}+\delta_{rs}+\delta_{kl}\delta_{rs})\beta_{is}\beta_{jl}\ell_{kr}\pt{\ell_{sr}}\pt{\ell_{lk}},
\end{align*}
for all $1\leq i,j \leq m+2n$.
\end{Prop}

\begin{proof}
This is a long and straightforward calculation. See Appendix \ref{AppS_Long}, Proposition \ref{Prop_App_Expl_Bessel}
\end{proof}

We define $\ell^{ij} := \sum_{k,l=1}^{m+2n}\ell_{kl}\beta^{ki}\beta^{lj}$\symindex{$\ell^{ij}$, $\pu{ij}$ and $\pt{ij}$} for all $i,j\in \{1, \ldots, m+2n\}$ and denote $\pu{ij} := \pt{\ell_{ij}}$. We also define $\pt{ij} := \sum_{k,l=1}^{m+2n}\pu{kl}\beta_{ik}\beta_{jl}$. Then it holds that $\pt{ij}\ell^{kl} = \pu{ij}\ell_{kl} = \delta_{ik}\delta_{jl}+(-1)^{|i||j|}\delta_{il}\delta_{jk}- \delta_{ij}\delta_{kl}\delta_{ik}$. The Bessel operator of $\ell_{ij}$ can be rewritten as
\begin{align*}
\bessel(\ell_{ij}) &= -2\lambda(1+\delta_{ij})\pt{ji}+\sum_{k,l=1}^{m+2n}(-1)^{|k||i|}(1+\delta_{jk}+\delta_{il}+\delta_{jk}\delta_{il})\ell^{kl}\pt{il}\pt{jk}.
\end{align*}

\subsection{The polynomial realisation $\pil$}

\begin{Def}
The \textbf{space of superpolynomials} over $\ds K$ is defined as
\begin{gather*}
\mathcal P\big(\ds K^{m|d}\big):=\mathcal P \big(\ds K^{m}\big)\otimes_\C\Lambda\big(\ds K^{d}\big),
\end{gather*}\symindex{$\mc P(\K^{m\vert 2n})$}
where $\mc P(\ds K^{m})$ denotes the space of complex-valued polynomials over the field $\ds K$ in $m$ variables and $\Lambda(\ds K^{d})$ denotes the Grassmann algebra in $d$ variables. The variables of $\mc P(\ds K^{m})$ and $\Lambda(\ds K^{d})$ are called \textbf{even} and \textbf{odd} variables, respectively.
\end{Def}
Let $z = (z_i)_{i=1}^{m+2n}$ denote the variables of $\mathcal P(\ds K^{m|2n})$, then they satisfy the commutation relations
\begin{gather*}
z_iz_j = (-1)^{|z_i||z_j|}z_j z_i,
\end{gather*}
for $i,j\in\{1, \ldots, m+2n\}$. We also define the space of superpolynomials of homogeneous degree $k$ as
\begin{gather*}
\mc P_k\big(\K^{m|2n}\big) := \big\{p\in\mc P\big(\K^{m|2n}\big)\colon \ds E p = k p\big\},
\end{gather*}\symindex{$\mc P_k(\K^{m\vert 2n})$}
where $\ds E := \sum_{i=1}^{m+2n}z_i\pt{z_i}$\symindex{$\ds E$} is the Euler operator.

From \cite[Section 4.1]{BC1} we obtain the following polynomial realisation $\pil$\symindex{$\pil$} of $\TKK(J) = J^- \oplus  \mf{istr}(J) \oplus J^+ $ on $\mc P(J)\cong \mc P(\ds K^{{\wh m}|2{\wh n}})$. Let $(z_i)_i$ be a homogeneous basis of $J$ and $(z_i^\pm)_i$ the corresponding homogeneous bases of $J^\pm$. For $i,j\in\{1, \ldots, {\wh m}+2{\wh n}\}$ we have
\begin{align*}
&\bullet\, \pil(z_i^-) = -2\imath z_i,\\
&\bullet\, \pil(L_{z_i}) = \lambda(L_{z_i}) - \sum_{k=1}^{{\wh m}+2{\wh n}} L_{z_i}(z_k)\pt{z_k},\\
&\bullet\, \pil([L_{z_i}, L_{z_j}]) = \sum_{k=1}^{{\wh m}+2{\wh n}} [L_{z_i}, L_{z_j}](z_k)\pt{z_k},\\
&\bullet\, \pil(z_i^+) = -\dfrac{\imath}{2}\bessel(z_i),
\end{align*}
where we identify the homogeneous bases $(z_i^\pm)_i$ with the variables $(z_i)_i$ of $\mc P(\ds K^{{\wh m}|2{\wh n}})$, canonically.

In terms of the TKK-basis, the polynomial realisation becomes
\begin{itemize}
\item
\begin{align*}
\pil(\ell_{ij}^-) &= -2\imath \ell_{ij},
\end{align*}
\item \begin{align*}
\pil(2L_{\ell_{ij}}) &= 2\lambda(L_{\ell_{ij}}) - 2\sum_{1\leq k\leq l \leq m+2n} L_{\ell_{ij}}(\ell_{kl})\pt{\ell_{kl}}\\
&=2\lambda\beta_{ij} - \sum_{k,l=1}^{m+2n}(1+\delta_{kl})(\beta_{jk}\ell_{il}+(-1)^{|i||j|}\beta_{ik}\ell_{jl})\pt{\ell_{kl}},
\end{align*}
\item \begin{align*}
\pil(4[L_{\ell_{ij}}, L_{\ell_{rs}}]) &= 4\sum_{1\leq k\leq l \leq m+2n} [L_{\ell_{ij}}, L_{\ell_{rs}}](\ell_{kl})\pt{\ell_{kl}}\\
& = \sum_{k,l=1}^{m+2n}(1+\delta_{kl})((\beta_{sk}\beta_{jr}+(-1)^{|r||s|}\beta_{rk}\beta_{js})\ell_{il}\\
&\quad +(-1)^{|i||j|}(\beta_{sk}\beta_{ir}+(-1)^{|r||s|}\beta_{rk}\beta_{is})\ell_{jl}\\
&\quad -(-1)^{|k||s|}(\beta_{ik}\beta_{jr}+(-1)^{|i||j|}\beta_{jk}\beta_{ir})\ell_{sl}\\
&\quad -(-1)^{|k||r|+|r||s|}(\beta_{ik}\beta_{js}+(-1)^{|i||j|}\beta_{jk}\beta_{is})\ell_{rl})\pt{\ell_{kl}},
\end{align*}
\item \begin{align*}
\pil(\ell_{ij}^+) &= -\dfrac{\imath}{2}\bessel(\ell_{ij}).
\end{align*}
\end{itemize}

\subsection{The minimal representation}\label{SS_min}

Following the general approach, the minimal representation of $\mf g$ should be obtainable as a quotient of $\pil$ for specific values of $\lambda$. More specifically, suppose we have a non-trivial subspace $V_\lambda$ of $\mc P_2(\ds K^{m|n})$ on which the Bessel operators act trivially and which is also a $\mf{str}(J)$-module. From the Poincaré-Birkhoff-Witt theorem it then follows that
\begin{align*}
\mc I_\lambda := U(J^-) V_\lambda=  \mc{P}(\ds \K^{m|n}) V_\lambda
\end{align*}
is a submodule of $\pi_\lambda$. Here $U(J^-)$\symindex{$U( \cdot )$} denotes the universal enveloping algebra of $J^-$. We can then define the quotient representation of $\mf g$ on $\mc{P} (\mathds{K}^{m |n}) / \mc{I}_\lambda$. This quotient representation is then a prime candidate for being a minimal representation of $\mf g$. We now determine the elements of $\mc P_2(\ds K^{{\wh m}|2{\wh n}})$ on which the Bessel operators act trivially. An arbitrary element of $Q\in \mc P_2(\ds K^{{\wh m}|2{\wh n}})$ can be written as
\begin{align}\label{Eq_Arbit_Q}
&Q = \sum_{i,j,k,l=1}^{m+2n}\alpha_{ijkl}\ell_{ij}\ell_{kl}, \quad \text{ with }\\ &\alpha_{ijkl} = (-1)^{|i||j|}\alpha_{jikl} = (-1)^{|k||l|}\alpha_{ijlk}=(-1)^{(|i|+|j|)(|k|+|l|)}\alpha_{klij}\in \ds C.\nonumber
\end{align}

\begin{lemma}\label{Lemma_Q_Bessel}
Suppose $Q\in \mc P_2(\ds K^{{\wh m}|2{\wh n}})$ is given by equation \eqref{Eq_Arbit_Q}. Then the Bessel operators act trivially on $Q$ if and only if
\begin{align*}
2(-1)^{|i||j|}\lambda \alpha_{ijkl} &= (-1)^{|i||k|}\alpha_{jkil}+(-1)^{|k||l|+|i||l|}\alpha_{jlik},
\end{align*}
for all $i,j,k,l\in\{1, \ldots, m+2n\}$. 
\end{lemma}

\begin{proof}
This is a long and straightforward calculation. See Appendix \ref{AppS_Long}, Proposition \ref{Prop_App_Q_Bessel}.
\end{proof}

Now suppose $Q$ is given by equation \eqref{Eq_Arbit_Q} and that the Bessel operators act trivially on $Q$. We determine the values of $\lambda$ for which $Q$ is non-trivial. Lemma \ref{Lemma_Q_Bessel} implies
\begin{align*}
2\lambda (-1)^{|i||j|}\alpha_{ijkl} &= (-1)^{|i||k|}\alpha_{jkil}+(-1)^{|k||l|+|i||l|}\alpha_{jlik},\\
2\lambda (-1)^{|k||j|}\alpha_{jkil} &= (-1)^{|j||i|}\alpha_{kijl}+(-1)^{|i||l|+|j||l|}\alpha_{klji},\\
2\lambda (-1)^{|l||j|}\alpha_{jlik} &= (-1)^{|j||i|}\alpha_{lijk}+(-1)^{|i||k|+|j||k|}\alpha_{lkji}.
\end{align*}
If we multiply the second equation by $(-1)^{|i||k|+|j||k|}$ and the third equation by $(-1)^{|i||l|+|j||l|+|k||l|}$ we get
\begin{align}\label{Eq_lambda_1}
2\lambda (-1)^{|i||j|}\alpha_{ijkl} &= (-1)^{|i||k|}\alpha_{jkil}+(-1)^{|k||l|+|i||l|}\alpha_{jlik},\\\label{Eq_lambda_2}
2\lambda(-1)^{|i||k|}\alpha_{jkil} &= (-1)^{|l||i|+|l||k|}\alpha_{jlik}+(-1)^{|j||i|}\alpha_{ijkl},\\\label{Eq_lambda_3}
2\lambda (-1)^{|k||l|+|i||l|}\alpha_{jlik} &= (-1)^{|k||i|}\alpha_{jkil}+(-1)^{|j||i|}\alpha_{ijkl}.
\end{align}
Equation \eqref{Eq_lambda_1} can be rewritten as the equations
\begin{align}\label{Eq_lambda_1_1}
(-1)^{|i||k|}\alpha_{jkil} &= 2\lambda (-1)^{|i||j|}\alpha_{ijkl}-(-1)^{|k||l|+|i||l|}\alpha_{jlik},\\\label{Eq_lambda_1_2}
(-1)^{|k||l|+|i||l|}\alpha_{jlik} &= 2\lambda (-1)^{|i||j|}\alpha_{ijkl}-(-1)^{|i||k|}\alpha_{jkil}.
\end{align}
If we combine \eqref{Eq_lambda_2} and \eqref{Eq_lambda_3} with \eqref{Eq_lambda_1_1} and \eqref{Eq_lambda_1_2}, respectively, we obtain
\begin{align*}
4\lambda^2 \alpha_{ijkl} &= (-1)^{|l||i|+|l||k|+|i||j|}\alpha_{jlik}+\alpha_{ijkl}\\
&\quad +2\lambda(-1)^{|k||l|+|i||l|+|i||j|}\alpha_{jlik},\\
4\lambda^2 (-1)^{|k||l|+|i||l|}\alpha_{jlik} &= (-1)^{|l||i|+|l||k|}\alpha_{jlik}+(-1)^{|j||i|}\alpha_{ijkl}\\
&\quad +2\lambda(-1)^{|j||i|}\alpha_{ijkl},
\end{align*}
therefore,
\begin{align*}
(4\lambda^2 -1)\alpha_{ijkl}&= (-1)^{|l||i|+|l||k|+|i||j|}(1+2\lambda)\alpha_{jlik},\\
(4\lambda^2- 1)\alpha_{jlik} &= (-1)^{|j||i|+|k||l|+|i||l|}(1 +2\lambda)\alpha_{ijkl}.
\end{align*}

For $4\lambda^2- 1 = 0$ we now have
\begin{align*}
(1 +2\lambda)\alpha_{ijkl} = 0,
\end{align*}
for all $i,j,k,l\in\{1, \ldots, m+2n\}$. This implies either $Q=0$ or $\lambda= -1/2$.

For $4\lambda^2- 1 \neq 0$, we now have
\begin{align*}
(4\lambda^2- 1)^2\alpha_{ijkl}&= (2\lambda+1)^2\alpha_{ijkl}
\end{align*}
and therefore
\begin{align*}
\lambda(\lambda- 1)\alpha_{ijkl}&= 0,
\end{align*}
for all $i,j,k,l\in\{1, \ldots, m+2n\}$. This implies either $Q=0$ or $\lambda\in\{0,1\}$. The case $\lambda=0$ corresponds to working with the trivial character of $\mf{str}(J)$. Since then the Bessel operators act trivially on all of $\mc P_1(\ds K^{\wh m|2\wh n})$ this does not lead to what we would consider minimal representations.

Therefore, the values of interest are $\lambda= -1/2$ and $\lambda=1$. For $\lambda=1$ the calculations above give us
\begin{align*}
\alpha_{ijkl} = (-1)^{|k||j|}\alpha_{ikjl} =  (-1)^{(|j|+|k|)|l|}\alpha_{iljk},
\end{align*}
which implies the Bessel operators act trivially on
\begin{align*}
V_{1} := \left\lbrace\sum_{i,j,k,l}\alpha_{ijkl}(\ell_{ij}\ell_{kl} + (-1)^{|k||j|}\ell_{ik}\ell_{jl} + (-1)^{(|j|+|k|)|l|}\ell_{il}\ell_{jk}): \alpha_{ijkl}\in \ds K \right\rbrace
\end{align*}
For $\lambda=-1/2$ the calculations above show that the Bessel operators act trivially on
\begin{align*}
V_{-\frac{1}{2}} &:= \left\lbrace\sum_{i,j,k,l}\alpha_{ijkl}(\ell_{ij}\ell_{kl} - (-1)^{|k||j|}\ell_{ik}\ell_{jl}): \alpha_{ijkl}\in \ds K \right\rbrace.
\end{align*}
Furthermore, $V_{1}$ and $V_{-\frac{1}{2}}$ are $\mf{istr}(J)$-modules.\symindex{$V_1$ and $V_{-\frac{1}{2}}$}

Define
\begin{align*}
\mc I_\lambda := U(J^-) V_\lambda=  \mc{P}(\ds \K^{{\wh m}|2{\wh n}}) V_\lambda,
\end{align*}\symindex{$\mc I_\lambda$}
which is a submodule of $\pi_\lambda$ according to the Poincaré-Birkhoff-Witt theorem. In the rest of the paper we will study the quotient representations of $\mf g$ on $\mc{P} (\mathds{K}^{m |n}) / \mc{I}_\lambda$ for $\lambda = -\frac{1}{2}$ and $\lambda=1$. These two quotient representations have different behaviours and are in general not equivalent, which is why we handle them separately from Section \ref{S_lambda12} onwards. The $\lambda = -\frac{1}{2}$ case will be the main subject of the paper and it generalises the symplectic case. The $\lambda=1$ case will be discussed in Section \ref{S_lambda_1} and generalises the split orthogonal case.

Denote the superdimension of a supervector space $V = V_{\ol 0}\oplus V_{\ol 1}$ by $\sdim(V) := \dim(V_{\ol 0})-\dim(V_{\ol 1})$.\symindex{$\sdim$} For $J$ we find
\begin{align*}
\sdim(J) &= \dfrac{1}{2}M(M+1),
\end{align*}
with $M := m-2n$. For $V_\lambda$ we have the following superdimensions.

\begin{Prop}\label{Prop_Dim}
We have
\begin{align*}
\sdim(V_1) &= \dfrac{1}{24}M(M+1)(M+2)(M+3),\\
\sdim(V_{-\frac{1}{2}}) &= \dfrac{1}{12}(M-1)M^2(M+1), 
\end{align*}
with $M = m-2n$.
\end{Prop}

\begin{proof}
This is a straightforward counting exercise. More generally, we can show that $V_1$ is a $(\frac{1}{24}(m^4+24m^2n^2+16n^4+6m^3-12m^2n+24mn^2-48n^3+11m^2-12mn+44n^2+6m-12n)|\frac{1}{3}mn(m^2+4n^2+3m-6n+4))$-dimensional subspace of $\mc P_2\big(\K^{{\wh m}|2{\wh n}}\big)$ and $V_{-\frac{1}{2}}$ is a $(\frac{1}{12}(m^4 +24m^2 n^2+16n^4-m^2-12mn-4n^2)|\frac{2}{3}mn(m^2+4n^2-2))$-dimensional subspace of $\mc P_2\big(\K^{{\wh m}|2{\wh n}}\big)$.

Suppose first that $\lambda=1$. We can distinguish five distinct types of basis elements in $V_1$, depending on how many of the indices are equal to each other.
\begin{itemize}
\item We have $\ell_{ii}^2$, for $i\in\{1, \ldots, m\}$, which gives $m$ even elements.
\item We have $\ell_{ii}\ell_{ij}$, for $i\in\{1, \ldots, m\}$ and $j\in\{1, \ldots, m+2n\}$, with $i\neq j$. This gives $m(m+2n-1)$ elements of which $2mn$ are odd.
\item We have $2\ell_{ij}^2+\ell_{ii}\ell_{jj}$, for $i,j\in\{1, \ldots, m\}$, with $i\neq j$. This gives $\frac{1}{2}m(m-1)$ even elements.
\item We have $2\ell_{ik}\ell_{il}+\ell_{ii}\ell_{kl}$, for $i\in\{1, \ldots, m\}$ and $k,l\in\{1, \ldots, m+2n\}$, with $i\neq k\neq l \neq i$. This gives $\frac{1}{2}m(m+2n-1)(m+2n-2)$ elements of which $2mn(m-1)$ are odd.
\item We have $\ell_{ij}\ell_{kl} + (-1)^{|k||j|}\ell_{ik}\ell_{jl} + (-1)^{(|j|+|k|)|l|}\ell_{il}\ell_{jk}$, for $i,j,k,l\in\{1, \ldots, m+2n\}$ and distinct. This gives $\frac{1}{24}(m+2n)(m+2n-1)(m+2n-2)(m+2n-3)$ elements of which $\frac{1}{3}mn((m-1)(m-2)+(2n-1)(2n-2))$ are odd.
\end{itemize}
Now suppose $\lambda=-1/2$. Similarly, we can distinguish five distinct types of basis elements in $V_{-\frac{1}{2}}$.
\begin{itemize}
\item We have $\ell_{ij}^2-\ell_{ii}\ell_{jj}$, for $i,j\in\{1, \ldots, m\}$, which gives $\frac{1}{2}m(m-1)$ even elements.
\item We have $\ell_{ij}^2$, for $i,j\in\{m+1, \ldots, m+2n\}$, with $i\neq j$. This gives $n(2n-1)$ even elements.
\item We have $\ell_{ii}\ell_{kl}-\ell_{ik}\ell_{il}$, for $i\in\{1, \ldots, m\}$ and $k,l\in\{1, \ldots, m+2n\}$, with $i\neq k\neq l \neq i$. This gives $\frac{1}{2}m(m+2n-1)(m+2n-2)$ elements of which $2mn(m-1)$ are odd.
\item We have $\ell_{ik}\ell_{il}$, for $i\in\{m+1, \ldots, m+2n\}$ and $k,l\in\{1, \ldots, m+2n\}$, with $i\neq k\neq l \neq i$. This gives $n(m+2n-1)(m+2n-2)$ elements of which $2mn(2n-1)$ are odd.
\item We have $\ell_{ij}\ell_{kl} - (-1)^{|k||j|}\ell_{ik}\ell_{jl}$, for $i,j,k,l\in\{1, \ldots, m+2n\}$ and distinct. This gives $\frac{1}{12}(m+2n)(m+2n-1)(m+2n-2)(m+2n-3)$ elements of which $\frac{2}{3}mn((m-1)(m-2)+(2n-1)(2n-2))$ are odd.
\end{itemize}
If we subtract the number of odd basis elements from the number of even basis elements, we get the superdimension of $V_\lambda$.
\end{proof}

\subsection{The polynomial realisation $\rol$}\label{SS_rol}

We introduce the notations
\begin{align*}
\mf k_{mcs} := \mf{so}(2n)\oplus \mf{so}(2n)\oplus \mf u(m), && \mf k_c &:= \{(a, [L_b,L_c], -a) : a,b,c\in J\}.
\end{align*}\symindex{$\mf k_{mcs}$ and $\mf k_c$}

Then $\mf k_{mcs}$ is a maximal compact subalgebra of the even part of $\mf g$ and $\mf k_c\cap \mf{istr}(J) = \Inn(J)\cong \mf{osp}(m|2n)$.

\begin{Prop}
We have $\mf k_c \cong \mf u(m|2n,\beta')$ as (real) Lie superalgebras. Here $\beta'$ denotes the superhermitian, non-degenerate, even sesquilinear form where the matrix form is given by $\beta$.
\end{Prop}

\begin{proof}
A basis of $\mf k_c$ is given by
\begin{align*}
&U_{\ot i, \ul j}-(-1)^{|i||j|}U_{\ot j, \ul i}, &&\text{ for } i<j,\\
&2U_{\ot i, \ul i}, &&\text{ for } |i|=1,\\
&U_{\ul i, \ul j}+ U_{\ot i, \ot j}, &&\text{ for } i<j,\\
&U_{\ul i, \ul i}+ U_{\ot i, \ot i}, &&\text{ for } |i|=0.
\end{align*}
and a basis of $\mf u(m|2n,\beta')$ is given by
\begin{align*}
&\sum_{k=1}^{m+2n}\beta_{jk}E_{ik}-(-1)^{|i||j|}\beta_{ik}E_{jk}, &&\text{ for } i<j,\\
&2\sum_{k=1}^{m+2n}\beta_{ik}E_{ik}, &&\text{ for } |i|=1,\\
&\imath\sum_{k=1}^{m+2n}\beta_{jk}E_{ik}+(-1)^{|i||j|}\beta_{ik}E_{jk}, &&\text{ for } i<j,\\
&2\imath\sum_{k=1}^{m+2n}\beta_{ik}E_{ik}, &&\text{ for } |i|=0.
\end{align*}
From a straightforward verification, it follows that mapping the basis elements of $\mf k_c$ to the respective basis elements of $\mf u(m|2n,\beta')$ induces a Lie superalgebra isomorphism.
\end{proof}

Recall that $e^\pm$ denotes the unit of $J^\pm$. We define the Cayley transform $c \in \End(\mf g_{\ds C})$\symindex{$c$} as
\begin{align*}
c := \exp(\frac{\imath}{2}\ad(e^-))\exp(\imath\ad(e^+)),
\end{align*}
which has the following property.
\begin{Prop}\label{Prop Cayley}
Using the decomposition $\TKK(J_\C)= J^-_{\C}\oplus \mathfrak{istr}(J_{\C})\oplus J^+_{\C} $ we obtain the following explicit expression for the Cayley transform
\begin{itemize}
\item $c(a,0,0) = \left(\dfrac{a}{4}, \imath L_a, a\right)$
\item $c(0,L_a+I,0) = \left(\imath \dfrac{a}{4}, I, -\imath a\right)$
\item $c(0,0,a) = \left(\dfrac{a}{4}, -\imath L_a, a\right)$,
\end{itemize}
with $a\in J_{\C}$ and $I\in \Inn(J_{\C})$. It induces a Lie superalgebra isomorphism:
\begin{align*}
c:\;\mf k_{c,\ds C} \rightarrow \mathfrak{istr}(J_{\C}),\quad (a,I,-a) \mapsto I + 2\imath L_a.
\end{align*}
\end{Prop}
\begin{proof}
This follows from the same straightforward calculations as given in the proof of \cite[Proposition 5.1]{BCD}.
\end{proof}

Let $\pil_{,\ds C}$ denote the unique representation of $\mf g_{\ds C}$ obtained as the $\ds C$-linear extension of $\pil$. As discussed in the general approach, we can define a Fock model of $\mf g$ by twisting $\pil_{,\ds C}$ with an endomorphism $\gamma\in \End(\mf g_{\ds C})$ which maps $\mf k_{mcs, \ds C}$ into $\mf{istr}(J_\ds C)$. For $n\leq 1$ we have that $\mf k_{mcs}$ is a maximal compact subalgebra of the even part of $\mf k_c$ and then the Cayley transform $c$ is as desired. In particular, in the symplectic case we have $\mf k_{mcs}=\mf k_c$ and $c$ is the Cayley transform used in
\cite{HKMO}.
For $n\geq 2$ we no longer have that $\mf k_{mcs}$ is a subalgebra of $\mf k_c$. However, for $\lambda= -\frac{1}{2}$ we will still twist our representation with $c$ even when $n \geq 2$.

The first reason for this is that then we will not have to deal with the $n\leq 1$ case separately. For $n \geq 2$ we will not obtain a Fock model in the sense of
\cite{BCARXIV2},
but we obtain a Fock-like model which still has many of the desired properties. The second reason is that in the $\lambda= -\frac{1}{2}$ case, we will not need the method given in the general approach to integrate our representation to the group level. Therefore, twisting by $c$ is not necessary, but doing so makes it easier to compare our model with the Fock model in the symplectic case.

From now on we denote by $\rol$\symindex{$\rol$} the polynomial realisation of $\mf g$ obtained by twisting $\pil_{,\ds C}$ with the Cayley transform $c$, i.e.\ $\rol := \pil_{,\ds C}\circ c$.

\subsection{The Schr\"odinger model $W_\lambda$ and Fock model $F_\lambda$}\label{ss Sch and F mods}

We can define the following Schr\"odinger and Fock models using the polynomial realisation $\rol$. We will sometimes call these models minimal to distinguish them from the Schr\"odinger and Fock models defined later.

\begin{Def}
We define the \textbf{(minimal) Fock representation} as
\begin{align*}
F_\lambda := \mc{P}(\ds C^{{\wh m}|2{\wh n}})/\mc I_\lambda,
\end{align*}\symindex{$F_\lambda$}
where the $\mf g$-module structure is given by $\rol$.
\end{Def}

Let $e^\pm$ denote the unit of $J^\pm$ and $e$ the corresponding element in $\mc{P}(\ds R^{{\wh m}|{2\wh n}})$. As discussed in \cite{BCARXIV2}, with this Fock-like model we can associate a Schr\"odinger-like model by acting with $\pil$ on $\pil(C)^{-1}1$. Here $C := \exp(\frac{\imath}{2}e^-)\exp(\imath e^+)$ has the property $c = \Ad(C)$. Explicitly, we obtain the following Schr\"odinger model.

\begin{Def}
We define the \textbf{(minimal) Schr\"odinger representation} as
\begin{align*}
W_\lambda &:= U(\mf g)\exp(-2e) \mod \mc I_\lambda,
\end{align*}
where $U(\mf g)$ is the universal enveloping algebra of $\mf g$ and the $\mf g$-module structure is given by $\pil$.
\end{Def}\symindex{$W_\lambda$}

Note that the operators occurring in $\pil$ are not only well-defined on polynomials but can be extended to smooth functions. Therefore, $\pil(X) \exp(-2e)$ is well-defined for all $X\in \mf g$. Unlike our Fock model, this Schr\"odinger model is a Schr\"odinger model in the sense of the general approach when $\lambda= -\frac{1}{2}$. Indeed, the only thing we need to prove is that $\exp(-2e)$ is $\mf k_{mcs}$-finite, which will follow from Proposition \ref{Prop_k_finite}.

\begin{Opm}\label{Rem_Gen_Schrod}
Perhaps a more natural choice for the generator of the Schr\"odinger representation would be $\exp(-\tr(\ell))$, especially for $n=1$. Here
\begin{align*}
\tr(\ell) := \sum_{i,j=1}^{m+2n}2^{-|i||j|}\ell_{ij}\beta^{ij} = \sum_{i=1}^{m}\ell_{ii} + \sum_{m+1\leq i<j\leq m+2n}\ell_{ij}\beta^{ij}
\end{align*}
denotes an element of $\mc P(\ds R^{\wh m|2\wh n})$ associated to the Jordan trace of $J$. This choice corresponds more closely to the choice of generators for Euclidean Lie algebras used in \cite{HKM}. The reason we opt for $\exp(-2e)$ instead is mainly because of the simple connection to our Fock model using the Cayley transform. Note that
\begin{align*}
2e = \sum_{i,j=1}^{m+2n}\ell_{ij}\beta^{ij} = \sum_{i=1}^{m}\ell_{ii} + \frac{1}{2}\sum_{m+1\leq i<j\leq m+2n}\ell_{ij}\beta^{ij}
\end{align*}
only differs slightly from $\tr(\ell)$ and is equal in the symplectic case.
\end{Opm}

\section{The metaplectic representation of $\mf{spo}(2m|2n,2n)$}\label{S_lambda12}

From now on we will always assume $\lambda=-\frac{1}{2}$, unless otherwise stated.

In this section, we construct the metaplectic representation of $\mf{spo}(2m|2n,2n)$ as the composition of two minimal representations. The Schr\"odinger and Fock models we obtained in the previous section generalise the corresponding models on the minimal orbit in the symplectic case as constructed in \cite{HKM} and \cite{HKMO}. However, to obtain the minimal representation as a component of the metaplectic representation, we first need to introduce the folding isomorphism.

\subsection{The folding isomorphism $\psi$}

Recall from the introduction that there exists a folding map that induces a unitary isomorphism between the classical $L^2_{\text{even}}$ space and the $L^2$ space on the minimal orbit and a complexified folding map induces a unitary isomorphism between the even part of the Fock space and the Fock space on the minimal orbit. We will generalise the isomorphisms induced by the folding maps to the super case.

Define the space of even (resp. odd) degree superpolynomials by
\begin{align*}
\mc P_{\text{even}}(\ds K^{m|2n}):= \bigoplus\limits_{k=0}^{\infty} \mc P_{2k}(\ds K^{m|2n}),\qquad \mc P_{\text{odd}}(\ds K^{m|2n}):= \bigoplus\limits_{k=0}^{\infty} \mc P_{2k+1}(\ds K^{m|2n}).
\end{align*}\symindex{$\mc P_{\text{even}}(\ds K^{m\vert 2n})$ and $\mc P_{\text{odd}}(\ds K^{m\vert 2n})$}
Here the ``even'' and ``odd'' in $\mc P_{\text{even}}(\ds K^{m|2n})$ and $\mc P_{\text{odd}}(\ds K^{m|2n})$, respectively, refers to the degree and not the parity of the superpolynomial terms.

\begin{Def}\symindex{$\psi$}
The \textbf{folding isomorphism} $\psi$ is defined as
\begin{align*}
\begin{matrix}
\psi : &\mc P(\ds K^{\wh m|2\wh n})/\mc I_{-\frac{1}{2}} &\rightarrow &\mc P_{\text{even}}(\ds K^{m|2n}),\\
 &\ell_{ij}&\mapsto &\ell_i\ell_j,
\end{matrix}
\end{align*}
where $(\ell_i)_{i=1}^{m+2n}$ denotes the variables of $\mc P(\ds K^{m|2n})$.
\end{Def}

Note that this is a well-defined isomorphism since acting modulo $\mc I_{-\frac{1}{2}}$ on $\mc P(\ds K^{\wh m|2\wh n})$ gives us precisely the commutation relations of the variables of $\mc P_{\text{even}}(\ds K^{m|2n})$.

\subsection{The polynomial realisations $\dotpi$ and $\dotrho$}\label{ss pol real}

We can now define polynomial realisations $\dotpi$ and $\dotrho$ on $\mc P_{\text{even}}(\ds K^{m|2n})$ by
\begin{align*}
\dotpi(X) := \psi\circ \pil(X) \circ \psi^{-1} \quad \text{ and } \quad \dotrho(X) := \psi\circ \rol(X)\circ \psi^{-1}
\end{align*}\symindex{$\dotpi$ and $\dotrho$}
for all $X\in \mf g$. We can extend this realisation to smooth functions. To achieve this, we give this realisation explicitly.

We first introduce the following notations. Set $\ell^j := \sum_{i=1}^{m+2n}\ell_i\beta^{ij}$ and $\pt i := \sum_{j=1}^{m+2n}\beta_{ij}\pt{\ell_j}$\symindex{$\ell^i$ and $\pt i$}. We also introduce the following operators on $\mc P(\ds K^{m|2n})$.
\begin{gather*}
R^2 := \sum_{i,j=1}^{m+2n} \beta^{ij}\ell_i\ell_j, \qquad \mbox{and} \qquad \Delta := \sum_{i,j=1}^{m+2n}\beta^{ij}\pt{i} \pt{j}.
\end{gather*}\symindex{$R^2$ and $\Delta$}
Here, the operator $R^2$ is called the square of the radial coordinate and acts through multiplication and $\Delta$ is called the Laplacian. Note that $\psi(2e) = R^2$. We also have the Euler operator $\ds E = \sum_{i,j=1}^{m+2n}\ell_i\pt{\ell_i}$ on $\mc P(\ds K^{m|2n})$. Note that this Euler operator is twice the Euler operator on $F_\lambda$, i.e.\ $\psi\circ (2\ds E) = \ds E\circ \psi$. 

\begin{lemma}\label{Lemsltriple}
The operators $R^2$, $\E$ and $\Delta$ satisfy
\begin{gather*}
[\Delta, R^2] = 4\E+2M, \qquad [\Delta, \E] = 2\Delta, \qquad [R^2,\E] = -2R^2,
\end{gather*}
where $M=m-2n$. In particular, $\big(R^2,\E+\frac{M}{2}, -\frac{\Delta}{2}\big)$ forms an $\mf{sl}(2)$-triple.
Furthermore, they commute in $\End\big(\mc P\big(\ds K^{m|2n}\big)\big)$ with the operators \[ L_{ij} := \ell_i\pt j - (-1)^{|i||j|}\ell_j\pt i. \]\symindex{$L_{ij}$}
\end{lemma}
\begin{proof}A straightforward calculation or see, for example, \cite{DS}.
\end{proof}

The following proposition is a direct consequence of Lemma \ref{Lemsltriple}.

\begin{Prop}
We have $\dotpi(e^+) = -\dfrac{\imath}{4}\Delta$. In particular, the Bessel operator of two times the unit acts as the Laplacian, i.e.\ $\psi\circ \bessel(2e) = \Delta\circ \psi$.
\end{Prop}
\begin{proof}
The $\mf{sl}(2)$-triple $\left\lbrace e^-, 2L_{e},e^+ \right\rbrace$ in Section \ref{SSTKK2} implies that
\begin{align*}
\left\lbrace \pil(e^-), \pil(2L_{e}),\pil(e^+)\right\rbrace
\end{align*}
is also an $\mf{sl}(2)$-triple. We have
\begin{align*}
\dotpi(e^-) &= \psi\circ(-2\imath e)\circ\psi^{-1} = -\imath R^2,\\
\dotpi(2L_{e}) &= \psi\circ(-\dfrac{M}{2}-2\ds E)\circ \psi^{-1} = -(\dfrac{M}{2}+\ds E).
\end{align*}
Lemma \ref{Lemsltriple} now implies that $\dotpi(e^+) = -\dfrac{\imath}{4}\Delta$, as desired.
\end{proof}

To give $\dotpi$ explicitly, we need the following lemma.

\begin{lemma}\label{Lembesderiv}
Define
\begin{align*}
E_{ij}(p) &:= \sum_{k,l=1}^{m+2n}(1+\delta_{kl})\beta_{jl}\ell_{ik}\pt{\ell_{lk}}(p)
\end{align*}
and $\ot E_{ij}(\ot p) := \ell_i\pt j (\ot p)$, for $p\in \mc P(\ds K^{\wh m|2\wh n})/\mc I_{-\frac{1}{2}}$ and $\ot p\in \mc P(\ds K^{m|2n})$. We have
\begin{align*}
&\psi(E_{ij}(p)) =\ot E_{ij}(\ot p),
\end{align*}
for all $i,j\in\{1, \ldots, m+2n\}$ with $\psi(p) =\ot p$.
\end{lemma}

\begin{proof}
We will use induction on the degree of polynomials. The case of degree $0$ polynomials is trivial. Suppose we have proven the lemma for $p\in \mc P_k(\ds K^{\wh m|2\wh n})/\mc I_{-\frac{1}{2}}$, $k\in \ds N$. We now look at $\ell_{rs}p$. On the one hand, we have
\begin{align*}
E_{ij}(\ell_{rs}q) &= E_{ij}(\ell_{rs})q + (-1)^{(|r|+|s|)(|i|+|j|)}\ell_{rs}E_{ij}(q)\\
&= (\beta_{jr}\ell_{is}+(-1)^{|r||s|}\beta_{js}\ell_{ir}) + (-1)^{(|r|+|s|)(|i|+|j|)}\ell_{rs}E_{ij}(q)
\end{align*}
and on the other hand, we have
\begin{align*}
\ot E_{ij}(\ell_r\ell_s\ot q) &= \ot E_{ij}(\ell_r\ell_s)\ot q + (-1)^{(|r|+|s|)(|i|+|j|)}\ell_r\ell_s\ot E_{ij}(\ot q)\\
&=(\beta_{jr}\ell_{i}\ell_s+(-1)^{|r||s|}\beta_{js}\ell_{i}\ell_r) + (-1)^{(|r|+|s|)(|i|+|j|)}\ell_r\ell_s\ot E_{ij}(\ot q).
\end{align*}
Using the induction hypothesis, we get the desired result.
\end{proof}

\begin{theorem}\label{Th pil on psiW}
The action $\dotpi$ of $\mf g$ on $\mc P_{\text{even}}(\ds K^{m|2n})$ is given by
\begin{itemize}
\item $\dotpi(\ell_{ij}^-) = -2\imath\ell_i\ell_j$,
\item $\dotpi(2L_{\ell_{ij}}) = -\beta_{ij}-(\ell_i\pt j +(-1)^{|i||j|}\ell_j\pt i)$,
\item $\dotpi(4[L_{\ell_{ij}}, L_{\ell_{rs}}])$\\
\indent $= \beta_{jr}L_{is}+(-1)^{|r||s|}\beta_{js}L_{ir}+(-1)^{|i||j|}\beta_{ir}L_{js}+(-1)^{|i||j|+|r||s|}\beta_{is}L_{jr}$, 
with $L_{ij} = \ell_i\pt j - (-1)^{|i||j|}\ell_j\pt i$ and
\item $\dotpi(2\imath \ell_{ij}^+) = \psi(\bessel(\ell_{ij})) = \pt i \pt j = \sum_{k,l=1}^{m+2n}\beta_{ik}\beta_{jl}\pt{\ell_k}\pt{\ell_l}$,
\end{itemize}
for all $i,j,r,s\in\{1, \ldots, m+2n\}$.
\end{theorem}

\begin{proof}
The first equation follows directly from the definition of $\dotpi$. The second and third equations follow directly from Lemma \ref{Lembesderiv}. For the last equation, we will use induction on the degree of polynomials. The case of degree $0$ polynomials is trivial. Note that the degree $1$ case follows from
\begin{align*}
\bessel(\ell_{ij})\ell_{rs} = \beta_{jr}\beta_{is} + (-1)^{rs}\beta_{js}\beta_{ir} = \pt i \pt j \ell_r \ell_s,
\end{align*} 
Now, suppose we have proven the theorem for $p\in \mc P_k(\ds K^{\wh m|2\wh n})/\mc I_{-\frac{1}{2}}$ and define $\ot p := \psi(p)\in \mc P_{2k}(\ds C^{m|2n})$. We now look at $\ell_{rs}p$. On the one hand, we have
\begin{align*}
\bessel(\ell_{ij})(\ell_{rs}p) &= (-1)^{(|i|+|j|)(|r|+|s|)}\ell_{rs}\bessel(\ell_{ij})(p)+[\bessel(\ell_{ij}), \ell_{rs}](p)\\
&= (-1)^{(|i|+|j|)(|r|+|s|)}\ell_{rs}\bessel(\ell_{ij})(p) +(\beta_{jr}\beta_{is}+(-1)^{|i||j|}\beta_{js}\beta_{ir})p \\
&\quad + \sum_{k,l=1}^{m+2n}(1+\delta_{kl})((-1)^{|i||s|}\beta_{il}\beta_{jr}\ell_{sk}+(-1)^{|i||r|}\beta_{il}\beta_{js}\ell_{rk}\\
&\quad +(-1)^{|i||j|+|j||s|}\beta_{jl}\beta_{ir}\ell_{sk}+(-1)^{|i||j|+|j||r|}\beta_{jl}\beta_{is}\ell_{rk})\pt{\ell_{lk}}(p)
\end{align*}
and on the other hand, we have
\begin{align*}
\pt i \pt j (\ell_r\ell_s \ot p) &= (-1)^{(|i|+|j|)(|r|+|s|)}\ell_r\ell_s \pt i \pt j (\ot p) +\pt i\pt j (\ell_r\ell_s)\ot p\\
&\quad + (-1)^{|i|(|j|+|r|+|s|)}\pt j(\ell_r\ell_s)\pt i (\ot p)+ (-1)^{|j|(|r|+|s|)}\pt i(\ell_r\ell_s)\pt j (\ot p)\\
&= (-1)^{(|i|+|j|)(|r|+|s|)}\ell_r\ell_s \pt i \pt j (\ot p) ++(\beta_{jr}\beta_{is}+(-1)^{|i||j|}\beta_{js}\beta_{ir})\ot p\\
&\quad + (-1)^{|i|(|j|+|r|+|s|)}(\beta_{jr}\ell_s+(-1)^{|j||r|}\beta_{js}\ell_r)\pt i (\ot p)\\
&\quad + (-1)^{|j|(|r|+|s|)}(\beta_{ir}\ell_s+(-1)^{|i||r|}\beta_{is}\ell_r)\pt j (\ot p).
\end{align*}
The theorem now follows from using the induction hypothesis together with Lemma \ref{Lembesderiv}.
\end{proof}

The operators occurring in the explicit form of $\dotpi$ given in Theorem \ref{Th pil on psiW} are all well-defined on smooth functions. This means there is a canonical way to extend the realisations to act on smooth functions in the variable $x$. In particular, we may view $\dotpi$ as a polynomial realisation on $\mc P(\ds K^{m|2n})$, where the action is given by the explicit form in Theorem \ref{Th pil on psiW}. We can then also extend $\dotrho$ to act on $\mc P(\ds K^{m|2n})$ using $\dotrho = \dotpi_{\ds C}\circ c$.

Up until now, we used the basis of $\mf g$ obtained from the TKK-construction. However, in certain situations, it will be more convenient to use the $(U_{ij})$-basis of $\mf g$ given in Section \ref{SSTKK2}.

\begin{Prop}\label{Prop_pi_in_U}
The representation $\dotpi$ of $\mf g$ is given by
\begin{itemize}
\item $\dotpi(U_{\ul i, \ul j}) = 2\imath\ell_i\ell_j$,\\
\item $\dotpi(U_{\ot i, \ul j}) = -\dfrac{1}{2} \beta_{ij}- (-1)^{|i||j|}\ell_j\pt i$,\\
\item $\dotpi(U_{\ot i, \ot j}) = -\dfrac{\imath}{2}\pt i \pt j$.
\end{itemize}
\end{Prop}
\begin{proof}
This follows from a straightforward verification using the isomorphism $\phi$ given in Section \ref{SSTKK2}.
\end{proof}

\subsection{The Schr\"odinger model $\ot W$ and Fock model $\ot F$}

The minimal Schr\"odinger and Fock models generalise the corresponding models on the minimal orbit in the symplectic case. Using the polynomial realisations $\dotpi$ and $\dotrho$ we can now also construct Schr\"odinger and Fock models that generalise the corresponding models of the metaplectic representation. In the symplectic case, the metaplectic representation decomposes into a minimal representation acting on even functions and a minimal representation acting on odd functions. Therefore, we introduce the following models.

\begin{Def}
We define the \textbf{even Fock representation} as
\begin{align*}
\ot F_e := U(\mf g)1 \mod \mc I_\lambda
\end{align*}
and the \textbf{odd Fock representation} as
\begin{align*}
\ot F_o := U(\mf g)\ell_1 \mod \mc I_\lambda,
\end{align*}
where $U(\mf g)$ is the universal enveloping algebra of $\mf g$ and the $\mf g$-module structure is given by $\dotrho$. We also define the \textbf{(metaplectic) Fock representation} as
\begin{align*}
\ot F := \ot F_e\oplus \ot F_o.
\end{align*}\symindex{$\ot F_e$, $\ot F_o$ and $\ot F$}
\end{Def}

\begin{Def}
We define the \textbf{even Schr\"odinger representation} as
\begin{align*}
\ot W_e := U(\mf g)\exp(-R^2) \mod \mc I_\lambda
\end{align*}
and the \textbf{odd Schr\"odinger representation} as
\begin{align*}
\ot W_o := U(\mf g)\ell_1 \exp(-R^2) \mod \mc I_\lambda,
\end{align*}
where $U(\mf g)$ is the universal enveloping algebra of $\mf g$ and the $\mf g$-module structure is given by $\dotpi$. We also define the \textbf{(metaplectic) Schr\"odinger representation} as
\begin{align*}
\ot W := \ot W_e\oplus \ot W_o.
\end{align*}\symindex{$\ot W_e$, $\ot W_o$ and $\ot W$}
\end{Def}

Note that $\dotpi(X)\exp(-R^2)$ is well-defined by the explicit form of $\dotpi$ given in Theorem \ref{Th pil on psiW}.

\begin{Prop}\label{Prop_k_finite}
The elements $\exp(-R^2)$ and $\ell_1\exp(-R^2)$ are $\mf k_{mcs}$-finite.
\end{Prop}

\begin{proof}
The maximal compact subalgebra $\mf k_{mcs}\subseteq \mf g$ is explicitly given by
\begin{align*}
\mf k_{mcs} &= \left\langle U_{i, j} + U^{i, j}: |i|=|j|\right\rangle\\
&= \left\langle U_{\ul i, \ul j}+U^{\ul i, \ul j} : |i|=|j|\right\rangle\oplus \left\langle U_{\ot i, \ul j} +U^{\ot i, \ul j}: |i|=|j|\right\rangle,
\end{align*}
where $U^{ij} := \sum_{k,l=1}^{2m+4n}U_{kl}\Omega^{ki}\Omega^{lj}$.
Acting with $2U_{\ul i, \ul j}+2U^{\ul i, \ul j}$ on $\exp(-R^2)$ and $\ell_k\exp(-R^2)$ gives
\begin{align*}
\dotpi(2U_{\ul i, \ul j}+2U^{\ul i, \ul j}) \exp(-R^2) &=\imath (4\ell_i\ell_j-\pt i \pt j)\exp(-R^2) = 2\imath \delta_{ij}\exp(-R^2),\\
\dotpi(2U_{\ul i, \ul j}+2U^{\ul i, \ul j}) \ell_k\exp(-R^2) &=2\imath (\delta_{ij}\ell_k+\delta_{ik}\ell_j+\delta_{jk}\ell_i)\exp(-R^2),
\end{align*}
for all $i,j,k$ with $|i|=|j|=|k|=0$. This implies the action of $2(U_{\ul i, \ul j}+U^{\ul i, \ul j})$ leaves $\exp(-R^2)$ and $\mc P_1(\ds K^m)\exp(-R^2)$ invariant. For $|i|=|j|=0$, the action of $U_{\ot i, \ul j} +U^{\ot i, \ul j}$ is given by
\begin{align*}
\dotpi(U_{\ot i, \ul j} +U^{\ot i, \ul j}) = L_{ij},
\end{align*}
with $L_{ij}$ as defined in Lemma \ref{Lemsltriple}. From Lemma \ref{Lemsltriple} we then obtain that $U_{\ot i, \ul j} +U^{\ot i, \ul j}$  leaves $\exp(-R^2)$ and $\mc P_1(\ds K^m)\exp(-R^2)$ invariant. The action of $U_{i, j} + U^{i, j}$ leaves $\Lambda(\ds K^{2n})\exp(-R^2)$ and $\Lambda(\ds K^{2n})\mc P_1(\ds R^m)\exp(-R^2)$ invariant, for all $i,j$ with $|i|=|j|=1$. As a consequence, the action of $\mf k_{mcs}$ on $\exp(-R^2)$ and $\ell_1\exp(-R^2)$ is contained in $\Lambda(\ds K^{2n})\exp(-R^2)$ and $\Lambda(\ds K^{2n}) \mc P_1(\ds K^m) \exp(-R^2)$, respectively. Since both of these spaces are finite-dimensional, we are finished.
\end{proof}

\section{Properties of the metaplectic representation}\label{S_Prop_Meta}

In this section, we study some properties of the metaplectic representation of $\mf{spo}(2m|2n,2n)$. In Section \ref{S_SB} we will construct the Segal-Bargmann, which intertwines the Schr\"odinger model and the Fock model. Therefore, the properties in this section will only be given for the Fock model. Analogous properties for the Schr\"odinger model can then be obtained using the Segal-Bargmann transform.

\subsection{Decompositions of $\ot F$}

Let us start by introducing spherical harmonics.

\begin{Def}
The \textbf{space of spherical harmonics of degree $k$} is defined by
\begin{gather*}
\mc H_k\big(\ds K^{m|2n}\big) := \big\{ p\in \mc P_k\big(\ds K^{m|2n}\big)\colon \Delta p = 0\big\},
\end{gather*}\symindex{$\mc H_k(\ds K^{m \vert 2n})$}
i.e.\ it is the space of polynomials of homogeneous degree $k$ in the kernel of the Laplacian.
\end{Def}

There exists a decomposition of the space of superpolynomials using these spherical harmonics known as the Fischer decomposition.
\begin{Prop}[\cite{DS}, Theorem 3]\label{PropFC}
If $M = m-2n\not\in -2\N$, then $\mc P\big(\ds K^{m|2n}\big)$ decomposes as
\begin{gather*}
\mc P\big(\ds K^{m|2n}\big) = \bigoplus_{k=0}^\infty \mc P_k\big(\ds K^{m|2n}\big) = \bigoplus_{k=0}^\infty\bigoplus_{j=0}^\infty R^{2j}\mc H_k\big(\ds K^{m|2n}\big),
\end{gather*}
in particular
\begin{align}\label{Eq FC even}
\mc P_{2k}\big(\ds K^{m|2n}\big) = \bigoplus_{l=0}^{k} R^{2l}\mc H_{2k-2l}\big(\ds K^{m|2n}\big)
\end{align}
and
\begin{align}\label{Eq FC odd}
\mc P_{2k+1}\big(\ds K^{m|2n}\big) = \bigoplus_{l=0}^{k} R^{2l}\mc H_{2k-2l+1}\big(\ds K^{m|2n}\big).
\end{align}
\end{Prop}

For $M = m-2n\in -2\ds N$, we need to introduce generalised spherical harmonics. We define
\begin{gather*}
\ot{\mc H}_k\big(\ds K^{m|2n}\big) := \big\{ p\in \mc P_k\big(\ds K^{m|2n}\big)\colon \Delta R^2\Delta p = 0\big\}
\end{gather*}
as the space of generalised spherical harmonics. Note that $\mc H_k\big(\ds K^{m|2n}\big)\subseteq \ot{\mc H}_k\big(\ds K^{m|2n}\big)$. We have the following generalised Fischer decomposition.

\begin{Prop}[\cite{LS}, Corollary 1]\label{PropFCgeneralised}
Define the sets
\begin{align*}
I_M &:= \left\lbrace
\begin{matrix}
\emptyset & \text{if } M\not\in -2\ds N,\\
\{k\in \ds N : 2-\frac{M}{2}\leq k \leq 2-M\} & \text{if } M\in -2\ds N,
\end{matrix}\right.\\
N_k &:= \{k-2j : j\in\{0, \ldots, \left\lfloor k/2\right\rfloor \}\},\\
\ot J_k &:= N_k\cap I_M,\\
J_k^0 &:= \{2-M-l : l\in \ot J_k\} \text{ and }\\
J_k &:= N_k\setminus (\ot J_k\cup J_k^0),
\end{align*}
for $k\in \ds N$ and $M=m-2n$. Then we have that
\begin{align}\label{Eq FC degen}
\mc P_k(\ds K^{m|2n}) = \bigoplus_{l\in \ot J_k}R^{k-l}\ot{\mc H}_l\big(\ds K^{m|2n}\big) \oplus  \bigoplus_{l\in J_k}R^{k-l}\mc H_l\big(\ds K^{m|2n}\big).
\end{align}
\end{Prop}

This Fischer decomposition is precisely the decomposition of $\ot F$ as a $(\mf g, \mf k_c)$-module.

\begin{theorem}\label{ThDecF}
We have the following.
\begin{itemize}
\item[(1)]
\begin{itemize}
\item[(a)] For $M = m-2n\not\in -2\ds N$, explicit decompositions of $\mc P_{2k}(\ds C^{m|2n})$ and $\mc P_{2k+1}(\ds C^{m|2n})$ into irreducible $\Inn(J)$-modules are given by \eqref{Eq FC even} and \eqref{Eq FC odd}, respectively.
\item[(b)] For $M\in -2\ds N$, an explicit decomposition of $\mc P_{k}(\ds C^{m|2n})$ into indecomposable $\Inn(J)$-modules is given by \eqref{Eq FC degen}.
\end{itemize}
\item[(3)] 
\begin{itemize}
\item[(a)] For $M = m-2n\not\in -2\ds N$, $\ot F_e$ and $\ot F_o$ are irreducible $\mf g$-modules and their $\mf k_c$-type decompositions are given by
\begin{align*}
\ot F_e = \bigoplus_{k=0}^{\infty} \mc P_{2k}(\ds C^{m|2n})\quad \text{ and }\quad \ot F_o = \bigoplus_{k=0}^{\infty} \mc P_{2k+1}(\ds C^{m|2n}),
\end{align*}
respectively.
\item[(b)] For $M\in -2\ds N$, $\ot F_e$ and $\ot F_o$ are still indecomposable, but not irreducible $\mf g$-modules.
\end{itemize}
\end{itemize}
\end{theorem}
\begin{proof}
From \cite[Theorem 5.2]{C} it follows that for $M\not\in -2\ds N$ we have an irreducible representation of $\Inn(J)\cong \mf{osp}(m|2n)$ on $\mc H_{2k}(\ds R^{m|2n})$ and from \cite[Theorem 2]{LS} it follows that for $M\in -2\ds N$ we have an indecomposable representation of $\Inn(J)$ on $\ot{\mc H}_{2k}(\ds R^{m|2n})$. We also have $\dotrho([L_{\ell_{ij}},L_{\ell_{rs}}])R^2 = 0$, for all $i,j,r,s\in\{1,\ldots, m+2n\}$, i.e.\ the elements of $\Inn(J)$ leave $R^2$ invariant and therefore also leave powers of $R^2$ invariant. Using the Fischer decomposition and Proposition \ref{PropFCgeneralised} proves (1). Since $\mf k_{c,\ds C} \cong \mf {istr}(J_{\ds C})\cong \mf{gl}_{\ds C}(m|2n)$ by Proposition \ref{Prop Cayley} and Theorem \ref{ThTKK-JP}, we have an irreducible representation of $\mf k_c$ on $\mc P_{k}(\ds C^{m|2n})$, which proves (2).

Define
\begin{align*}
\rho^+ := \dotrho(c^{-1}(-e^-)) = \dotpi(-e^-)) = \imath R^2,\\
\rho^- := \dotrho(c^{-1}(-4e^+)) = \dotpi(-4e^+)) = \imath \Delta.
\end{align*}
Then we have $\rho^+(R^{2k}\ell_1^l) = \imath R^{2k+2}\ell_1^l$ and $\rho^-(R^{2k}\ell_1^l) = 2k({M}+2k-2+2l) R^{2k-2}\ell_1^l$, for $k\in \ds N$ and $l\in \{0,1\}$. This shows that $\rho^+$ allows us to go to superpolynomials of 2 degrees higher while $\rho^-$ allows us to go the other direction for $M\not\in -2\ds N$. Therefore we obtain (3).
\end{proof}

The following is an immediate result of this theorem.

\begin{Cor}\label{Cor F isom F}
The minimal Fock representation is isomorphic to the even Fock representation, i.e.\ $F_{-\frac{1}{2}} \cong \ot F_e$
\end{Cor}

Recall the $\mf{sl}(2)$-triple $\left\lbrace e^-, 2L_{e},e^+ \right\rbrace$ from Section \ref{SSTKK2}. Using the Cayley transform $c$ we obtain another $\mf{sl}(2)$-triple $\mf s :=\left\lbrace f^-, h,f^+ \right\rbrace$\symindex{$\mf s$} where 
\begin{align*}
f^- := c^{-1}(e^-), \quad f^+ &:= c^{-1}(e^+) \quad \mbox{ and } \quad h := c^{-1}(2L_{e}).
\end{align*}\symindex{$f^{\pm}$ and $h$}
We have
\begin{align*}
\dotrho(f^-) &= \dotpi(e^-) = -\imath R^2,\\
\dotrho(h) &= \dotpi(2L_{e}) = -(\dfrac{M}{2}+\ds E),\\
\dotrho(f^+) &= \dotpi(e) = -\dfrac{\imath}{4}\Delta.
\end{align*}
In particular, we have
\begin{align*}
\dotrho(f^-)\phi_l R^{2k} &=  -\imath \phi_lR^{2k+2},\\
\dotrho(h)\phi_l R^{2k} &= -(\dfrac{M}{2}+2k+l)\phi_lR^{2k},\\
\dotrho(f^+)\phi_l R^{2k} &= -\imath(\dfrac{M}{2}+k-1)k\phi_l R^{2k-2},
\end{align*}
for $k,l\in \ds N$ and $\phi_l\in \mc H_{l}(\ds C^{m|2n})$. Therefore, for fixed $\phi_{l}\in \mc H_{l}(\ds C^{m|2n})$ the span of superpolynomials $R^{2k}\phi_{l}$, $k\in \ds N$ is invariant under the action of $\mf s$ and defines an irreducible representation of $\mf s\cong \mf{sl}(2, \ds K)$ of lowest weight $\dfrac{M}{2}+l$. Define $G_{\frac{M}{2}+l} := \sum_{k=0}^{\infty}\ds C R^{2k}$. Then it is clear that $G_{\frac{M}{2}+l}$ and $\mc H_{l}(\ds C^{m|2n})$ are invariant under the action of $\mf s$, for all $l\in \ds N$. From Lemma \ref{Lemsltriple} it also follows that $\mf s$ and $\Inn(J)$ commute. Putting things together gives us the analogue of \cite[Theorem 2.24]{HKMO}.

\begin{theorem}\label{Th_DecF2}
Suppose $M=m-2n\not\in -2\ds N$. Under the action of $(\mf s, \Inn(J))$ we have the decomposition
\begin{align*}
\ot F = \bigoplus_{l=0}^\infty G_{\frac{M}{2}+l} \boxtimes \mc H_{l}(\ds C^{m|2n}),
\end{align*}
where $G_{\frac{M}{2}+l}$ denotes the irreducible representation of $\mf s\cong \mf{sl}(2, \ds K)$ of lowest weight $\dfrac{M}{2}+l$.
\end{theorem}

\subsection{Connection with a Joseph-like ideal}\label{SS_Joseph}

For Lie groups, minimal representations are characterised by the property that their annihilator ideal is the Joseph ideal, see \cite{GanSavin}. In \cite{CSS2}, two Joseph-like ideals were constructed for the orthosymplectic Lie superalgebra. In the orthosymplectic case, it was proven that the annihilator ideal of its minimal representation of $\mf{osp}_{\ds C}(p+q|2m)$ is equal to one of these Joseph-like ideals for $p+q-2m \geq 3$, \cite[Theorem 6.4]{BF}. We will now briefly introduce this other Joseph-like ideal and describe its connection with our minimal representation. 

We use the standard root system for $\mf g_{\ds C}\cong D(2n,m)$, where the simple roots are given by
\begin{align*}
\delta_1-\delta_2, \ldots, \delta_{m-1}-\delta_m, \delta_m-\epsilon_1, \epsilon_1-\epsilon_2, \ldots \epsilon_{2n-1}-\epsilon_{2n}, \epsilon_{2n-1}+\epsilon_{2n}.
\end{align*}
By \cite[Theorem 3.1 and Theorem 3.6]{CSS2}, the tensor product $\mf g_{\ds C}\otimes\mf g_{\ds C}$ contains a decomposition summand isomorphic to the simple $\mf g_{\ds C}$-module of highest weight $4\delta_1$. This decomposition factor is called the \textbf{Cartan product} of $\mf g_{\ds C}$ and is denoted by $\mf g_{\ds C} \circledcirc \mf g_{\ds C}$. For $X,Y\in \mf g_{\ds C}$ the projection of $X\otimes Y$ on $\mf g_{\ds C} \circledcirc \mf g_{\ds C}$ is denoted by $X\circledcirc Y$.\symindex{$\circledcirc$} Denote the Tensor algebra of $\mf g_{\ds C}$ by
\begin{align*}
T(\mf g_{\ds C}) := \bigoplus_{j\geq 0}\mf g_{\ds C}^{\otimes j}
\end{align*}\symindex{$T( \cdot )$}
and let $B(\cdot\, , \cdot)$\symindex{$B(\cdot\, , \cdot)$} denote the renormalisation of the Killing form of $\mf g_{\ds C}$ defined by \cite[Lemma 2.2]{CSS2}. Note that this renormalised Killing form also defines a non-degenerate form on $\mf g_{\ds C}$ when the Killing form of $\mf g_{\ds C}$ is zero, i.e.\ when $\mf g_{\ds C}\cong \mf{spo}_{\ds C}(2m|2m+2)$, \cite[Proposition 2.4.1]{K2}. We have a one-parameter family $\{\mc J_{\mu}|\mu \in \ds C \}$ of quadratic two-sided ideals in $T(\mf g_{\ds C})$, where $\mc J_\mu$ is generated by
\begin{align*}
\left\{ X\otimes Y - X\circledcirc Y - \frac{1}{2}[X,Y] - \mu B(X,Y)|X,Y\in \mf g_{\ds C} \right\}.
\end{align*}\symindex{$\mc J_\mu$, $\mf J_\mu$ and $J$}
By construction there is a unique ideal $\mf J_\mu$ in the universal enveloping algebra $U(\mf g_{\ds C})$, which satisfies $T(\mf g_{\ds C})/\mc J_{\mu} \cong U(\mf g_{\ds C})/\mf J_\mu$. The ideal $\mf J_\mu$ has finite codimension if $\mu \neq \frac{1}{4}$ and infinite codimension if $\mu = \frac{1}{4}$, \cite[Theorem 6.3]{CSS2}.

\begin{Def}
The \textbf{Joseph-like ideal} of $\mf g_{\ds C}$ (with respect to the standard root system) is given by $\mf J:= \mf J_{\mu}$ for the critical value $\mu = \frac{1}{4}$
\end{Def}
Let us now relate our minimal representation to a representation of $\mf g_{\ds C}$ studied in \cite{C2}. 

\begin{Def}\symindex{$\Lambda\big(\ds K^{d\vert m}\big)$}
The \textbf{Grassmann superalgebra} over $\ds K$ is defined as
\begin{gather*}
\Lambda\big(\ds K^{d|m}\big):=\Lambda\big(\ds K^{d}\big)\otimes_\C\mathcal P \big(\ds K^{m}\big),
\end{gather*}
The variables of $\Lambda(\ds K^{d})$ and $\mc P(\ds K^{m})$ are now the even and odd variables, respectively.
\end{Def}
Let $\theta = (\theta_i)_{i=1}^{2n+m}$ denote the variables of $\Lambda\big(\ds K^{2n|m}\big)$, then they satisfy the commutation relations
\begin{gather*}
\theta_i\theta_j = -(-1)^{|\theta_i||\theta_j|}\theta_j \theta_i,
\end{gather*}
for $i,j\in \{1, \ldots, m+2n\}$.

In \cite[Definition 5]{C2} a realisation of $\mf g_{\ds C}$ on $\Lambda\big(\ds K^{2n|m}\big)$ was defined. It is generated by the operators $\theta_i\theta_j$ and $\pt{\theta_i}\pt{\theta_j}$. By \cite[Remark 3]{C2}, there is a corresponding polynomial realisation on $\mc P(\ds K^{m|2n})$ generated by $\ell_{i}\ell_{j}$ and $\pt{\ell_i}\pt{\ell_j}$. Let $\dotrho_{\ds C}$ denote the unique representation of $\mf g_{\ds C}$ obtained as the $\ds C$-linear extension of $\dotrho$. Then $\dotrho_{\ds C}$ is also generated by $\ell_{i}\ell_{j}$ and $\pt{\ell_i}\pt{\ell_j}$, i.e.\ the realisation on $\mc P(\ds K^{m|2n})$ defined by \cite[Definition 5 and Remark 3]{C2} is, up to isomorphisms, the same as the polynomial realisation $\dotrho_{\ds C}$ on $\ot F\cong \mc P(\ds K^{m|2n})$.

\begin{Def}
The \textbf{annihilator ideal} of a representation $(\dnu, V)$ of $\mf g_{\ds C}$ is the ideal in $U(\mf g_{\ds C})$ given by
\begin{align*}
\Ann(\dnu, V) := \{X \in U(\mf g_{\ds C}) | \dnu (X)v = 0 \text{ for al }v\in V\}.
\end{align*}
\end{Def}\symindex{$\Ann$}

From \cite[Theorem 6.7]{CSS2}, we now obtain the following result.

\begin{theorem}\label{Th_Joseph}
If $M=m-2n\not\in\{-1,0\}$, then the annihilator ideal of even and odd Fock representations is the Joseph ideal, i.e.
\begin{align*}
\Ann(\dotrho_{\ds C}, \ot F) = \Ann(\dotrho_{\ds C}, \ot F_{e}) = \Ann(\dotrho_{\ds C}, \ot F_o) = \mf J.
\end{align*}
\end{theorem}

\begin{proof}
\cite[Theorem 6.7]{CSS2} states this result with respect to the representation of $\mf g_{\ds C}$ on $\Lambda\big(\ds K^{2n|m}\big)$ from \cite{C2}. However, The results and proofs of \cite[Theorem 6.6]{CSS2} and \cite[Theorem 6.7]{CSS2} still hold when working with $\dotrho$ instead, mutatis mutandis.
\end{proof}

\subsection{The Gelfand-Kirillov dimension}
The Gelfand-Kirillov dimension is a measure of the size of a representation that roughly measures how fast a representation grows to infinity. In particular, the  Gelfand-Kirillov dimension is zero for finite-dimensional representations. Classically, i.e.\ for non-super Lie algebras, minimal representations have the property that they attain the lowest possible Gelfand-Kirillov dimension of all infinite-dimensional representations \cite{GanSavin}. 

Let $R$ be a finitely generated algebra, then the Gelfand-Kirillov dimension of a finitely generated $R$-module $F$ is defined by
\begin{align*}
GK(F) = \lim\sup_{k\rightarrow \infty} \left(\log_k\dim(V^kF_0)\right).
\end{align*}\symindex{$GK( \cdot )$}
Here $V$ is a finite-dimensional subspace of $R$ which contains the unit element $1$ and generators of $R$, and $F_0$ is a finite-dimensional subspace of $F$, which generates $F$ as an $R$-module. The definition is independent of the chosen $V$ and $F_0$, see \cite[Section 7.3]{Mu}.

In the symplectic case (i.e\ the $n=0$ case) we have $GK(\ot F_e) = m$, see e.g. \cite[Theorem 1.4 (1)]{Nishiyama}. We will now prove that this result still holds regardless of $n\in \ds N$.

\begin{theorem} \label{Th Gelfand-Kirillov dimension}
The Gelfand-Kirillov dimension of the $U(\mf g)$-module $\ot F_e$ is given by $GK(\ot F_e) = m$.
\end{theorem}

\begin{proof}
We choose $\ds C$ for $F_0$ and $\mf g \oplus 1\subset U(\mf g)$ for $V$. Then $V^k = U_k(\mf g)$ is the canonical filtration on the universal enveloping algebra.
We have
\begin{align*}
\dim\left(U_k(\mf g) F_{\lambda,0}\right) &= \dim\left(\bigoplus_{j=0}^k \mc P_{2j}(\ds C^{m|2n})\right) = \sum_{j=0}^k \dim(\mc P_{2j}(\ds C^{m|2n})),
\end{align*}
with
\begin{gather*}
\dim (\mc P_{2j}(\ds C^{m|2n})) = \sum_{i=0}^{\min(2j,2n)}\binom{2n}{i}\binom{2j-i+m-1}{m-1}.
\end{gather*}
By \cite[Lemma 7.3.1]{Mu}, it is sufficient to know the highest exponent of $k$ in the expression for $\dim\left(U_k(\mf g) \ds C\right)$ to calculate $\lim\sup_{k\rightarrow \infty} \left(\log_k\dim\left(U_k(\mf g) \ds C\right)\right)$. We may also assume $k \gg n$ and therefore
\begin{align*}
\dim\left(U_k(\mf g) F_{\lambda,0}\right) &= \sum_{i=0}^{2n}\binom{2n}{i}\sum_{j=n}^k\binom{2j-i+m-1}{m-1} + (\text{lower order terms in }k),
\end{align*}
which approaches
\begin{align*}
\sum_{i=0}^{2n}\binom{2n}{i}k\binom{2k-i+m-1}{m-1}
\end{align*}
as $k$ increases. The highest exponent of $k$ in $\binom{2k-i+m-1}{m-1}$ is $m-1$. Therefore we conclude $GK(\ot F_e) = m-1+1=m$.
\end{proof}
%

In the orthosymplectic case the Gelfand-Kirillov dimension of its minimal representation was calculated in \cite[Section 7]{BF}. For $\mf{osp}(p,q|2m) = \mf{osp}(2n,2n|2m)\cong \mf g$ it is equal to $4n-3$. This implies that for $m\neq 4n-3$ the minimal representation of $\mf g$ constructed in the $\lambda=-\frac{1}{2}$ case is not equivalent to the minimal representation constructed in the orthosymplectic case. Moreover, contrary to the classical case, it follows that minimal representations do not necessarily have minimal nonzero Gelfand-Kirillov dimension. Indeed, if $m\neq 0\neq n$ and $m\neq 4n-3$, then either the orthosymplectic case or the $\lambda=-\frac{1}{2}$ case does not have minimal Gelfand-Kirillov dimension.

\section{Hermitian superspaces and super-inner products}\label{S_Herm_superspace}

In \cite{dGM} a non-degenerate, superhermitian, sesquilinear form is called a \textbf{superhermitian inner product} (we will abbreviate this to \textbf{super-inner product}) and a supervector space endowed with a super-inner product is then called a \textbf{Hermitian superspace}. In this section, we will construct super-inner products on some of the supervector spaces we defined previously, which turns them into Hermitian superspaces.

We still only consider the case $\lambda=-\frac{1}{2}$, except in Section \ref{ss BF prod}, where the $\lambda=1$ and $\lambda=-\frac{1}{2}$ cases can be treated simultaneously.

\subsection{The $L^2$-product on $\ot W$}

Let $x = (x_i)_{i=1}^{m+2n}$ denote the variables of $\mc P(\ds R^{m|2n})$. The first product we introduce is the well-known $L^2$-product, generalised to superspace. This product has been studied in the super case before, see e.g.\ \cite{BdGT}.

\begin{Def}
The integral over the superspace $\ds R^{m|2n}$ is defined by
\begin{align*}
\int_{\ds R^{m|2n}} dx := \int_{\ds R^{m}} dx_{\ol 0}\int_{B},
\end{align*}
where $x_{\ol 0} := (x_i)_{i=1}^m$ denotes the even variables of $x$ and
\begin{align*}
\int_{B} := \pi^{-n}\pt{x_{m+2n}}\pt{x_{m+2n-1}}\ldots \pt{x_{m+1}}
\end{align*}
is the Berezin integral on $\Lambda\big(\ds R^{2n}\big)$, see e.g.\ \cite{Le}.
\end{Def}

The standard definition of the $L^2$-product is given by
\begin{align*}
\left<f,g\right>_{L^2} := \int_{\ds R^{m|2n}}f(x)\ol{g(x)}dx,
\end{align*}\symindex{$\left<\cdot \, ,\cdot \right>_{L^2}$}
for superfunctions $f$ and $g$. However, we will use the following renormalisation.

\begin{Def}
The \textbf{$L^2$-product} on $\ot W$ is defined as
\begin{align*}
\ipW{f,g} := \dfrac{1}{\omega}\int_{\ds R^{m|2n}}f(x)\ol{g(x)}dx, \quad \text{ with } \quad \omega := \int_{\ds R^{m|2n}}\exp(-2R^2)dx,
\end{align*}
for all $f,g \in \ot W$.
\end{Def}\symindex{$\ipW{\cdot \, ,\cdot }$}

\begin{lemma}\symindex{$\omega$}
We have $\omega = 2^n \left(\dfrac{\pi}{2}\right)^{\frac{M}{2}}$, with $M=m-2n$. In particular, the $L^2$-product is well-defined on $\ot W$.
\end{lemma}

\begin{proof}
We have
\begin{align*}
\omega &= \int_{\ds R^{m|2n}}\exp(-2R^2)dx\\
&= \int_{\ds R^m}\exp(-2R_{0}^2)dx_{\ol 0}\int_B\exp(-2R_{1}^2)\\
&= 2^{-\frac{m}{2}}\pi^{\frac{m}{2}}\int_B\exp(-2R_{1}^2),
\end{align*}
where $R_{0}^2$ and $R_{1}^2$ denote the terms of $R^2$ consisting of even and odd variables, respectively. For the Berezin integral we only need to know the coefficient in front of $x_{m+1}\ldots x_{m+2n}$ in the expansion of $\exp(-2R_{1}^2)$. We find
\begin{align*}
\int_B\exp(-2R_{1}^2) &= \int_B\sum_{k=0}^\infty \dfrac{(-1)^k 2^k}{k!}\left(\sum_{i=m+1}^{m+2n}x^ix_i \right)^k\\
&= \dfrac{(-1)^n 2^n}{n!}\int_B \left(\sum_{i=m+1}^{m+2n}x^ix_i \right)^n\\
&= \dfrac{(-1)^n 2^n 2^n}{n!}\int_B \left(\sum_{i=m+1}^{m+n}x_ix_{i+n} \right)^n\\
&= 4^n\int_B (-1)^n  x_{m+1}x_{m+1+n}\ldots x_{m+n}x_{m+2n}\\
&= 4^n\int_B x_{m+1}x_{m+2}\ldots x_{m+2n-1}x_{m+2n}\\
&= 4^n\pi^{-n},
\end{align*}
as desired.
\end{proof}

We obtain the following result from the study of this product in \cite{BdGT}.

\begin{Prop}\label{Prop L2 prod}
The $L^2$-product $\ipW{\cdot\, , \cdot}$ defines a superhermitian, non-degenerate sesquilinear form on $\ot W$. In particular, $(\ot W,\ipW{\cdot\, , \cdot})$ is a Hermitian superspace
\end{Prop}

The following proposition shows that $\dotpi$ is ``infinitesimally superunitary'' on $(\ot W,\ipW{\cdot\, , \cdot})$.

\begin{Prop}\label{Prop pi skew sup sym}
The Schr\"odinger representation $\dotpi$ is skew-supersymmetric with respect to the $L^2$-product, i.e.\
\begin{align*}
\ipW{\ot \pi(X)f,g} = - (-1)^{|X||f|} \ipW{f,\ot \pi(X)g},
\end{align*}
for all $X\in \mf g$ and $f,g\in \ot W$.
\end{Prop}
\begin{proof}
Recall that the action on $\ot W$ is given in Theorem \ref{Th pil on psiW}.
\begin{itemize}
\item For $\ot \pi(\ell_{ij}^-)$ we find
\begin{align*}
\ipW{\ot \pi(\ell_{ij}^-)f,g} &=   \dfrac{1}{\omega}\int_{\ds R^{m|2n}}(-2\imath x_ix_j)f(x)\ol{g(x)}dx\\
 &= -(-1)^{(|i|+|j|)|f|} \dfrac{1}{\omega}\int_{\ds R^{m|2n}}f(x)\ol{(-2\imath x_ix_j)g(x)}dx\\
 &= -(-1)^{(|i|+|j|)|f|}\ipW{f,\ot \pi(\ell_{ij})g}.
\end{align*}
\item For $\dotpi(\ell_{ij}^+)$ we find
\begin{align*}
\ipW{\dotpi(\ell_{ij}^+)f,g} &= -\dfrac{\imath}{2\omega} \int_{\ds R^{m|2n}}\pt{i}\pt{j}(f(x))\ol{g(x)}dx\\
&= (-1)^{|i|(|j|+|f|)}\dfrac{\imath}{2\omega} \int_{\ds R^{m|2n}}\pt{j}(f(x))\ol{\pt{i}(g(x))}dx\\
&= -(-1)^{(|i|+|j|)|f|}\dfrac{\imath}{2\omega} \int_{\ds R^{m|2n}}f(x)\ol{\pt{i}\pt{j}(g(x))}dx\\
&= -(-1)^{(|i|+|j|)|f|}\ipW{f,\dotpi(\ell_{ij}^+)g}.
\end{align*}
The second and third steps are justified by the fact that the terms of elements in $\ot W$ are exponentials in $-R_x^2$ times superpolynomials, i.e.\ $\ot W$ consists of rapidly decreasing superfunctions.
\item For $\dotpi(2L_{\ell_{ij}})$ and $\dotpi(4[L_{\ell_{ij}}, L_{\ell_{rs}}])$ we first need
\begin{align*}
\ipW{x_i\pt j f,g} &=  \dfrac{1}{\omega}\int_{\ds R^{m|2n}} x_i\pt j(f(x))\ol{g(x)}dx\\
&= -(-1)^{(|i|+|j|)|f|+|i||j|} \dfrac{1}{\omega}\int_{\ds R^{m|2n}} (f(x))\ol{\pt j (x_ig(x))}dx\\
&= -(-1)^{(|i|+|j|)|f|+|i||j|}\ipW{ f,\pt jx_ig}\\
&= -(-1)^{(|i|+|j|)|f|}\ipW{ f,x_i\pt j g} -(-1)^{(|i|+|j|)|f|}\beta_{ij}\ipW{ f,g}\\
&= -(-1)^{(|i|+|j|)|f|}\ipW{ f,x_i\pt j g} -\beta_{ij}\ipW{f,g},
\end{align*}
where we used $|i|=|j|$ if $\beta_{ij}\neq 0$ in the last step. Multiplying both sides with $(-1)^{|i||j|}$ and switching the roles of $i$ and $j$ gives us
\begin{align*}
(-1)^{|i||j|}\ipW{x_j\pt i f,g} = -(-1)^{(|i|+|j|)|f|+|i||j|}\ipW{ f,x_j\pt i g} -\beta_{ij}\ipW{f,g}.
\end{align*}
If we subtract both equations we get
\begin{align*}
\ipW{L_{ij}f,g} = - (-1)^{(|i|+|j|)|f|} \ipW{f,L_{ij} g},
\end{align*}
which implies
\begin{align*}
\ipW{\dotpi(4[L_{\ell_{ij}}, L_{\ell_{rs}}])f,g} = - (-1)^{(|i|+|j|+|r|+|s|)|f|} \ipW{f,\dotpi(4[L_{\ell_{ij}}, L_{\ell_{rs}}]) g}.
\end{align*}
If we add both equations we get
\begin{align*}
\ipW{(x_i\pt j+ (-1)^{|i||j|} x_j\pt i)f,g} &= -(-1)^{(|i|+|j|)|f|}\ipW{f,(x_i\pt j+ (-1)^{|i||j|} x_j\pt i)g}\\
&\quad -2\beta_{ij}\ipW{f,g},
\end{align*}
which implies
\begin{align*}
\ipW{\dotpi(2L_{\ell_{ij}})f,g} = - (-1)^{(|i|+|j|)|f|} \ipW{f,\dotpi(2L_{\ell_{ij}}) g},
\end{align*}
as desired. \qedhere
\end{itemize}
\end{proof}

\subsection{The Schr\"odinger product on $W_{-\frac{1}{2}}$}

If we combine the $L^2$-product $\ipW{\cdot \, , \cdot}$ on $\ot W$ with the folding morphism $\psi_{\ds R}$, then we can define a product on $W_{-\frac{1}{2}}$ as well.

\begin{Def}
The \textbf{Schr\"odinger product} on $W_{-\frac{1}{2}}$ is defined as
\begin{align*}
\ipO{f,g} := \ipW{\psi_{\ds R}(f), \psi_{\ds R}(g)},
\end{align*}
for all $f,g \in W_{-\frac{1}{2}}$.
\end{Def}\symindex{$\ipO{\cdot \, , \cdot}$}

The $\mc O$\symindex{$\mc O$} used in the subindex of the Schr\"odinger product is merely a symbolic way to indicate a connection with minimal orbits. We do not define this minimal orbit explicitly.

We immediately obtain the following from Propositions \ref{Prop L2 prod} and \ref{Prop pi skew sup sym}.

\begin{Cor}
The Schr\"odinger product $\ipO{\cdot\, , \cdot}$ defines a superhermitian, non-degenerate sesquilinear form on $W_{-\frac{1}{2}}$, for which $\pil$ is skew-supersymmetric. In particular, $(W_{-\frac{1}{2}},\ipO{\cdot\, , \cdot})$ is a Hermitian superspace
\end{Cor}

\subsection{The Bessel-Fischer product on $F_{\lambda}$}\label{ss BF prod}

In \cite[Section 2.3]{HKMO} an inner product on the polynomial space $\mc P(\C^m)$ was introduced, namely the Bessel-Fischer inner product 
\begin{align*}
\bfip{p,q} := \left. p(\bessel)\bar q(z)\right|_{z=0}.
\end{align*}
Here $p(\bessel)$ is obtained by replacing $z_i$ by $\bessel (z_i)$ and $\bar q(z) = \overline{q(\bar z)}$ is obtained by conjugating the coefficients of the polynomial $q$. In the classical setting, it was proven that the Bessel-Fischer inner product is equal to the $L^2$-inner product on the Fock space \cite[Proposition 2.6]{HKMO}. In the orthosymplectic and $D(2,1;\alpha)$ cases this product was used as the starting point to generalise the Fock space to superspace.

We define the Bessel-Fischer product on a superpolynomial space as follows.
\begin{Def}
For $p, q\in \mathcal P(\ds K^{{\wh m}|2{\wh n}})$ we define the \textbf{Bessel-Fischer product} of $p$ and $q$ as
\begin{align*}
\bfip{p,q} := \left. p(\bessel)\bar q(\ell_{ij})\right|_{z=0},
\end{align*}\symindex{$\bfip{\cdot \, , \cdot}$}
where $\bar q(z) = \overline{q(\bar z)}$ is obtained by conjugating the coefficients of the polynomial $q$ and $p(\bessel)$ is obtained by replacing the occurences of $\ell_{ij}$ in $p(z)$ with $\bessel(\ell_{ij})$ for all $i,j\in \{1, \ldots, m+2n\}$.
\end{Def}

In the orthosymplectic and $D(2,1;\alpha)$ cases the Bessel-Fischer product is a non-degenerate superhermitian sesquilinear form when restricted to the Fock space $F_\lambda$. Moreover, the Fock representation $\rol$ is skew-supersymmetric with respect to the Bessel-Fischer product.

For the classical setting a reproducing kernel for the Fock space was constructed in Section 2.4 of \cite{HKMO}. A generalisation of this reproducing kernel in superspace was constructed for both the orthosymplectic and $D(2,1;\alpha)$ cases. Similarly, we can construct a ``reproducing kernel'' for $\mf g$. The non-degeneracy of the Bessel-Fischer product then follows from the existence of this reproducing kernel.

Recall $\ell^{ij} = \sum_{k,l=1}^{m+2n}\ell_{kl}\beta^{ki}\beta^{lj}$ for all $i,j\in \{1, \ldots, m+2n\}$. Let $z = (z_{ij})_{ij}$ denote the variables of $\mathcal P(\ds C^{{\wh m}|2{\wh n}})$ corresponding to $(l_{ij})_{ij}$ and let $w = (w_{ij})_{ij}$ be a copy. We define
\begin{align*}
z|w := \dfrac{1}{2}\sum_{i<j}z^{ji}w_{ij}+ \dfrac{1}{4}\sum_{i=1}^m z^{ii}w_{ii} = \dfrac{1}{4}\sum_{i,j=1}^{m+2n}z^{ji}w_{ij} .
\end{align*}\symindex{$\cdot \vert \cdot$}

Note that for $i\leq j$ we have
\begin{align*}
\bessel(z_{ij})z^{kl} &= -2\lambda\sum_{r,s}(1+\delta_{rs})\beta_{jr}\beta_{is}\pt{z_{rs}}z^{kl} = -2\lambda\sum_{r,s, a, b}(1+\delta_{rs})\beta_{jr}\beta_{is}\beta^{ak}\beta^{bl}\pt{z_{rs}}z_{ab}\\
&= -2\lambda\sum_{r,s, a, b}(\delta_{ra}\delta_{sb}+(-1)^{|r||s|}\delta_{rb}\delta_{sa})\beta_{jr}\beta_{is}\beta^{ak}\beta^{bl}\\
&= -2\lambda\sum_{r,s}\beta_{jr}\beta_{is}\beta^{rk}\beta^{sl}-2\lambda\sum_{r,s}(-1)^{|r||s|}\beta_{jr}\beta_{is}\beta^{sk}\beta^{rl}\\
&= -2\lambda(\delta_{jk}\delta_{il}+(-1)^{|i||j|}\delta_{ik}\delta_{jl})
\end{align*}
and therefore
\begin{align*}
\bessel(z_{ij})(z|w) &= -\lambda w_{ij}.
\end{align*}

\begin{lemma}\label{Lemma rep kernel}
Define the superfunction $\ds I_{\lambda,k}(z,w)$, with $\lambda\in\{1,-1/2\}$, by
\begin{align*}
\ds I_{-1/2,k}(z,w) &:= \dfrac{(-1)^k}{k!}\left(\dfrac{1}{2}-k\right)_k^{-1}(z|\overline{w})^k,\\
\ds I_{1,k}(z,w) &:= \dfrac{2^k}{k!}\left(-1-k\right)_k^{-1}(z|\overline{w})^k,
\end{align*}
where we used the Pochhammer symbol $(a)_k = a(a+1)(a+2)\ldots (a+k-1)$. For all $p\in F_\lambda$ of degree $k$ we have
\begin{align*}
\bfip{p(z),\ds I_{\lambda,k}(z,w)} = p(w) \mod  \mc I_\lambda.
\end{align*}
\end{lemma}

\begin{proof}
We calculate $\bessel(z_{ij})(z|w)^k$ for all $i,j$. We have
\begin{align*}
-2\lambda\sum_{h,l}(1+\delta_{hl})\beta_{jh}\beta_{il}\pt{z_{hl}}(z|w)^k = -\lambda k (z|w)^{k-1}w_{ij}
\end{align*}
and
\begin{align*}
&\sum_{h,l,r,s}(-1)^{|h||i|}(1+\delta_{hl}+\delta_{rs}+\delta_{hl}\delta_{rs})\beta_{is}\beta_{jl}z_{hr}\pt{z_{sr}}\pt{z_{lh}}(z|w)^k\\
&\quad = k(k-1)(z|w)^{k-2}\dfrac{1}{4}\sum_{h,l,r,s,t,u,v,y}(-1)^{|h||i|}\beta_{is}\beta_{jl}\beta^{su}\beta^{rt}\beta^{ly}\beta^{hv}z_{hr}w_{tu}w_{vy}\\
&\quad = k(k-1)(z|w)^{k-2}\dfrac{1}{4}\sum_{h,r,t,v}(-1)^{|v||i|}\beta^{rt}\beta^{hv}z_{hr}w_{ti}w_{vj}\\
&\quad = k(k-1)(z|w)^{k-2}\dfrac{1}{4}\sum_{t,v}(-1)^{|v||i|}z^{vt}w_{ti}w_{vj}.
\end{align*}
First suppose $\lambda=-1/2$, then we are working modulo $\mc I_{-\frac{1}{2}}$ and therefore we have $(-1)^{|v||i|}w_{ti}w_{vj} = w_{tv}w_{ij}$. This implies
\begin{align*}
&\sum_{h,l,r,s}(-1)^{|h||i|}(1+\delta_{hl}+\delta_{rs}+\delta_{hl}\delta_{rs})\beta_{is}\beta_{jl}z_{hr}\pt{z_{sr}}\pt{z_{lh}}(z|w)^k\\
&\quad = k(k-1)(z|w)^{k-2}\left(\dfrac{1}{4}\sum_{t,v}z^{vt}w_{tv}\right)w_{ij}\\
&\quad = k(k-1)(z|w)^{k-1}w_{ij},
\end{align*}
which gives us
\begin{align*}
\bessel(z_{ij})(z|w)^k = -k(\lambda -k+1)(z|w)^{k-1}w_{ij} = -k(\frac{1}{2}-k)(z|w)^{k-1}w_{ij}
\end{align*}
and therefore
\begin{align*}
\bfip{p(z), (z|\ol w)^{k}} = (-1)^k k! \left(\frac{1}{2}-k\right)_k p(w).
\end{align*}
Now suppose $\lambda=1$. In this case we are working modulo $\mc I_{1}$ which implies
\begin{align*}
w_{ti}w_{vj} = - (-1)^{|i||v|}w_{tv}w_{ij} - (-1)^{(|i|+|v|)|j|}w_{tj}w_{iv}.
\end{align*}
This gives us
\begin{align*}
\sum_{t,v}(-1)^{|v||i|}z^{vt}w_{ti}w_{vj} &= -(z|w)w_{ij}-\sum_{t,v}(-1)^{|i||j|+ |v||j|}z^{vt}w_{tj}w_{vi}\\
&= -(z|w)w_{ij}-\sum_{t,v}(-1)^{|i||j|+ |t||j|}z^{tv}w_{vj}w_{ti}\\
&= -(z|w)w_{ij}-\sum_{t,v}(-1)^{|v||i|}z^{vt}w_{ti}w_{vj},
\end{align*}
which implies
\begin{align*}
\sum_{t,v}(-1)^{|v||i|}z^{vt}w_{ti}w_{vj} = -\dfrac{1}{2}(z|w)w_{ij}
\end{align*}
and then
\begin{align*}
&\sum_{h,l,r,s}(-1)^{|h||i|}(1+\delta_{hl}+\delta_{rs}+\delta_{hl}\delta_{rs})\beta_{is}\beta_{jl}z_{hr}\pt{z_{sr}}\pt{z_{lh}}(z|w)^k\\
&\quad = -\dfrac{1}{2}k(k-1)(z|w)^{k-1}w_{ij}.
\end{align*}
This gives us
\begin{align*}
\bessel(z_{ij})(z|w)^k = -\dfrac{k}{2}(2\lambda +k-1)(z|w)^{k-1}w_{ij} = \dfrac{k}{2}(-1-k)(z|w)^{k-1}w_{ij}
\end{align*}
and therefore
\begin{align*}
\bfip{p(z), (z|\ol w)^{k}} = \dfrac{k!}{2^k} \left(-1-k\right)_k p(w),
\end{align*}
from which the lemma follows.
\qedhere
\end{proof}

We will give a closed formula of the reproducing kernel in terms of the renormalised I-Bessel function. The I-Bessel function $I_\gamma(t)$ (or modified Bessel function of the first kind) is defined by
\begin{align*}
I_\gamma(t) := \left(\dfrac{t}{2}\right)^{\gamma}\sum_{k=0}^\infty \dfrac{1}{k!\Gamma(k+\gamma+1)}\left(\dfrac{t}{2}\right)^{2k},
\end{align*}
for $\gamma, t\in\C$, see \cite{AAR}, Section 4.12. We will use the renormalisation
\begin{align*}
\widetilde I_\gamma(t) := \left(\dfrac{t}{2}\right)^{-\gamma} I_\gamma(t).
\end{align*}\symindex{$I_\gamma$ and $\widetilde I_\gamma$}

\begin{theorem}[Reproducing kernel of $F_\lambda$]\label{Theorem repr kernel}
Define the superfunction $\ds I_\lambda(z,w)$ by
\begin{align*}
\ds I_{-\frac{1}{2}}(z,w) &:= \sqrt{\pi}\widetilde I_{-\frac{1}{2}}\left(2\sqrt{(z|\overline w)}\right) = \cosh\left(2\sqrt{(z|\overline w)}\right),\\
\ds I_1(z,w) &:= \widetilde I_{1}\left(2\sqrt{-2(z|\overline w)}\right).
\end{align*}
For all $p\in F_\lambda$ we have
\begin{align*}
\bfip{p(z),\ds I_\lambda(z,w)} = p(w).
\end{align*}
\end{theorem}\symindex{$\ds{I}_\lambda(\cdot \, ,\cdot)$}

\begin{proof}
Note that 
\begin{align*}
\sqrt{\pi}\widetilde I_{-\frac{1}{2}}\left(2\sqrt{(z|\overline w)}\right) &= \sum_{k=0}^\infty \dfrac{1}{k!}\dfrac{\Gamma(\frac{1}{2})}{\Gamma(k+\frac{1}{2})}(z|\overline{w})^k  \\
&=  \sum_{k=0}^\infty \dfrac{(-1)^k}{k!}\left(-\frac{1}{2}-k+1\right)_k^{-1}(z|\overline{w})^k \\
&=\sum_{k=0}^\infty \ds I_{-\frac{1}{2},k}(z,w), 
\end{align*}
and similarly
\begin{align*}
\widetilde I_{1}\left(2\sqrt{-2(z|\overline w)}\right)=\sum_{k=0}^\infty \ds I_{1,k}(z,w) .
\end{align*}
The proposition then follows from Lemma \ref{Lemma rep kernel} and the orthogonality property.
\end{proof}

\begin{Prop}[non-degeneracy]\label{Prop Non deg bfip}
The Bessel-Fischer product is non-degenerate on $F_\lambda$, i.e.\ if $\bfip{p,q}=0$, for all $q\in F_\lambda$, then $p=0$.
\end{Prop}

\begin{proof}
Suppose $p\in F_\lambda$ is such that $\bfip{p,q}=0$, for all $q\in F_\lambda$. Using the reproducing we obtain $p(w) = \bfip{p(z),\ds I_\lambda(z,w)} = 0$. Hence $p=0$.
\end{proof}

To show that $(F_{\lambda},\bfip{\cdot\, , \cdot})$ is a Hermitian superspace, we still need to prove that the Bessel-Fischer product is superhermitian. In the orthosymplectic case this is proven in \cite[Proposition 4.7]{BCD} using long and technical calculations. Similar long and technical calculations could potentially be used to prove the superhermitianity in our cases. However, the $\lambda =-\frac{1}{2}$ case will follow immediately from Section \ref{SS_FockProduct}, while the $\lambda=1$ case will not be discussed in depth enough for us to need this property.

\subsection{The Fock product on $\mc P(\ds C^{m|2n})$}\label{SS_FockProduct}

Let $(z_i)_{i=1}^{m+2n}$ be the variables of $\mc P(\ds C^{m|2n}) = \ot F$.

\begin{Def}
The \textbf{Fischer product} on $\mc P(\ds C^{m|2n})$ is defined as
\begin{align*}
\fip{p,q} := \left. p(\partial)\bar q(z)\right|_{z=0},
\end{align*}\symindex{$\fip{\cdot \, ,\cdot}$}
for all $p,q\in \mc P(\ds C^{m|2n})$. Here $p(\partial)$ is obtained by replacing $z_i$ by $\pt i$ ($= \pt{z^i}$) and $\bar q(z) = \overline{q(\bar z)}$ is obtained by conjugating the coefficients of the polynomial $q$.
\end{Def}

We have shown that $\psi_{\ds K}\circ\bessel(\ell_{ij})\circ \psi^{-1}_{\ds K} = \pt i \pt j$ in Theorem \ref{Th pil on psiW}. This implies that the Bessel-Fischer product on $F_{-\frac{1}{2}}$ corresponds with the Fischer product on $\ot F_e$, i.e.\
\begin{align}\label{Eq_Bessel-Fock}
\bfip{p,q} = \fip{\psi(p),\psi(q)},
\end{align}
for $p,q\in F_{-\frac{1}{2}}$.

In the symplectic case the Fischer product $\fip{p,q}$, for $p,q\in\mc P(\ds C^{m})$, is equal to the integral form
\begin{align*}
\dfrac{1}{\gamma}\int_{\ds C^{m}} \exp(-\norm{z}^2)p(z)\ol{q(z)}dz,\quad \gamma := \int_{\ds C^{m}} \exp(-\norm{z}^2)dz = \pi^m.
\end{align*}
See e.g.\ \cite[Section 5]{PSS}. We now wish to generalise this result. Note that we can view complex conjugates of the odd variables $z_{\ol 1} = ({z_i})_{i=m+1}^{m+2n}$ as an added set of odd variables $\ol {z_{\ol 1}} = (\ol{z_i})_{i=m+1}^{m+2n}$, i.e.\ we have
\begin{align*}
z_i \ol{z_j} = (-1)^{|i||j|}\ol{z_j}z_i, && \ol{z_i}\,\ol{z_j} = (-1)^{|i||j|}\ol{z_j}\,\ol{z_i},
\end{align*}
for all $i,j\in\{1, \ldots, m+2n\}$.

\begin{Def}\label{Def_trace_product}
Let $z = ({z_i})_{i=0}^{m+2n}$ and $w = ({w_i})_{i=0}^{m+2n}$ denote the variables of two, possibly equal, instances of $\mc P(\ds K^{m|2n})$ which supercommute, i.e.\ $z_iw_j = (-1)^{|i||j|}w_jz_i$. Then, we define the \textbf{trace product} of $z$ and $w$ as
\begin{align*}
z\bullet w := \sum_{i=1}^{m+2n} z^i w_i.
\end{align*}
\end{Def}\symindex{$\cdot\bullet\cdot$}

Note that $R^2 = z\bullet z$. We now define the square of the norm of $z$ as the superpolynomial
\begin{align*}
\norm{z}^2 := z\bullet \ol z = \sum_{i=1}^{m+2n}z^i \ol{z_i}.
\end{align*}\symindex{$\norm{\cdot}$}
It generalises (the square of) the norm of a multidimensional complex variable, with respect to the orthosymplectic metric induced by $\beta$. We have $\norm{z}^2 = \norm{z_{\ol 0}}^2 + \norm{z_{\ol 1}}^2$, where $\norm{z_{\ol 0}}^2$ and $\norm{z_{\ol 1}}^2$ denote the terms in $\norm{z}^2$ consisting of even and odd variables, respectively.

\begin{Def}
The integral over the superspace $\ds C^{m|2n}$ is defined by
\begin{align*}
\int_{\ds C^{m|2n}} dz := \int_{\ds C^{m}} dz_{\ol 0}\int_{B_\ds C},
\end{align*}
where
\begin{align*}
\int_{B_\ds C} := \pi^{-2n}\pt{\ol{z_{m+2n}}}\pt{z_{m+2n}}\pt{\ol{z_{m+2n-1}}}\pt{z_{m+2n-1}}\ldots \pt{\ol{z_{m+1}}}\pt{z_{m+1}}
\end{align*}
is the complexified Berezin integral on $\Lambda\big(\ds C^{2n}\big)$ and $z_{\ol 0} := (z_i)_{i=1}^m$ denotes the even variables of $z = (z_i)_{i=1}^{m+2n}$.
\end{Def}

\begin{Def}
The \textbf{Fock product} on $\ot F$ is defined as
\begin{align*}
\fp{p,q} := \dfrac{1}{\gamma}\int_{\ds C^{m|2n}} \exp(-\norm{z}^2)p(z)\ol{q(z)}dz,\quad \text{ with }\quad \gamma := \int_{\ds C^{m|2n}} \exp(-\norm{z}^2)dz,
\end{align*}
for all $p,q\in \mc P(\ds C^{m|2n})$.
\end{Def}\symindex{$\fp{\cdot \, , \cdot}$}

\begin{lemma}\symindex{$\gamma$}
We have $\gamma = \pi^M$, with $M=m-2n$. In particular, the Fock product is well-defined.
\end{lemma}

\begin{proof}
We have
\begin{align*}
\int_{\ds C^{m|2n}} \exp(-\norm{z}^2)dz = \int_{\ds C^{m}} dz_{\ol 0}\int_{B_\ds C}\exp(-\norm{z}^2)
\end{align*}
with
\begin{align*}
\int_{B_\ds C}\exp(-\norm{z}^2) &= \pi^{-2n}\pt{\ol{z_{m+1}}}\pt{\ol{z_{m+2}}}\ldots\pt{\ol{z_{m+2n}}}\pt{z_{m+2n}}\pt{z_{m+2n-1}}\ldots \pt{z_{m+1}}\exp(-\norm{z}^2)\\
&= \pi^{-2n}\pt{\ol{z_{m+1}}}\pt{\ol{z_{m+2}}}\ldots\pt{\ol{z_{m+2n}}} \ol{z^{m+2n}}\, \ol{z^{m+2n-1}}\ldots \ol{z^{m+1}}\exp(-\norm{z_{\ol 0}}^2)\\
&= (-1)^n \pi^{-2n}\pt{\ol{z_{m+1}}}\pt{\ol{z_{m+2}}}\ldots\pt{\ol{z_{m+2n}}} \ol{z_{m+n}}\ldots \ol{z_{m+1}}\, \ol{z_{m+2n}}\ldots\ol{z_{m+n+1}}\\
&\quad \times \exp(-\norm{z_{\ol 0}}^2)\\
&= \pi^{-2n}\pt{\ol{z_{m+1}}}\pt{\ol{z_{m+2}}}\ldots\pt{\ol{z_{m+2n}}} \ol{z_{m+2n}}\ldots\ol{z_{m+1}}\exp(-\norm{z_{\ol 0}}^2)\\
&= \pi^{-2n}\exp(-\norm{z_{\ol 0}}^2)
\end{align*}
and therefore
\begin{align*}
\int_{\ds C^{m|2n}} \exp(-\norm{z}^2)dz &= \pi^{-2n}\int_{\ds C^{m}}\exp(-\norm{z_{\ol 0}}^2) dz_{\ol 0}= \pi^{-2n}\pi^{m},
\end{align*}
as desired.
\end{proof}

\begin{Opm}\label{Rem_SB_Space}
In \cite[Example 3.22]{dGM} the Fock product was already defined on the Segal-Bargmann superspace. There, the norm of $z$ is defined with respect to a different metric. Using our notations and conventions it is given by
\begin{align*}
\left< p, q\right>_{SB} := (-4)^{n}\pi^{-M}\int_{\ds C^{m|2n}}\exp\left(-\norm{z_{\ol 0}}^2 - \dfrac{\imath}{2}\sum_{i=m+1}^{2n}z_i\ol{z_i}\right)p(z)\ol{q(z)}dz
\end{align*}
for $p,q \in \mc P(\ds C^{m|2n})$.
\end{Opm}

\begin{Prop}\label{Prop SuperInd}
The Fock product is superhermitian. We also have
\begin{align*}
\fp{\pt i p,q} = (-1)^{|i||p|}\fp{p, z_i q}\quad \text{ and } \quad \fp{z_i p,q} = (-1)^{|i||p|}\fp{p, \pt i q},
\end{align*}
for all $p,q\in \mc P(\ds C^{m|2n})$ and $i\in\{1, \ldots, m+2n\}$.
\end{Prop}

\begin{proof}
The first claim is trivial. The second claim follows from
\begin{align*}
\fp{\pt i p,q} &= \dfrac{1}{\gamma}\int_{\ds C^{m|2n}} \pt i (p(z))\ol{q(z)}\exp(-\norm{z}^2)dz\\
&= - (-1)^{|i||p|}\dfrac{1}{\gamma}\int_{\ds C^{m|2n}} p(z)\pt i (\ol{q(z)}\exp(-\norm{z}^2))dz\\
&= - (-1)^{|i||p|}\dfrac{1}{\gamma}\int_{\ds C^{m|2n}} p(z)\pt i (\exp(-\norm{z}^2))\ol{q(z)}dz\\
&= (-1)^{|i||p|}\dfrac{1}{\gamma}\int_{\ds C^{m|2n}} p(z)\ol{z_i}\exp(-\norm{z}^2)\ol{q(z)}dz\\
&= (-1)^{|i||p|}\fp{p, z_i q}
\end{align*}
and the third claim follows from combining the first and second claims.
\end{proof}

\begin{Prop}\label{Prop fip is fp}
For $p,q\in \mc P(\ds C^{m|2n})$ we have
\begin{align*}
\fip{p,q} = \fp{p,q}
\end{align*} 
\end{Prop}

\begin{proof}
This proof is a straightforward generalisation of the proof in the symplectic case, see \cite[Proposition 2.6]{HKMO}. We give it here anyway, for completeness' sake.

First note that for all $p,q\in \mc P(\ds C^{m|2n})$ and $i\in\{1, \ldots m+2n\}$ we have
\begin{align*}
\fip{z_i p,q} &= (-1)^{|i||p|}\fip{p, \pt i q},\\
\fp{z_i p,q} &= (-1)^{|i||p|}\fp{p, \pt i q}.
\end{align*}
The first equation follows directly from the definition of the Fischer product and the second equation follows from Proposition \ref{Prop SuperInd}. We prove this proposition by using induction on the degree of $q$, $\deg(q)$. First, if $p=q=1\in \ds C$, it is clear that $\fip{p,q}=1$. From the way we normalised the Fock product, it is also clear that $\ipW{p,q}=1$. We conclude that the proposition holds for $\deg(p)=\deg(q)=1$. If now $\deg(p)$ is arbitrary and $\deg(q)=0$ then $\pt i q = 0$ for all  $i\in\{1, \ldots m+2n\}$ and hence
\begin{align*}
\fip{z_i p,q} &= (-1)^{|i||p|}\fip{p, \pt i q}=0,\\
\fp{z_i p,q} &= (-1)^{|i||p|}\fp{p, \pt i q}=0.
\end{align*}
Therefore the theorem holds if $\deg(q)=0$. We note that the theorem also holds if $\deg(p)=0$ and $\deg(q)$ is arbitrary. In fact,
\begin{align*}
\fip{p,q} = p(0)\ol{q(0)} = \fipbar{q,p}, \quad \text{ and }\quad \fp{p,q} = \fpbar{q,p}
\end{align*}
and then the theorem follows from previous considerations. Now assume the theorem holds for $\deg(q) \leq k$. For $\deg(q) \leq k+1$ we then have $\deg(\pt i q)\leq k$ and therefore, by the induction hypothesis
\begin{align*}
\fip{z_i p, q} = (-1)^{|i||p|}\fip{p, \pt i q} = (-1)^{|i||p|}\fp{p, \pt i q} = \fp{z_i p,q}.
\end{align*}
This shows the theorem holds for $\deg(q)\leq k+1$ and $p(0)=0$, i.e.\ without the constant term. But for constant $p$, i.e.\ $\deg(p)=0$ we have already seen that the theorem holds and therefore the proof is complete.
\end{proof}

In particular, since the Fock product is trivially superhermitian, we find that the Fischer product and Bessel-Fischer product are also superhermitian.

\begin{Cor}
The pairs $(\ot F, \fp{\cdot \, , \cdot})$ and $(F_{-\frac{1}{2}}, \bfip{\cdot \, , \cdot})$ are Hermitian superspaces.
\end{Cor}

We now prove that $\dotrho$ is ``infinitesimally superunitary'' on $(\ot F, \fp{\cdot \, , \cdot})$.

\begin{Prop}\label{Prop rho skew sup sym}
The Fock representation $\rol$ is skew-supersymmetric with respect to the Fock product, i.e.\
\begin{align*}
\fp{\dotrho(X)p,q} = - (-1)^{|X||p|} \fp{p,\dotrho(X)q},
\end{align*}
for all $X\in \mf g$ and $p,q\in \mc P(\ds C^{m|2n})$.
\end{Prop}

\begin{proof}
We have the following four cases.
\begin{itemize}
\item If $X =\ell_{ij}^-+ \ell_{ij}^+$, then $\dotrho(X)$ on $\mc P(\ds C^{m|2n})$ is given by
\begin{align*}
\dotrho(X) = \pil(\frac{1}{2}\ell_{ij}^-+2\ell_{ij}^+) = -\imath(z_iz_j+\pt i\pt j).
\end{align*}
\item If $X =\ell_{ij}^+- \ell_{ij}^-$, then $\dotrho(X)$ on $\mc P(\ds C^{m|2n})$ is given by
\begin{align*}
\dotrho(X) = \pil(-2\imath L_{\ell_{ij}}) = \imath(\beta_{ij}+z_i\pt j +(-1)^{|i||j|}z_j \pt i)
\end{align*}
\item If $X =2L_{\ell_{ij}}$, then $\dotrho(X)$ on $\mc P(\ds C^{m|2n})$ is given by
\begin{align*}
\dotrho(X) = \imath\pil(\frac{1}{2}\ell_{ij}^- - 2\ell_{ij}^+) = z_iz_j-\pt i \pt j.
\end{align*}
\item If $X = 4[L_{\ell_{ij}}, L_{\ell_{rs}}]$, then $\dotrho(X)$ on $\mc P(\ds C^{m|2n})$ is given by
\begin{align*}
\dotrho(X) &= \pil(4[L_{\ell_{ij}}, L_{\ell_{rs}}])\\
&=\beta_{jr}L_{is}+(-1)^{|r||s|}\beta_{js}L_{ir}+(-1)^{|i||j|}\beta_{ir}L_{js}+(-1)^{|i||j|+|r||s|}\beta_{is}L_{jr},
\end{align*}
with $L_{ij} = z_i\pt j - (-1)^{|i||j|}z_j\pt i$.
\end{itemize}
All four cases follow directly from Proposition \ref{Prop SuperInd}.
\end{proof}

We also have a reproducing kernel for the Fock product.

\begin{Prop}\label{Prop Repr Kern Fock}
The reproducing kernel for the Fock product is given by the superfunction $\exp(z\bullet\ol w)$, i.e.\
\begin{align*}
\fp{p(z), \exp(z\bullet \ol w)} = p(w),
\end{align*}
for all $p\in \mc P(\ds C^{m|2n})$.
\end{Prop}

\begin{proof}
This is a straightforward verification using $\pt{z^i} \exp(z\bullet \ol w) = w_i$.
\end{proof}

Let us use the following notations, $\wh z = (z_{ij})_{i,j=1}^{m+2n}$, $\wh w = (w_{ij})_{i,j=1}^{m+2n}$, $z = (z_i)_{i=1}^{m+2n}$, $w = (w_i)_{i=1}^{m+2n}$ and recall from section \ref{ss BF prod} that
\begin{align*}
\wh z|\wh w = \dfrac{1}{4}\sum_{i,j=1}^{m+2n}z^{ji}w_{ij}.
\end{align*}
We have
\begin{align*}
\psi(4(\wh z|\wh w)) &= \sum_{i,j=1}^{m+2n}\psi(z^{ji}w_{ij}) = \sum_{i,j=1}^{m+2n}z^{j}z^i w_{i}w_j = \sum_{i=1}^{m+2n}z^i w_{i}\sum_{j=1}^{m+2n}z^{j}w_j\\
&= (z\bullet w)^2.
\end{align*}
In particular, the reproducing kernel for the Bessel-Fischer product becomes
\begin{align*}
\psi(\ds I_{-\frac{1}{2}}(\wh z,\wh w)) = \cosh(z | \ol{w}).
\end{align*}
By Proposition \ref{Prop fip is fp} this implies that the reproducing kernel of the Fock product on $\ot F_e$ is given by $\cosh(z | \ol{w})$, which is consistent with Proposition \ref{Prop Repr Kern Fock}.

\section{The Segal-Bargmann transform}\label{S_SB}

In this section, we construct the Segal-Bargmann transform $\SB$ and prove that it intertwines the actions of $\dotpi$ and $\dotrho$. Moreover, the Segal-Bargmann transform brings us to a straightforward generalisation of the classical Hermite polynomials as the preimages of the monomials in the Fock model. To prove the Segal-Bargmann transform is superunitary we first extend our Hermitian superspaces to Hilbert superspaces.

\subsection{Definition and properties}

\begin{Def}\symindex{$\SB$}
The \textbf{Segal-Bargmann transform} is the superintegral operator given by
\begin{align*}
\SB(f(x))(z) := \dfrac{1}{\omega}\exp(-\frac{1}{2}R^2_z)\int_{\ds R^{m|2n}}\exp(2(z\bullet x))\exp(-R^2_x)f(x)dx,
\end{align*}
where $R^2_x$ and $R^2_z$ denote $R^2$ in the variable $x$ and $z$, respectively.
\end{Def}

Note that if $f\in \ot W_e$, then $f$ is an even function and the Segal-Bargmann transform becomes
\begin{align*}
\SB(f(x))(z) = \exp(-\frac{1}{2}R^2_z)\int_{\ds R^{m|2n}}\cosh(2(z\bullet x))\exp(-R^2_x)f(x)dx.
\end{align*}
From Corollary \ref{Cor F isom F} it now follows that the Segal-Bargmann transform induces a Segal-Bargmann transform $\wh{\SB} := \psi^{-1}_{\ds C}\circ\SB\circ\, \psi_{\ds R}$\symindex{$\wh{\SB}$} between $W_{-\frac{1}{2}}$ and $F_{-\frac{1}{2}}$. In the symplectic case our Segal-Bargmann transform $\SB$ coincides with the classical Segal-Bargmann transform, up to a scalar multiple, while $\wh\SB$ coincides with the Segal-Bargmann transform for minimal representations, defined in \cite[Section 3]{HKMO}. In Section \ref{SS_Superunitary} we will prove that $\SB$ is superunitary with respect to the renormalised $L^2$ and Fock products. Note that the classical Segal-Bargmann transform is unitary with respect to the regular $L^2$ and Fock products, which is why they only coincide up to a scalar multiple.

\begin{theorem}[Intertwining property]\label{Th Intertwining Prop}
The Segal-Bargmann transform intertwines the action $\dotpi$ on $\ot W$ with the action $\dotrho$ on $\ot F$, i.e.\
\begin{align}\label{Eq Intertwining}
\SB\, \circ\, \dotpi(X) = \dotrho(X) \circ \SB,
\end{align}
for all $X\in g$.
\end{theorem}

\begin{proof}
We will introduce the following notations to shorten the length of the calculations,
\begin{align*}
E_x := \exp(-R_x^2), \quad E_z := \exp(-\frac{1}{2}R_z^2), \quad E_t :=  \dfrac{1}{\omega}\exp(2(z\bullet x)).
\end{align*}
Then for $f\in \ot W$ the Segal-Bargmann transform becomes
\begin{align*}
\SB(f(x))(z) =\int_{\ds R^{m|2n}}E_zE_tE_xf(x)dx.
\end{align*}
Since we are working in both the $x$ and $z$ variables we also introduce the notations
\begin{align*}
\pt i^x := \pt {x^i}, \quad \text{ and }\quad \pt i^z := \pt {z^i}
\end{align*}
to denote $\pt i$ in the variable $x_i$ and $z_i$, respectively. We now have the following simple derivation rules:
\begin{align*}
\pt i^x E_x = -2x_iE_x, \quad \pt i^x E_t = 2z_iE_t, \quad \pt i^z E_z = -z_iE_z, \quad \pt i^z E_t = 2x_iE_t.
\end{align*}

Depending on the form of $X\in \mf g$, there are four distinct cases we need to deal with.

{\bf Case 1.}\\
Suppose $X = \ell_{ij}^++\ell_{ij}^-$, then
\begin{align*}
\dotpi(X) = -\imath (2x_ix_j + \dfrac{1}{2}\pt i^x \pt j^x) \quad \text{ and }\quad \dotrho(X) = -\imath (z_iz_j+\pt i^z \pt j^z).
\end{align*}
For $f\in \psi(W_\lambda)$ the left hand side of \eqref{Eq Intertwining} becomes
\begin{align*}
\SB(\dotpi(X) f(x))(z) &= -\imath(2\SB(x_ix_j f(x))(z) + \dfrac{1}{2}\SB(\pt i^x \pt j^x f(x))(z)),
\end{align*}
with
\begin{align*}
\dfrac{1}{2}\SB(\pt i^x \pt j^x f(x))(z) &= \dfrac{1}{2}\int_{\ds R^{m|2n}}E_zE_tE_x\pt i^x \pt j^x(f(x))dx\\
&= \dfrac{1}{2}\int_{\ds R^{m|2n}}\pt i^x \pt j^x(E_zE_tE_x)f(x)dx\\
&= \int_{\ds R^{m|2n}}(2z_iz_j-2z_ix_j-2x_iz_j+2x_ix_j-\beta_{ij})E_zE_tE_xf(x)dx.
\end{align*}
The right-hand side of \eqref{Eq Intertwining} becomes
\begin{align*}
\dotrho(X)\SB(f(x))(z) = -\imath (z_iz_j\SB(f(x))(z)+\pt i^z\pt j^z (\SB(f(x))(z)))
\end{align*}
with
\begin{align*}
\pt i^z \pt j^z\SB( f(x))(z) &= \int_{\ds R^{m|2n}}\pt i^z \pt j^z(E_zE_tE_x)f(x)dx\\
&= \int_{\ds R^{m|2n}}(-\beta_{ij}+z_iz_j-2z_ix_j-2x_iz_j+4x_ix_j)E_zE_tE_xf(x)dx,
\end{align*}
which gives us the same terms as on the left-hand side.

{\bf Case 2.}\\
Suppose $X= \ell_{ij}^+-\ell_{ij}^-$, then
\begin{align*}
\dotpi(X) = \imath (2x_ix_j - \dfrac{1}{2}\pt i^x \pt j^x) \quad \text{ and }\quad \dotrho(X) = \imath (\beta_{ij}+z_i\pt j^z+(-1)^{|i||j|}z_j\pt i^z).
\end{align*}
Using similar calculations as in the first case, the left-hand side of \eqref{Eq Intertwining} becomes
\begin{align*}
\SB(\dotpi(X) f(x))(z) &= \imath(2\SB(x_ix_j f(x))(z) - \dfrac{1}{2}\SB(\pt i^x \pt j^x f(x))(z))\\
&= \imath\int_{\ds R^{m|2n}}(-2z_iz_j+2z_ix_j+2x_iz_j+\beta_{ij})E_zE_tE_xf(x)dx.
\end{align*}
For the right-hand side of \eqref{Eq Intertwining}, we first note that
\begin{align*}
z_i\pt j^z \SB(f(x))(z) &= \int_{\ds R^{m|2n}}z_i\pt j^z (E_zE_t)E_xf(x)dx\\
&=  \int_{\ds R^{m|2n}}(-z_iz_j+2z_ix_j)E_zE_tE_xf(x)dx
\end{align*}
and therefore
\begin{align*}
\dotrho(X)\SB(f(x))(z) = \imath\int_{\ds R^{m|2n}}(\beta_{ij}-z_iz_j+2z_ix_j-z_iz_j+2x_iz_j)E_zE_tE_xf(x)dx,
\end{align*}
as desired.

{\bf Case 3.}\\
Suppose $X= 2L_{\ell_{ij}}$, then
\begin{align*}
\dotpi(X) = -(\beta_{ij}+x_i\pt j^x+(-1)^{|i||j|}x_j\pt i^x) \quad \text{ and }\quad \dotrho(X) = z_iz_j-\pt i^z\pt j^z.
\end{align*}
Using similar calculations as in the first case, the right-hand side of \eqref{Eq Intertwining} becomes
\begin{align*}
\dotrho(X)\SB(f(x))(z) &= z_iz_j\SB(f(x))(z)- \pt i^z \pt j^z\SB(f(x))(z)\\
&= \int_{\ds R^{m|2n}}(\beta_{ij}+2z_ix_j+2x_iz_j-4x_ix_j)E_zE_tE_xf(x)dx.
\end{align*}
For the left-hand side of \eqref{Eq Intertwining}, we first note that
\begin{align*}
\SB(x_i\pt j^x f(x))(z) &= \int_{\ds R^{m|2n}} E_zE_tE_x x_i\pt j^x(f(x))dx\\
&= -\int_{\ds R^{m|2n}} x_i\pt j^x(E_zE_tE_x) f(x)dx\\
&= -\int_{\ds R^{m|2n}}(2x_iz_j-2x_ix_j)E_zE_tE_xf(x)dx
\end{align*}
and therefore
\begin{align*}
\SB(\dotpi(X)f(x))(z) = \int_{\ds R^{m|2n}}(\beta_{ij}+2x_iz_j-2x_ix_j+2z_ix_j-2x_ix_j)E_zE_tE_xf(x)dx,
\end{align*}
as desired.

{\bf Case 4.}\\
Suppose $X= 4[L_{\ell_{ij}}, L_{\ell_{rs}}]$, then it is sufficient to prove
\begin{align*}
\SB(L_{ij}^x f(x)) = L_{ij}^z\SB(f(x)),
\end{align*}
with $L_{ij}^x = x_i\pt j^x-(-1)^{|i||j|}x_j\pt i^x$ and $L_{ij}^z = z_i\pt j^z+(-1)^{|i||j|}z_j\pt i^z$. Using the calculations from the previous two cases we find
\begin{align*}
\SB(L_{ij}^x f(x)) &= \int_{\ds R^{m|2n}}(-2x_iz_j+2z_ix_j)E_zE_tE_xf(x)dx
\end{align*}
and
\begin{align*}
L_{ij}^z\SB(f(x)) &= \int_{\ds R^{m|2n}}(2z_ix_j-2x_iz_j)E_zE_tE_xf(x)dx,
\end{align*}
as desired.
\end{proof}

We now wish to prove that the Segal-Bargmann induces a superunitary isomorphism between $\ot W$ and $\ot F$. We first need two technical lemmas.

\begin{lemma}\label{Lem SB 1}
We have $\SB(\exp(-R_x^2))(z) = 1$.
\end{lemma}

\begin{proof}
We start by separating the even and odd variables in the Segal-Bargmann transform. We find
\begin{align*}
\SB(\exp(-R_x^2))(z) &= \dfrac{1}{\omega}\exp(-\frac{1}{2} R^2_z)\int_{\ds R^{m|2n}}\exp(2(z|x))\exp(-2R^2_x)dx\\
&= \exp\left(-\frac{1}{2} \sum_{i=1}^{m}z_i^2\right)\exp\left(-\frac{1}{2} \sum_{i=m+1}^{m+2n}z^iz_i\right)\\
&\quad \times \left(\left(\dfrac{\pi}{2}\right)^{-\frac{m}{2}}\int_{\ds R^{m}}\exp\left(2\sum_{i=1}^m z_ix_i\right)\exp\left(-2\sum_{i=1}^m x_i^2\right)dx\right)\\
&\quad \times \left(2^{-2n}\pi^n\int_{B}\exp\left(2\sum_{i=m+1}^{m+2n} z^ix_i\right)\exp\left(-2\sum_{i=m+1}^{m+2n} x^ix_i\right)\right).
\end{align*}
What we wish to prove now splits into two parts. The first part is
\begin{align*}
\int_{\ds R^{m}}\exp\left(2\sum_{i=1}^m z_ix_i\right)\exp\left(-2\sum_{i=1}^m x_i^2\right)dx = \left(\dfrac{\pi}{2}\right)^{\frac{m}{2}}\exp\left(\frac{1}{2} \sum_{i=1}^{m}z_i^2\right),
\end{align*}
which follows from the classical case. The second part is
\begin{align*}
\int_{B}\exp\left(2\sum_{i=m+1}^{m+2n} z^ix_i\right)\exp\left(-2\sum_{i=m+1}^{m+2n} x^ix_i\right) = 2^{2n}\pi^{-n}\exp\left(\frac{1}{2} \sum_{i=m+1}^{m+2n}z^iz_i\right),
\end{align*}
which we will prove by a straightforward calculation. Since the Berezin integral only depends on the coefficient in front of $x_{m+1}\ldots x_{m+2n}$ we have\begingroup\allowdisplaybreaks
\begin{align*}
&\int_{B}\exp\left(2\sum_{i=m+1}^{m+2n} z^ix_i\right)\exp\left(-2\sum_{i=m+1}^{m+2n} x^ix_i\right)\\
&\quad = \int_B \sum_{k=0}^{2n}\sum_{l=0}^n \dfrac{(-1)^l 2^{k+l}}{k!l!}\left(\sum_{i=m+1}^{m+2n}z^ix_i\right)^k\left(\sum_{i=m+1}^{m+2n}x^ix_i\right)^l\\
&\quad = \int_B \sum_{k,l=0}^{n}\dfrac{(-1)^l 2^{2k+l}}{(2k)!l!}\left(\sum_{i=m+1}^{m+2n}z^ix_i\right)^{2k}\left(\sum_{i=m+1}^{m+2n}x^ix_i\right)^l\\
&\quad = \int_B \sum_{k,l=0}^{n}\dfrac{(-1)^{k+l} 2^{2k+l}}{(2k)!l!}\left(\sum_{i,j=m+1}^{m+2n}z^iz^jx_ix_j\right)^{k}\left(\sum_{i=m+1}^{m+2n}x^ix_i\right)^l\\
&\quad = \int_B \sum_{k=0}^{n}\dfrac{(-1)^{n} 2^{n+k}}{(2k)!(n-k)!}\left(\sum_{i,j=m+1}^{m+2n}z^iz^jx_ix_j\right)^{k}\left(\sum_{i=m+1}^{m+2n}x^ix_i\right)^{n-k}\\
&\quad = \int_B \sum_{k=0}^{n}\dfrac{(-1)^{n} 2^{n+k}}{(2k)!(n-k)!}\left(\sum_{i=m+1}^{m+2n}z^iz_ix^ix_i\right)^{k}\left(\sum_{i=m+1}^{m+2n}x^ix_i\right)^{n-k}\\
&\quad = \int_B \sum_{k=0}^{n}\dfrac{(-1)^{n} 2^{2n+k}}{(2k)!(n-k)!}\left(\sum_{i=m+1}^{m+n}z_iz_{i+n}x_ix_{i+n}\right)^{k}\left(\sum_{i=m+1}^{m+n}x_ix_{i+n}\right)^{n-k}\\
&\quad = 2^{2n}\sum_{k=0}^{n}\dfrac{1}{k!}\left(\sum_{i=m}^{m+n}z_iz_{i+n}\right)^k\int_B  (-1)^{n} x_{m+1}x_{m+1+n}\ldots x_{m+n}x_{m+2n}\\
&\quad = 2^{2n}\pi^{-n} \sum_{k=0}^{n}\dfrac{1}{k!}\left(\sum_{i=m}^{m+n}z_iz_{i+n}\right)^k = 2^{2n}\pi^{-n} \sum_{k=0}^{n}\dfrac{1}{2^k k!}\left(\sum_{i=m}^{m+2n}z^iz_{i}\right)^k \\
&\quad = 2^{2n}\pi^{-n}\exp\left(\frac{1}{2} \sum_{i=m+1}^{m+2n}z^iz_i\right),
\end{align*}
as desired.\endgroup
\end{proof}

\begin{lemma}\label{Lem SB x}
We have $\SB(2x_1\exp(-R_x^2))(z) = z_1$.
\end{lemma}

\begin{proof}
This will be proved in a more general setting in Proposition \ref{Prop Hermite}, see Remark \ref{Rem Hermite}. Note that Proposition \ref{Prop Hermite} only depends on the theorems we have proven before this one.
\end{proof}

\begin{Prop}\label{Prop SB gmod isom}
The Segal-Bargmann transform $\SB$ induces a $\mf g$-module isomorphism between $\ot W$ and $\ot F$.
\end{Prop}

\begin{proof}
From Lemma \ref{Lem SB 1} it is clear that the Segal-Bargmann transform maps the generating element of $\ot W_e$ to the generating element of $\ot F_e$. Similarly, from Lemma \ref{Lem SB x} it is clear that the Segal-Bargmann transform maps the generating element of $\ot W_o$ to the generating element $\ot F_o$. It also intertwines the actions of $\dotpi$ and $\dotrho$. We conclude that $\SB$ is an isomorphism of $\mf g$-modules.
\end{proof}

\begin{theorem}[Superunitary property]\label{Th_superunitary}
The Segal-Bargmann transform preserves the super-inner products, i.e.\
\begin{align*}
\fp{\SB f, \SB g} = \ipW{f,g},
\end{align*}
for all $f,g\in \ot W$.
\end{theorem}

\begin{proof}
We first look at the case $f= \exp(-R^2_x)$. Because of Lemma \ref{Lem SB 1} and Proposition \ref{Prop fip is fp} we have
\begin{align*}
\fp{\SB f, \SB g} &= \fp{1, \SB g} = \ol{\SB(g(x))}(z)|_{z=0} = \dfrac{1}{\omega}\int_{\ds R^{m|2n}}\exp(-R^2_x)\ol{g(x)}dx\\
&= \ipW{f,g},
\end{align*}
for all $g\in \ot W_e$. Now suppose $f,g\in \ot W_e$. Since $\exp(-R^2_x)$ is the element that generates $\ot W_e$, there exists a $Y\in U(\mf g)$ such that $f=\pi(Y)\exp(-R^2_x)$. Therefore we can reduce the general case to the previous case using the intertwining property (Theorem \ref{Th Intertwining Prop}) and the fact that the super-inner products are skew-supersymmetric for $\dotpi$ and $\dotrho$ (Propositions \ref{Prop pi skew sup sym} and \ref{Prop rho skew sup sym}):
\begin{align*}
\fp{\SB f, \SB g} &= \fp{\SB(\dotpi(Y)\exp(-R^2_x)), \SB g} = \fp{\dotrho(Y)\SB(\exp(-R^2_x)), \SB g}\\
&= - \fp{\SB(\exp(-R^2_x)), \dotrho(Y)\SB g}\\
&= -\fp{\SB(\exp(-R^2_x)), \SB (\dotpi(Y)g)}\\
&= -\ipW{\exp(-R^2_x), \dotpi(Y)g}\\
&= \ipW{\dotpi(Y)\exp(-R^2_x), g} = \ipW{f,g},
\end{align*}
which proves the theorem for $f,g\in \ot W_e$. Note that for all $f\in \ot W_e$ and $g\in \ot W_o$ we have
\begin{align*}
\fp{\SB f, \SB g} = \ipW{f,g} = 0,
\end{align*}
since $f$ and $\SB f$ are even while $g$ and $\SB g$ are odd. For $f,g\in \ot W_o$ we first look at the case $f= 2x_1\exp(-R^2_x)$. We obtain
\begin{align*}
\fp{\SB f, \SB g} &= \fp{z_1, \SB g} = \pt{i}^z\ol{\SB(g(x))}(z)|_{z=0} = \dfrac{1}{\omega}\int_{\ds R^{m|2n}}2x_1\exp(-R^2_x)\ol{g(x)}dx\\
&= \ipW{f,g},
\end{align*}
and the rest is entirely analogous to the $f,g\in \ot W_e$ case.
\end{proof}

\begin{Cor}
The inverse Segal-Bargmann transform is given by
\begin{align*}
\SB^{-1}(p(z))(x) = \dfrac{1}{\gamma}\exp(-R^2_x)\int_{\ds C^{m|2n}}\exp(-\norm{z}^2)\exp(-\frac{1}{2}R^2_{\ol z})\exp(2(\ol z\bullet x))p(z)dz,
\end{align*}
for all $p\in \ot F$.
\end{Cor}

\begin{proof}
We have
\begin{align*}
&\ipW{\SB^{-1}p, f} = \fp{p,\SB f} = \dfrac{1}{\gamma}\int_{\ds C^{m|2n}}\exp(-\norm{z}^2)p(z) \ol{\SB(f(x))(z)}dz\\
&\quad= \dfrac{1}{\gamma\omega}\int_{\ds C^{m|2n}}\int_{\ds R^{m|2n}}\exp(-\norm{z}^2)\exp(-\frac{1}{2}R^2_{\ol z})\exp(2(\ol z\bullet x))\exp(-R^2_x) p(z) \ol{f(x)}dxdz\\
&\quad =\ipW{\dfrac{1}{\gamma}\int_{\ds C^{m|2n}}\exp(-\norm{z}^2)\exp(-\frac{1}{2}R^2_{\ol z})\exp(2(\ol z\bullet x))\exp(-R^2_x)p(z)dz, f},
\end{align*}
for all $p\in \ot F$ and $f\in \ot W$. Because the Fock product and the $L^2$-product are non-degenerate (Propositions \ref{Prop Non deg bfip} and \ref{Prop L2 prod}), we obtain the desired result.
\end{proof}

\subsection{Hermite superpolynomials}\label{SS_Hermite}

We have the following generalisations of the Hermite functions and Hermite polynomials.

\begin{Def}\symindex{$h_\alpha$, $H_\alpha$, $\ot h_\alpha$ and $\ot H_\alpha$}
Suppose $\alpha\in \ds N^m\times \{0,1\}^{2n}$. The \textbf{Hermite superfunctions} on $\ds R^{m|2n}$ are defined by
\begin{align*}
h_\alpha(x) := (-1)^{|\alpha|}\exp\left(\frac{R^2_x}{2}\right)\pt x^\alpha \exp(-R^2_x),
\end{align*}
with
\begin{align*}
\pt x^\alpha := (\pt 1^x)^{\alpha_1}\ldots (\pt {m+2n}^x)^{\alpha_{m+2n}} \quad \text{ and } \quad |\alpha| := \sum_{i=1}^{m+2n}\alpha_i.
\end{align*}
The \textbf{Hermite superpolynomials} in $\mc P(\ds R^{m|2n})$ are defined by
\begin{align*}
H_\alpha(x) := \exp\left(\frac{R^2_x}{2}\right)h_\alpha(x).
\end{align*}
\end{Def}

We also introduce the following renormalisations:
\begin{align*}
\ot h_\alpha(x) &:= (\sqrt{2})^{-|\alpha|}h_\alpha(\sqrt{2}x) = (-1)^{|\alpha|}2^{-|\alpha|}\exp(R^2_x)\pt x^\alpha \exp(-2R^2_x),\\
\ot H_\alpha(x) &:= (\sqrt{2})^{-|\alpha|}H_\alpha(\sqrt{2}x) = \exp(R^2_x)\ot h_\alpha(x).
\end{align*}

For $m=1$ and $n=0$ the Hermite superpolynomials are precisely the classical Hermite polynomials. In the Symplectic case the renormalised Hermite superfunctions/superpolynomials are precisely the generalised Hermite functions/polynomials defined in \cite[Section 3.3]{HKMO}. The following proposition is a straightforward generalisation of \cite[Proposition 3.13]{HKMO}.

\begin{Prop}\label{Prop Hermite}
$\SB(\ot h_\alpha(x))(z) = z^\alpha$.
\end{Prop}

\begin{proof}
We have
\begin{align*}
\SB(\ot h_\alpha(x))(z) &= \dfrac{1}{\omega}\exp(-\frac{1}{2}R^2_z)\int_{\ds R^{m|2n}}\exp(2(z\bullet x))\exp(-R^2_x)\ot h_\alpha(x)dx\\
&= \dfrac{(-1)^{|\alpha|}2^{-|\alpha|}}{\omega}\exp(-\frac{1}{2}R^2_z)\int_{\ds R^{m|2n}}\exp(2(z\bullet x))\pt x^\alpha (\exp(-2R^2_x))dx\\
&= \dfrac{2^{-|\alpha|}}{\omega}\exp(-\frac{1}{2}R^2_z)\int_{\ds R^{m|2n}}\pt x^\alpha(\exp(2(z\bullet x))) \exp(-2R^2_x)dx\\
&= \dfrac{1}{\omega}\exp(-\frac{1}{2}R^2_z)\int_{\ds R^{m|2n}}z^\alpha\exp(2(z\bullet x)) \exp(-2R^2_x)dx\\
&= z^\alpha\SB(\exp(-R_x^2))(z).
\end{align*}
Because of Lemma \ref{Lem SB 1}, the theorem follows.
\end{proof}

\begin{Opm}\label{Rem Hermite}
For $\alpha = (1, 0, \ldots, 0)$ we get $\ot h_\alpha(x) = 2x_1\exp(-R^2_x)$. Proposition \ref{Prop Hermite} now implies $\SB(2x_1\exp(-R^2_x))(z) = z_1$, which is as we claimed in Lemma \ref{Lem SB x}.
\end{Opm}

\subsection{Hilbert superspaces}\label{SS_Superunitary}

In this section, we extend the Hermitian superspaces of Section \ref{S_Herm_superspace} to Hilbert superspaces. We also show that the folding isomorphisms and Segal-Bargmann transforms define superunitary isomorphisms between these Hilbert superspaces. Let us first introduce the definition of a Hilbert superspace in \cite{dGM}.

\begin{Def}
A \textbf{fundamental symmetry} of a Hermitian superspace $(\mc H, \left<\cdot\, ,\cdot \right>)$ is an endomorphism $J$ of $\mc H$ such that $J^4=1$, $\left<J(x) ,J(y) \right> = \left<x ,y \right>$ and $(\cdot \, , \cdot)_J$ defined by
\begin{align*}
\ipJ{x,y} := \left<x,J(y)\right>,
\end{align*}
for all $x,y\in \mc H$ is an inner product on $\mc H$.\\
A Hermitian superspace $(\mc H, \left<\cdot\, ,\cdot \right>)$ is a \textbf{Hilbert superspace} if there exists a fundamental symmetry $J$ such that $(\mc H, \ipJ{\cdot \, , \cdot})$ is a Hilbert space.
\end{Def}

Note that the choice of a fundamental symmetry does not matter for the topology, according to \cite[Theorem 3.4]{dGM}.

An example of a Hilbert superspace, which is also given in \cite{BdGT} and \cite{dGM} is the following.

\begin{Def}\symindex{$L^2(\ds R^{m\vert 2n})$}
We define the \textbf{Lebesgue superspace} $L^2(\ds R^{m|2n})$ by
\begin{align*}
L^2(\ds R^{m|2n}) := L^2(\ds R^{m})\otimes \Lambda(\ds R^{2n}),
\end{align*}
where $L^2(\ds R^{m})$ is the space of square integrable functions on $\ds R^{m}$.
\end{Def}

The super-inner product on $L^2(\ds R^{m|2n})$ is given by $\left<\cdot \, ,\cdot\right>_{L^2}$, or equivalently by $\ipW{\cdot\, , \cdot}$ and $L^2(\ds R^{m|2n})$ extends $\ot W$.

\begin{Prop}\label{Prop_DenseW}
The Hermitian superspace $(\ot W, \ipW{\cdot\, , \cdot})$ is dense in the Hilbert superspace $(L^2(\ds R^{m|2n}), \ipW{\cdot\, , \cdot})$.
\end{Prop}
\begin{proof}
The symplectic case follows from \cite[Theorem 2.30]{HKM} together with the unitary folding isomorphism $\psi$. In particular, we have that $\mc P(\ds R^m)\exp(-\norm{x_{\ol 0}}^2)$ is dense in $L^2(\ds R^m)$. This implies $\ot W = \mc P(\ds R^m)\exp(-\norm{x_{\ol 0}}^2) \otimes \Lambda(\ds R^{2n})\exp(-\norm{x_{\ol 1}}^2)$ is dense in $L^2(\ds R^{m})\otimes \Lambda(\ds R^{2n})\exp(-\norm{x_{\ol 1}}^2) = L^2(\ds R^{m})\otimes \Lambda(\ds R^{2n}) =L^2(\ds R^{m|2n})$.
\end{proof}

Let $S_F$\symindex{$S_F$, $T_F$ and $T_W$} be the fundamental symmetry of $(\ot F, \fp{\cdot\, , \cdot})$ given by
\begin{align*}
S_F(z_1z_2 \cdots z_k) = z^k z^{k-1} \cdots z_1,
\end{align*}
on monomials and extended linearly to all of $\ot F$. Then, we can define fundamental symmetries on $(F_{-\frac{1}{2}}, \bfip{\cdot\, , \cdot})$ and $(W_{-\frac{1}{2}}, \ipO{\cdot\, , \cdot})$ by $T_F := \psi_{\ds C}^{-1}\circ S_F\circ \psi_{\ds C}$ and $T_W := \SB^{-1}\circ T_F\circ \SB$, respectively.

\begin{Def}\symindex{$\mc F(\ds C^{m\vert 2n})$, $\mc F(\mc O_{\ds C})$ and $L(\mc O_{\ds R})$}
We define the \textbf{Fock superspace} $\mc F(\ds C^{m|2n})$ as the completion of $\ot F$ with respect to
\begin{align*}
(\cdot\, , \cdot)_{S_F} := \fp{\cdot\, , S_F(\cdot)}.
\end{align*}
We define $\mc F(\mc O_{\ds C})$ as the completion of $F_{-\frac{1}{2}}$ with respect to
\begin{align*}
(\cdot\, , \cdot)_{T_F} := \bfip{\cdot\, , T_F(\cdot)},
\end{align*}
We define $L(\mc O_{\ds R})$ as the completion of $W_{-\frac{1}{2}}$ with respect to
\begin{align*}
(\cdot\, , \cdot)_{T_W} := \ipO{\cdot\, , T_W(\cdot)},
\end{align*}
\end{Def}

It follows directly from the definitions that $(\mc F(\ds C^{m|2n}), \fp{\cdot\, , \cdot})$, $(\mc F(\mc O_{\ds C}), \bfip{\cdot\, , \cdot})$ and $(L(\mc O_{\ds R}), \ipO{\cdot\, , \cdot})$ are Hilbert superspaces. As an addendum to Remark \ref{Rem_SB_Space}, note that the Fock superspace and the Segal-Bargmann superspace defined in \cite[Example 3.22]{dGM} are both generalisations of the classical Fock space.

We can also define a fundamental symmetry on $\ot W$ by $S_W := \SB^{-1}\circ S_F\circ \SB$.\symindex{$S_W$} Then, it follows from Proposition \ref{Prop_DenseW} that the completion of $\ot W$ with respect to
\begin{align*}
(\cdot\, , \cdot)_{S_W} := \ipW{\cdot\, , S_W(\cdot)}
\end{align*}
is $L^2(\ds R^{m|2n})$.

Let us now introduce some definitions from \cite{dGM} related to superunitarity.

\begin{Def}\symindex{$\mc B(\cdot, \cdot)$ and $\mc B(\cdot)$}
Let $(\mc H_1, \ip{\cdot\, , \cdot}_1)$ and $(\mc H_2, \ip{\cdot\, , \cdot}_2)$ be Hilbert superspaces and suppose $T: \mc H_1 \rightarrow \mc H_2$ is a linear operator. We call $T$ a \textbf{bounded operator} between $\mc H_1$ and $\mc H_2$ if it is continuous with respect to their Hilbert topologies. The set of bounded operators is denoted by $\mc B(\mc H_1, \mc H_2)$ and $\mc B(\mc H_1) := \mc B(\mc H_1, \mc H_1)$.
\end{Def}

\begin{Def}\symindex{$\cdot^\dagger$}
Let $(\mc H_1, \ip{\cdot\, , \cdot}_1)$ and $(\mc H_2, \ip{\cdot\, , \cdot}_2)$ be Hilbert superspaces and suppose $T\in \mc B(\mc H_1, \mc H_2)$. The \textbf{superadjoint} of $T$ is the operator $T^\dagger\in \mc B(\mc H_2, \mc H_1)$ such that
\begin{align*}
\ip{T^\dagger(x),y}_1 = (-1)^{|T||x|}\ip{x, T(y)}_2,
\end{align*}
for all $x\in \mc H_2$, $y\in \mc H_1$.
\end{Def}

\begin{Def}\symindex{$\mc U(\cdot, \cdot)$ and $\mc U(\cdot)$}
Let $(\mc H_1, \ip{\cdot\, , \cdot}_1)$ and $(\mc H_2, \ip{\cdot\, , \cdot}_2)$. A \textbf{superunitary operator} between $\mc H_1$ and $\mc H_2$ is an even parity operator $\psi\in\mc B(\mc H_1, \mc H_2)$ satisfying $\psi^\dagger\psi = \psi\psi^\dagger = \ds 1$. The set of superunitary operators is denoted by $\mc U(\mc H_1, \mc H_2)$ and $\mc U(\mc H_1) := \mc U(\mc H_1, \mc H_1)$.
\end{Def}

For this definition of a superunitary operator, we have the following theorem.

\begin{theorem}\label{Th_superunitary_isom}
\begin{itemize} The following maps are superunitary isomorphisms. 
\item[(1)] The Segal-Bargmann transform $\SB$ from $L^2(\ds R^{m|2n})$ to $\mc F(\ds C^{m|2n})$.
\item[(2)] The Folding isomorphism $\psi_{\ds C}$ from $\mc F(\mc O_{\ds C})$ to $\mc F_{\text{even}}(\ds C^{m|2n})$.
\item[(3)] The Folding isomorphism $\psi_{\ds R}$ from $L^2(\mc O_{\ds R})$ to $L^2_{\text{even}}(\ds R^{m|2n})$.
\item[(4)] The Segal-Bargmann transform $\wh\SB$ from $L^2(\mc O_{\ds R})$ to $\mc F(\mc O_{\ds C})$.
\end{itemize}
\end{theorem}
\begin{proof}
Item (1) follows directly from Theorem \ref{Th_superunitary}. Item (2) follows from equation \eqref{Eq_Bessel-Fock}. Item (3) follows from the definition of the Schr\"odinger product. Item (4) is now a direct consequence of the previous items.
\end{proof}

%

\section{Integration to the metaplectic Lie supergroup $\Mp(2m|2n,2n)$}\label{S_Meta}

In this section, we show that our representations of $\mf g$ integrate to superunitary representations of the Metaplectic Lie supergroup in the sense of \cite{dGM}. Specifically, we will show that our Schr\"odinger representation $\dotpi$ is, up to a Fourier transform, equal to the metaplectic representation constructed in \cite{dGM}. Then, the integrability of the other representations follows directly from Theorem \ref{Th_superunitary_isom}.

In \cite{dGM} they only work with real Lie superalgebras. Therefore, we will assume all Lie superalgebras occurring in this section are real.

\subsection{The Fourier transform}\label{SS_Fourier}

\begin{Def}\symindex{$\mc S(\ds R^{m|2n})$ and $\mc S'(\ds R^{m|2n})$}
We define the \textbf{Schwartz space} $\mc S(\ds R^{m|2n})$ by
\begin{align*}
\mc S(\ds R^{m|2n}) := \mc S(\ds R^{m})\otimes \Lambda(\ds R^{2n}),
\end{align*}
where $\mc S(\ds R^{m})$ is the Schwartz space of rapidly decreasing functions on $\ds R^{m}$. We define its dual space by
\begin{align*}
\mc S'(\ds R^{m|2n}) := \mc S'(\ds R^{m})\otimes \Lambda(\ds R^{2n}),
\end{align*}
where $\mc S'(\ds R^{m})$ is the space of tempered distributions on $\ds R^{m}$
\end{Def}

Note that we have the inclusions $\ot W\subseteq \mc S(\ds R^{m|2n})\subseteq L^2(\ds R^{m|2n}) \subseteq \mc S'(\ds R^{m|2n})$.

In \cite[Section 6.1]{BF} the following Fourier transform with respect to an orthosymplectic metric is defined.

\begin{Def}\symindex{$\ds F^{\pm}$}
The \textbf{super Fourier transform} $\ds F^{\pm}: \mc S'(\ds R^{m|2n}) \rightarrow \mc S'(\ds R^{m|2n})$ is given by
\begin{align*}
\ds F^{\pm}(f(\ell))(x) = \frac{1}{\sqrt{2^m\pi^M}}\int_{\ds R^{m|2n}}\exp(\pm\imath(x\bullet \ell))f(\ell)d\ell,
\end{align*}
with $(x\bullet \ell) = \sum_{i=1}^{m+2n}x^i\ell_i$ the trace product, see Definition \ref{Def_trace_product}.
\end{Def}

Let $\pt i^{\ell}$ and $\pt i^x$ denote $\pt i$ in the variable $\ell$ and $x$, respectively. We obtain the following properties from \cite[Proposition 6.1]{BF} or \cite[Theorem 7 and Lemma 3]{DeB}.
\begin{align*}
\ds F^{\pm}(\pt i^\ell f(\ell))(x) &= \pm \imath x_i\ds F^{\pm}(f(\ell))(x),\\
\ds F^{\pm}(\ell_i f(\ell))(x) &= \pm \imath \pt i^x\ds F^{\pm}(f(\ell))(x),\\
\ds F^{\pm}\ds F^{\mp} &= \id.
\end{align*}
Moreover, it follows from
\begin{align*}
\ipW{\ds F^{\pm}(f),g} &= \frac{1}{\omega}\int_{\ds R^{m|2n}}\ds F^{\pm}(f(\ell))(x)\ol{g(x)}dx\\
&= \frac{1}{\omega\sqrt{2^m\pi^M}}\int_{\ds R^{m|2n}}\int_{\ds R^{m|2n}}\exp(\pm\imath(x\bullet\ell))f(\ell)\ol{g(x)}d\ell dx\\
&= \frac{1}{\omega\sqrt{2^m\pi^M}}\int_{\ds R^{m|2n}}f(\ell)\int_{\ds R^{m|2n}}\ol{\exp(\mp\imath(x\bullet \ell))g(x)}dx d\ell\\
&= \frac{1}{\omega}\int_{\ds R^{m|2n}}f(\ell)\ol{\ds F^{\mp}(g(x))(\ell)} d\ell\\
&= \ipW{f, \ds F^{\mp}(g)}
\end{align*}
that $\ds F^{\pm}$ preserves the $L^2$-product.

Note that the canonical extension of $\dotpi$ to $\mc S'(\ds R^{m|2n})$ is well-defined. We define the representation $\hatpi$\symindex{$\hatpi$} on $\mc S(\ds R^{m|2n})$ or $\mc S'(\ds R^{m|2n})$ by
\begin{align*}
\hatpi(X) := \ds F^- \circ \dotpi (X) \circ \ds F^+,
\end{align*}
for all $X\in \mf g$.

\begin{Prop}\label{Prop_Fourier_min}
The representation $\hatpi$ of $\mf g$ on  $\mc S(\ds R^{m|2n})$ or $\mc S'(\ds R^{m|2n})$ is given by
\begin{itemize}
\item $\hatpi(U_{\ul i, \ul j}) = -2\imath\pt i\pt j$,\\
\item $\hatpi(U_{\ot i, \ul j}) = \dfrac{1}{2} \beta_{ij} + x_i\pt j$,\\
\item $\hatpi(U_{\ot i, \ot j}) = \dfrac{\imath}{2}x_i x_j$.
\end{itemize}
\end{Prop}
\begin{proof}
This follows immediately from
\begin{align*}
[\pt j, x_i] = \pt j x_i - (-1)^{|i||j|}x_i \pt j = \beta_{ij},
\end{align*}
together with Proposition \ref{Prop_pi_in_U} and the properties of the super Fourier transform.
\end{proof}

\subsection{The metaplectic representation}\label{SS_meta_rep}

Recall the Heisenberg Lie superalgebra $\mf h(2m|p,q)$ defined in Section \ref{SS_Heisenberg}. In \cite[Section 5]{dGM} a Schr\"odinger representation of $\mf h(2m|p,q)$ was constructed. This representation was then used in \cite[Section 7.3]{dGM} to construct a Schr\"odinger representation of the Metaplectic Lie supergroup $\Mp(2m|p,q)$. We will briefly reconstruct these representations here. Note that we only constructed a minimal representation for $\mf g = \mf{osp}(2m|4n, \Omega) = \mf{osp}(2m|2n, 2n)$. Therefore, we restrict ourselves to the case $p=q=2n$.

Set $\mf h := \mf h(2m|4n, \Omega)$\symindex{$\mf h$}. Let $(e_i)_{i=1}^{2m+4n}$ be a basis of $\ds R^{2m|4n}$ and let $\ul V$ and $\ot V$\symindex{$\ul V$ and $\ot V$} be the subspaces generated by $(e_{\ul i})_{i=1}^{m+2n}$ and $(e_{\ot i})_{i=1}^{m+2n}$, respectively. Then, we have the decomposition $\mf h \cong \ul V\oplus \ot V \oplus \ds R Z$. From \cite[Section 5.1]{dGM} we obtain the following representation on the Schwartz space $\mc S(\ds R^{m|2n})$.

\begin{Def}\symindex{$U_*$ and $\ol h$}
The \textbf{Schr\"odinger representation} $U_*$ of $\mf h$ with parameter $\ol h\in \ds R\setminus \{0\}$ is given by
\begin{itemize}
\item $U_*(e_{\ul i}) = \pt{i} = \sum_{j=1}^{m+2n}\beta_{ij}\pt{x_j}$,
\item $U_*(e_{\ot i}) = \imath \ol h x_i$,
\item $U_*(Z) = \imath \ol h$,
\end{itemize}
for all $i\in\{1, \ldots, m+2n\}$.
\end{Def}

This representation can be extended to a representation of the quotient algebra $U(\mf h)/\left<Z-\imath \ol h\right>$, where $U(\mf h)$ is the universal enveloping algebra of $\mf h$. The canonical Lie bracket on $U(\mf h)/\left<Z-1\right>$ turns the space of quadratic elements into a Lie algebra\symindex{$\mc L_2$}
\begin{align*}
\mc L_2 := \mbox{span}_{\ds R}\{XY+(-1)^{|X||Y|}YX| X,Y\in \mf h\}\subseteq U(\mf h)/\left<Z-1\right>.
\end{align*}
The basis $(e_i)_{i=1}^{2m+4n}$ of $\ds R^{2m|4n}$ induces the following basis of $\mc L_2$:\symindex{$V_{ij}$}
\begin{align*}
V_{ij} := e_ie_j + (-1)^{|e_i||e_j|}e_je_i,
\end{align*}
for $i,j\in\{1, \ldots, 2m+4n\}$. The Lie bracket is given by
\begin{align*}
[V_{ij},V_{kl}] = 2(\Omega_{jk}V_{il} + (-1)^{|i||j|}\Omega_{ik}V_{jl} + (-1)^{|k||l|}\Omega_{jl}V_{ik} + (-1)^{|i||j|+|k||l|}\Omega_{il}V_{jk}).
\end{align*}
By comparing the Lie brackets of $\mf g$ and $\mc L_2$, it follows directly that $\mc L_2\cong \mf g$ by the isomorphism $V_{ij}\mapsto 2U_{ij}$, for all $i,j\in\{1, \ldots, 2m+4n\}$.

In \cite[Section 7.3]{dGM} a representation $\mu_*$ of $\mf g$ on $\mc S'(\ds R^{m|2n})$ was then defined by setting\symindex{$\mu_*$}
\begin{align*}
\mu_* := \frac{1}{\imath \ol h}\left.(U_*)\right|_{\mc L_2}.
\end{align*}
In the $(U_{ij})$-basis we have
\begin{align*}
\mu_*(U_{ij}) := \frac{1}{2\imath \ol h}(U_*(e_i)U_*(e_j) + (-1)^{|e_i||e_j|}U_*(e_j)U_*(e_i)),
\end{align*}
which gives
\begin{align*}
\mu_*(U_{\ul i, \ul j}) &= -\frac{\imath}{\ol h}\pt i \pt j,\\
\mu_*(U_{\ot i, \ul j}) &= \frac{1}{2}\beta_{ij} + x_i\pt j,\\
\mu_*(U_{\ot i, \ot j}) &= \imath \ol h x_ix_j.
\end{align*}
From Proposition \ref{Prop_Fourier_min} it follows that $\mu_*=\hatpi$ for $\ol h=\dfrac{1}{2}$.

As in \cite{dGM}, the representation $\mu_0$ of $\Mp(2m, \ds R)\times \Spin^\circ(2n,2n)$ is defined as the tensor product of the metaplectic representation of $\Mp(2m, \ds R)$ with the spin representation of $\Spin^\circ(2n,2n)$. We now introduce the definition of a superunitary representation given in \cite{dGM}.

\begin{Def}\label{Defsuperunitary}
A \textbf{superunitary representation} of a Lie supergroup $G=(G_0,\mf g)$ is a triple $(\mc H, \pi_0, d\pi)$ such that
\begin{itemize}
\item $\mc H$ is a Hilbert superspace.
\item $\pi_0:G_0\rightarrow \mc U(\mc H)$ is a group morphism.
\item For all $v\in \mc H$, the maps $\pi_0^v : g \mapsto \pi_0(g)v$ are continuous on $G_0$.
\item $d\pi : \mf g \rightarrow \End(\mc H^{\infty})$ is a $\ds R$-Lie superalgebra morphism such that $d\pi = d\pi_0$ on $\mf g_{\ol 0}$, $d\pi$ is skew-supersymmetric with respect to $\left< \cdot \, ,\cdot\right>$ and
\begin{align*}
\pi_0(g)d\pi(X)\pi_0(g)^{-1} = d\pi(\Ad(g)(X)),\quad \text{ for all } g\in G_0 \text{ and } X\in \mf g_{\ol 1}.
\end{align*}
Here $\mc H^\infty$ is the space of smooth vectors of the representation $\pi_0$ and $\Ad$ is the adjoint representation of $G_0$ on $\mf g$.
\end{itemize}
\end{Def}

The triple $(L^2(\ds R^{m|2n}), \mu_0, \mu_*)$ is called the \textbf{metaplectic representation} and is a superunitary representation of $(\Mp(2m|2n,2n), \mf g)$ by \cite[Theorem 7.13]{dGM}.

If we apply the Fourier transform and the superunitary isomorphisms given in Theorem \ref{Th_superunitary_isom} on the metaplectic representation, we obtain the following theorem.

\begin{theorem}\label{Th_SUR}
The following triples are superunitary representations of the Metaplectic Lie supergroup $\Mp(2m|2n,2n)$.
\begin{itemize}
\item The triple $(L^2(\ds R^{m|2n}), \ot \pi_0, \dotpi)$, with $\ot \pi_0 := \ds F^+ \circ \mu_0\circ \ds F^-$.\\
\item The triple $(\mc F(\ds C^{m|2n}), \ot \rho_0, \dotrho)$, with $\ot \rho_0 := \SB\, \circ\, \ot \pi_0\circ \SB^{-1}$.\\
\item The triple $(L^2(\mc O_{\ds R}), \pi_0, \pil)$, with $\pi_0 := \psi_{\ds R}^{-1}\circ \ot \pi_0\circ \psi_{\ds R}$.\\
\item The triple $(\mc F(\mc O_{\ds C}), \rho_0, \rol)$, with $\rho_0 := \psi_{\ds C}^{-1}\circ \ot \rho_0\circ \psi_{\ds C}$.
\end{itemize}
\end{theorem}

The first two representations in Theorem \ref{Th_SUR} are superunitarily equivalent to the metaplectic representation in \cite{dGM}. By Theorem \ref{ThDecF}, the last two representations are superunitarily equivalent to an indecomposable component of the metaplectic representation, which is an irreducible component for $M=m-2n\not\in -2\ds N$.

%

\section{The other minimal representation}\label{S_lambda_1}

In this section, we briefly discuss our findings and hypotheses concerning the $\lambda=1$ case from Section \ref{SS_min}. This case generalises the split orthogonal case. Since the split orthogonal case is equivalent to the orthosymplectic case for $m=0$, it seems reasonable to suspect that the $\lambda=1$ case will often be equivalent to the orthosymplectic case. However, as far as we know, a classical equivalence has never been constructed explicitly and there is no straightforward generalisation of the classical arguments to the super setting.

Note that the minimal representation of $\mf o(2,2)$, i.e.\ the situation $(m,n) = (0,1)$, can only be obtained in the orthosymplectic case and not in the $\lambda=1$ case. This suggests that at least for $n=1$ the cases might not be equivalent. More generally, it seems that for $M=m-2n > -3$, which encompasses the $n=1$ situation, the orthosymplectic case and the $\lambda=1$ case might not always be equivalent.

For $M\neq -1$, it follows from \cite[Theorem 6.4]{BF} that the annihilator ideal in the orthosymplectic case is a Joseph-like ideal constructed in \cite{CSS2}. For $M\leq -3$ it then follows from the characterisation in \cite{Garfinkle} that the annihilator ideal is exactly this Joseph-like ideal. For $M>-3$ no such characterisation is known. This indicates that the orthosymplectic case and the $\lambda=1$ case could potentially lead to two non-equivalent representations, which both contain the same Joseph-like ideal in their annihilator ideal.

Note that the $\lambda=1$ case has been studied in \cite{BC3} and \cite{C3} for $m=n=1$, i.e.\ for the $D(2,1,\alpha)$ case. However, no equivalence or non-equivalence with the orthosymplectic case was found.

\appendix

\section{Long and straightforward calculations}\label{AppS_Long}

In this section, we give long but straightforward calculations concerning the Bessel operators.

\begin{Prop}[Proposition \ref{Prop_Explicit_Bessel}]\label{Prop_App_Expl_Bessel}
We have
\begin{align*}
\bessel(\ell_{ij}) &= -2\lambda\sum_{k,l=1}^{m+2n}(1+\delta_{kl})\beta_{jk}\beta_{il}\pt{\ell_{kl}}\\
&\quad +\sum_{k,l,r,s=1}^{m+2n}(-1)^{|k||i|}(1+\delta_{kl}+\delta_{rs}+\delta_{kl}\delta_{rs})\beta_{is}\beta_{jl}\ell_{kr}\pt{\ell_{sr}}\pt{\ell_{lk}},
\end{align*}
for all $1\leq i,j \leq m+2n$.
\end{Prop}

\begin{proof}
\begingroup
\allowdisplaybreaks
For all $1\leq i,j,k,l \leq m+2n$ we have
\begin{align*}
&\widetilde{P}_{\ell_{ij},\ell_{kl}}(\ell_{rs})\\
&= (-1)^{(|r|+|s|)(|i|+|j|+|k|+|l|)}\\
&\quad \times(L_{\ell_{ij}}L_{\ell_{kl}}+(-1)^{(|i|+|j|)(|k|+|l|)}L_{\ell_{kl}}L_{\ell_{ij}}-L_{\ell_{ij}\ell_{kl}})(\ell_{rs})\\
&=\dfrac{1}{2}(-1)^{(|r|+|s|)(|i|+|j|+|k|+|l|)}\\
&\quad \times(L_{\ell_{ij}}(\beta_{lr}\ell_{ks}+(-1)^{|k||l|}\beta_{kr}\ell_{ls}+(-1)^{|r||s|}\beta_{ls}\ell_{kr}+(-1)^{|k||l|+|r||s|}\beta_{ks}\ell_{lr})\\
&\quad +(-1)^{|i||k|+|i||l|+|j||k|+|j||l|}\\
&\quad \times L_{\ell_{kl}}(\beta_{jr}\ell_{is}+(-1)^{|i||j|}\beta_{ir}\ell_{js}+(-1)^{|r||s|}\beta_{js}\ell_{ir}+(-1)^{|i||j|+|r||s|}\beta_{is}\ell_{jr})\\
&\quad - (\beta_{jk}L_{\ell_{il}}+(-1)^{|i||j|}\beta_{ik}L_{\ell_{jl}}+(-1)^{|k||l|}\beta_{jl}L_{\ell_{ik}}+(-1)^{|i||j|+|k||l|}\beta_{il}L_{\ell_{jk}})(\ell_{rs}))\\
&=\dfrac{1}{4}(-1)^{|i||r|+|j||r|+|k||r|+|l||r|+|i||s|+|j||s|+|k||s|+|l||s|}\\
&\quad \times(\beta_{lr}\beta_{jk}\ell_{1s}+(-1)^{|i||j|}\beta_{lr}\beta_{ik}\ell_{js}\\
&\quad+(-1)^{|k||s|}\beta_{lr}\beta_{js}\ell_{ik}+(-1)^{|i||j|+|k||s|}\beta_{lr}\beta_{is}\ell_{jk}\\
&\quad +(-1)^{|k||l|} \beta_{kr}\beta_{jl}\ell_{is}+(-1)^{|i||j|+|k||l|}\beta_{kr}\beta_{il}\ell_{js}\\
&\quad +(-1)^{|k||l|+|l||s|}\beta_{kr}\beta_{js}\ell_{il}+(-1)^{|k||l|+|i||j|+|l||s|}\beta_{kr}\beta_{is}\ell_{jl}\\
&\quad + (-1)^{|r||s|} \beta_{ls}\beta_{jk}\ell_{ir}+(-1)^{|i||j|+|r||s|}\beta_{ls}\beta_{ik}\ell_{jr}\\
&\quad +(-1)^{|k||r|+|r||s|}\beta_{ls}\beta_{jr}\ell_{ik}+(-1)^{|i||j|+|k||r|+|r||s|}\beta_{ls}\beta_{ir}\ell_{jk}\\
&\quad +(-1)^{|k||l|+|r||s|}\beta_{ks}\beta_{jl}\ell_{ir}+(-1)^{|i||j|+|k||l|+|r||s|}\beta_{ks}\beta_{il}\ell_{jr}\\
&\quad +(-1)^{|k||l|+|r||s|+|l||r|}\beta_{ks}\beta_{jr}\ell_{il}+(-1)^{|k||l|+|r||s|+|i||j|+|l||r|}\beta_{ks}\beta_{ir}\ell_{jl})\\
&\quad +(-1)^{|i||k|+|i||l|+|j||k|+|j||l|}\\
&\quad \times(\beta_{jr}\beta_{li}\ell_{ks}+(-1)^{|k||l|}\beta_{jr}\beta_{ki}\ell_{ls}+(-1)^{|i||s|}\beta_{jr}\beta_{ls}\ell_{ki}+(-1)^{|k||l|+|i||s|}\beta_{jr}\beta_{ks}\ell_{li}\\
&\quad +(-1)^{|i||j|} \beta_{ir}\beta_{lj}\ell_{ks}+(-1)^{|k||l|+|i||j|}\beta_{ir}\beta_{kj}\ell_{ls}\\
&\quad +(-1)^{|i||j|+|j||s|}\beta_{ir}\beta_{ls}\ell_{kj}+(-1)^{|i||j|+|k||l|+|j||s|}\beta_{ir}\beta_{ks}\ell_{lj}\\
&\quad + (-1)^{|r||s|} \beta_{js}\beta_{li}\ell_{kr}+(-1)^{|k||l|+|r||s|}\beta_{js}\beta_{ki}\ell_{lr}\\
&\quad +(-1)^{|i||r|+|r||s|}\beta_{js}\beta_{lr}\ell_{ki}+(-1)^{|k||l|+|i||r|+|r||s|}\beta_{js}\beta_{kr}\ell_{li}\\
&\quad +(-1)^{|i||j|+|r||s|}\beta_{is}\beta_{lj}\ell_{kr}+(-1)^{|k||l|+|i||j|+|r||s|}\beta_{is}\beta_{kj}\ell_{lr}\\
&\quad +(-1)^{|i||j|+|r||s|+|j||r|}\beta_{is}\beta_{lr}\ell_{kj}+(-1)^{|i||j|+|r||s|+|k||l|+|j||r|}\beta_{is}\beta_{kr}\ell_{lj})\\
&\quad - \beta_{jk}\beta_{lr}\ell_{is} - (-1)^{|i||l|}\beta_{jk}\beta_{ir}\ell_{ls} - (-1)^{|r||s|}\beta_{jk}\beta_{ls}\ell_{ir} - (-1)^{|i||l|+|r||s|}\beta_{jk}\beta_{is}\ell_{lr}\\
&\quad - (-1)^{|i||j|}\beta_{ik}\beta_{lr}\ell_{js} - (-1)^{|i||j|+|j||l|}\beta_{ik}\beta_{jr}\ell_{ls} \\
&\quad - (-1)^{|i||j|+|r||s|}\beta_{ik}\beta_{ls}\ell_{jr} - (-1)^{|i||j|+|j||l|+|r||s|}\beta_{ik}\beta_{js}\ell_{lr}\\
&\quad - (-1)^{|k||l|}\beta_{jl}\beta_{kr}\ell_{is} - (-1)^{|i||k|+|k||l|}\beta_{jl}\beta_{ir}\ell_{ks}\\
&\quad - (-1)^{|k||l|+|r||s|}\beta_{jl}\beta_{ks}\ell_{ir} - (-1)^{|i||k|+|k||l|+|r||s|}\beta_{jl}\beta_{is}\ell_{kr}\\
&\quad - (-1)^{|i||j|+|k||l|}\beta_{il}\beta_{kr}\ell_{2s} - (-1)^{|i||j|+|k||l|+|j||k|}\beta_{il}\beta_{jr}\ell_{ks}\\
&\quad - (-1)^{|i||j|+|k||l|+|r||s|}\beta_{il}\beta_{ks}\ell_{jr} - (-1)^{|i||j|+|k||l|+|j||k|+|r||s|}\beta_{il}\beta_{js}\ell_{kr})\\
&=\dfrac{1}{2}(-1)^{|i||r|+|j||r|+|k||r|+|l||r|+|i||s|+|j||s|+|k||s|+|l||s|}\\
&\quad \times ((-1)^{|j||k|}(\beta_{lr}\beta_{js}+(-1)^{|j||l|}\beta_{jr}\beta_{ls})\ell_{ik}\\
&\quad +(-1)^{|i||j|+|i||k|}(\beta_{lr}\beta_{is}+(-1)^{|i||l|}\beta_{ir}\beta_{ls})\ell_{jk}\\
&\quad +(-1)^{|j||l|+|k||l|}(\beta_{kr}\beta_{js}+(-1)^{|j||k|}\beta_{jr}\beta_{ks})\ell_{il}\\
&\quad +(-1)^{|i||j|+|i||l|+|k||l|}(\beta_{kr}\beta_{is}+(-1)^{|i||k|}\beta_{ir}\beta_{ks})\ell_{jl}.
\end{align*}
This implies
\begin{align*}
&\sum_{\substack{1\leq k\leq l\leq m+2n\\1\leq r\leq s \leq m+2n}}\widetilde{P}_{\ell_{kl},\ell_{rs}}(\ell_{ij})\pt{\ell_{rs}}\pt{\ell_{kl}}\\
&= \sum_{k,l,r,s=1}^{m+2n}\dfrac{1}{2}(-1)^{|k||i|+|l||i|+|r||i|+|s||i|+|k||j|+|l||j|+|r||j|+|s||j|+|l||r|}\\
&\quad \times(1+\delta_{kl}+\delta_{rs}+\delta_{kl}\delta_{rs})\beta_{si}\beta_{lj}\ell_{kr}\pt{\ell_{rs}}\pt{\ell_{kl}}\\
&\quad +\sum_{k,l,r,s=1}^{m+2n}\dfrac{1}{2}(-1)^{|k||i|+|l||i|+|r||i|+|s||i|+|k||j|+|l||j|+|r||j|+|s||j|+|l||r|+|l||s|}\\
&\quad \times (1+\delta_{kl}+\delta_{rs}+\delta_{kl}\delta_{rs})\beta_{li}\beta_{sj}\ell_{kr}\pt{\ell_{rs}}\pt{\ell_{kl}}\\
&= \dfrac{1}{2}\sum_{k,l,r,s=1}^{m+2n}(-1)^{|k||s|}(1+\delta_{kl}+\delta_{rs}+\delta_{kl}\delta_{rs})\ell_{kr}(\beta_{is}\beta_{jl} + (-1)^{|i||j|}\beta_{il}\beta_{js})\pt{\ell_{sr}}\pt{\ell_{lk}}\\
&= \sum_{k,l,r,s=1}^{m+2n}(-1)^{|k||s|}(1+\delta_{kl}+\delta_{rs}+\delta_{kl}\delta_{rs})\beta_{is}\beta_{jl}\ell_{kr}\pt{\ell_{sr}}\pt{\ell_{lk}}.
\end{align*}
We also have
\begin{align*}
\sum_{1\leq k \leq l\leq m+2n}\lambda_{\ell_{kl}}(\ell_{ij})\pt{\ell_{kl}} &= -\sum_{k,l=1}^{m+2n}(1+\delta_{kl})(\beta_{jk}\lambda(L_{\ell_{il}}) + (-1)^{|i||j|} \beta_{ik}\lambda(L_{\ell_{jl}}))\pt{\ell_{kl}}\\
&= -\lambda\sum_{k,l=1}^{m+2n}(1+\delta_{kl})(\beta_{jk}\beta_{il} + (-1)^{|i||j|} \beta_{ik}\beta_{jl})\pt{\ell_{kl}}\\
&= -2\lambda\sum_{k,l=1}^{m+2n}(1+\delta_{kl})\beta_{jk}\beta_{il}\pt{\ell_{kl}},
\end{align*}
which implies
\endgroup
\begin{align*}
\bessel(\ell_{ij}) &= -2\lambda\sum_{k,l=1}^{m+2n}(1+\delta_{kl})\beta_{jk}\beta_{il}\pt{\ell_{kl}}\\
&\quad +\sum_{k,l,r,s=1}^{m+2n}(-1)^{|k||i|}(1+\delta_{kl}+\delta_{rs}+\delta_{kl}\delta_{rs})\beta_{is}\beta_{jl}\ell_{kr}\pt{\ell_{sr}}\pt{\ell_{lk}},
\end{align*}
as desired.
\end{proof}

\begin{Prop}[Lemma \ref{Lemma_Q_Bessel}]\label{Prop_App_Q_Bessel}
Suppose $Q\in \mc P_2(\ds K^{{\wh m}|2{\wh n}})$ is given by
\begin{align*}
&Q = \sum_{i,j,k,l=1}^{m+2n}\alpha_{ijkl}\ell_{ij}\ell_{kl}, \quad \text{ with }\\ &\alpha_{ijkl} = (-1)^{|i||j|}\alpha_{jikl} = (-1)^{|k||l|}\alpha_{ijlk}=(-1)^{(|i|+|j|)(|k|+|l|)}\alpha_{klij}\in \ds C.
\end{align*}
Then the Bessel operators act trivially on $Q$ if and only if
\begin{align*}
2(-1)^{|i||j|}\lambda \alpha_{ijkl} &= (-1)^{|i||k|}\alpha_{jkil}+(-1)^{|k||l|+|i||l|}\alpha_{jlik},
\end{align*}
for all $i,j,k,l\in\{1, \ldots, m+2n\}$. 
\end{Prop}

\begin{proof}
\begingroup
\allowdisplaybreaks
We have
\begin{align*}
&\sum_{c,d}(1+\delta_{cd})\beta_{bc}\beta_{ad}\pt{\ell_{cd}}Q \\
&\quad =\sum_{i,j,k,l,c,d}(1+\delta_{cd})\alpha_{ijkl}\beta_{bc}\beta_{ad}\pt{\ell_{cd}}(\ell_{ij}\ell_{kl})\\
&= \sum_{i,j,k,l,c,d}(1+\delta_{cd})\alpha_{ijkl}\beta_{bc}\beta_{ad}\pt{\ell_{cd}}(\ell_{ij})\ell_{kl}\\
&\quad +\sum_{i,j,k,l,c,d}(-1)^{(|c|+|d|)(|i|+|j|)}(1+\delta_{cd})\alpha_{ijkl}\beta_{bc}\beta_{ad}\ell_{ij}\pt{\ell_{cd}}(\ell_{kl})\\
&= \sum_{i,j,k,l,c,d}(1+\delta_{cd})(\delta_{ci}\delta_{dj}+(-1)^{|i||j|}\delta_{cj}\delta_{di}-\delta_{ci}\delta_{dj}\delta_{cj}\delta_{di})\alpha_{ijkl}\beta_{bc}\beta_{ad}\ell_{kl}\\
&\quad +\sum_{i,j,k,l,c,d}(-1)^{(|c|+|d|)(|i|+|j|)}(1+\delta_{cd})\\
&\quad \times(\delta_{ck}\delta_{dl}+(-1)^{|k||l|}\delta_{cl}\delta_{dk}-\delta_{ck}\delta_{dl}\delta_{cl}\delta_{dk}) \alpha_{ijkl}\beta_{bc}\beta_{ad}\ell_{ij}\\
&= \sum_{i,j,k,l,c,d}(\delta_{ci}\delta_{dj}+(-1)^{|i||j|}\delta_{cj}\delta_{di})\alpha_{ijkl}\beta_{bc}\beta_{ad}\ell_{kl}\\
&\quad +\sum_{i,j,k,l,c,d}(-1)^{(|c|+|d|)(|i|+|j|)}(\delta_{ck}\delta_{dl}+(-1)^{|i||j|}\delta_{ck}\delta_{dl})\alpha_{ijkl}\beta_{bc}\beta_{ad}\ell_{ij}\\
&= 4\sum_{i,j,k,l}(-1)^{|i||j|}\alpha_{ijkl}\beta_{ai}\beta_{bj}\ell_{kl}
\end{align*}
and
\begin{align*}
&\pt{\ell_{fe}}\pt{\ell_{dc}}(\ell_{ij}\ell_{kl})\\
&\quad = \pt{\ell_{fe}}\left(\pt{\ell_{dc}}(\ell_{ij})\ell_{kl} + (-1)^{(|c|+|d|)(|k|+|l|)} \ell_{ij}\pt{\ell_{dc}}(\ell_{kl})\right)\\
&= \pt{\ell_{dc}}(\ell_{ij})\pt{\ell_{fe}}(\ell_{kl}) +(-1)^{(|c|+|d|)(|i|+|j|)} \pt{\ell_{fe}}(\ell_{ij})\pt{\ell_{dc}}(\ell_{kl})\\
&= (\delta_{di}\delta_{cj}+(-1)^{|i||j|}\delta_{dj}\delta_{ci}-\delta_{di}\delta_{cj}\delta_{dj}\delta_{ci})(\delta_{fk}\delta_{el}+(-1)^{|k||l|}\delta_{fl}\delta_{ek}-\delta_{fk}\delta_{el}\delta_{fl}\delta_{ek})\\
&\quad +(-1)^{(|c|+|d|)(|i|+|j|)}\\
&\quad \times (\delta_{fi}\delta_{ej}+(-1)^{|i||j|}\delta_{fj}\delta_{ei}-\delta_{fi}\delta_{ej}\delta_{fj}\delta_{ei})(\delta_{dk}\delta_{cl}+(-1)^{|k||l|}\delta_{dl}\delta_{ck}-\delta_{dk}\delta_{cl}\delta_{dl}\delta_{ck})
\end{align*}
and
\begin{align*}
&\sum_{c,d,e,f}(-1)^{|a||c|}(1+\delta_{cd}+\delta_{ef}+\delta_{cd}\delta_{ef})\beta_{af}\beta_{bd}\ell_{ce}\pt{\ell_{fe}}\pt{\ell_{dc}}Q\\
&= \sum_{c,d,e,f,i,j,k,l}(-1)^{|a||c|}(1+\delta_{cd}+\delta_{ef}+\delta_{cd}\delta_{ef})\alpha_{ijkl}\beta_{af}\beta_{bd}\ell_{ce}\pt{\ell_{fe}}\pt{\ell_{dc}}(\ell_{ij}\ell_{kl})\\
&= \sum_{c,d,e,f,i,j,k,l}(-1)^{|a||c|}(\delta_{di}\delta_{cj}+(-1)^{|i||j|}\delta_{dj}\delta_{ci})(\delta_{fk}\delta_{el}+(-1)^{|k||l|}\delta_{fl}\delta_{ek})\\
&\quad \quad \times\alpha_{ijkl}\beta_{af}\beta_{bd}\ell_{ce}\\
&\quad + \sum_{c,d,e,f,i,j,k,l}(-1)^{|a||c|+(|c|+|d|)(|i|+|j|)}(\delta_{fi}\delta_{ej}+(-1)^{|i||j|}\delta_{fj}\delta_{ei})\\
&\quad\quad \times(\delta_{dk}\delta_{cl}+(-1)^{|k||l|}\delta_{dl}\delta_{ck})\alpha_{ijkl}\beta_{af}\beta_{bd}\ell_{ce}\\
&= \sum_{i,j,k,l}(-1)^{|a||j|}\alpha_{ijkl}\beta_{ak}\beta_{bi}\ell_{jl} + \sum_{i,j,k,l}(-1)^{|a||j|+|k||l|}\alpha_{ijkl}\beta_{al}\beta_{bi}\ell_{jk}\\
&\quad +\sum_{i,j,k,l}(-1)^{|a||i|+|i||j|}\alpha_{ijkl}\beta_{ak}\beta_{bj}\ell_{il}+\sum_{i,j,k,l}(-1)^{|a||i|+|k||l|+|i||j|}\alpha_{ijkl}\beta_{al}\beta_{bj}\ell_{ik}\\
&\quad +\sum_{i,j,k,l}(-1)^{|a||l|}\alpha_{klij}\beta_{ai}\beta_{bk}\ell_{lj}+\sum_{i,j,k,l}(-1)^{|a||k|+|k||l|}\alpha_{klij}\beta_{ai}\beta_{bl}\ell_{kj}\\
&\quad +\sum_{i,j,k,l}(-1)^{|a||l|+|i||j|}\alpha_{klij}\beta_{aj}\beta_{bk}\ell_{li}+\sum_{i,j,k,l}(-1)^{|a||k|+|k||l|+|i||j|}\alpha_{klij}\beta_{aj}\beta_{bl}\ell_{ki}\\
&= 8\sum_{i,j,k,l} (-1)^{|a||k|}\alpha_{jkil}\beta_{ai}\beta_{bj}\ell_{kl}
\end{align*}
which implies
\begin{align*}
\bessel(\ell_{ab})Q &= 8\sum_{i,j,k,l}\left( (-1)^{|a||k|}\alpha_{jkil}-(-1)^{|i||j|}\lambda\alpha_{ijkl}\right)\beta_{ai}\beta_{bj}\ell_{kl}\\
&= 8\sum_{i,j,|k|=0}\left(\alpha_{jkik}-(-1)^{|i||j|}\lambda\alpha_{ijkk}\right)\beta_{ai}\beta_{bj}\ell_{kk}\\
&\quad + 8\sum_{i,j,k<l}\left( (-1)^{|a||k|}\alpha_{jkil}-(-1)^{|i||j|}\lambda\alpha_{ijkl}\right)\beta_{ai}\beta_{bj}\ell_{kl}\\
&\quad + 8\sum_{i,j,k>l}\left( (-1)^{|a||k|}\alpha_{jkil}-(-1)^{|i||j|}\lambda\alpha_{ijkl}\right)\beta_{ai}\beta_{bj}\ell_{kl}\\
&= 8\sum_{i,j,|k|=0}\left(\alpha_{jkik}-(-1)^{|i||j|}\lambda\alpha_{ijkk}\right)\beta_{ai}\beta_{bj}\ell_{kk}\\
&\quad + 8\sum_{i,j,k<l}\left( (-1)^{|a||k|}\alpha_{jkil}-(-1)^{|i||j|}\lambda\alpha_{ijkl}\right)\beta_{ai}\beta_{bj}\ell_{kl}\\
&\quad + 8\sum_{i,j,k<l}(-1)^{|k||l|}\left( (-1)^{|a||l|}\alpha_{jlik}-(-1)^{|i||j|}\lambda\alpha_{ijlk}\right)\beta_{ai}\beta_{bj}\ell_{kl}\\
&= 8\sum_{i,j,k<l}\left( (-1)^{|a||k|}\alpha_{jkil}+(-1)^{|k||l|+|a||l|}\alpha_{jlik}-2(-1)^{|i||j|}\lambda\alpha_{ijkl}\right)\beta_{ai}\beta_{bj}\ell_{kl}\\
&\quad + 8\sum_{i,j,|k|=0}\left(\alpha_{jkik}-(-1)^{|i||j|}\lambda\alpha_{ijkk}\right)\beta_{ai}\beta_{bj}\ell_{kk}.
\end{align*}
If we now impose $\bessel(\ell_{ij})Q=0$, for all $i,j\in\{1, \ldots, m+2n\}$, we obtain the desired conditions on the coefficients of $Q$.
\endgroup
\end{proof}

\makeatletter
\let\@mkboth\@gobbletwo
\makeatother
\printindex[sym]

\bibliography{citations} 
\bibliographystyle{ieeetr}

\end{document}